\documentclass[11pt, leqno]{amsart}
\usepackage{amsfonts,amssymb}
 
\usepackage{amsmath}
\usepackage{amsthm}
\usepackage{amsrefs}
\usepackage{qsymbols}
\usepackage{latexsym}
\usepackage{chngcntr}
\usepackage{paralist}
\usepackage{mathtools}
\usepackage{esint}
\usepackage[hidelinks]{hyperref}
\usepackage{comment}
\usepackage{tikz-cd}
\usepackage{import}
\usepackage{todonotes}
\usepackage{csquotes}
\usepackage[dvipsnames]{xcolor}
\usepackage{enumitem}
\usepackage{booktabs}
\usepackage{makecell}

\newtheorem{theorem}{Theorem}[section]

\newtheorem{lemma}[theorem]{Lemma}
\newtheorem{proposition}[theorem]{Proposition}
\newtheorem{corollary}[theorem]{Corollary}
\theoremstyle{definition}
\newtheorem{definition}[theorem]{Definition}
\newtheorem{remark}[theorem]{Remark}
\newtheorem{example}[theorem]{Example}
\newtheorem{assumption}[theorem]{Assumption}
\newtheorem{observation}[theorem]{Observation}

\newcommand{\IR}{\mathbb{R}}
\newcommand{\IC}{\mathbb{C}}
\newcommand{\IN}{\mathbb{N}}
\newcommand{\IZ}{\mathbb{Z}}
\newcommand{\IP}{\mathbb{P}}

\renewcommand{\a}{\mathbf{a}}

\newcommand{\cM}{\mathcal{M}}

\newcommand{\cP}{\mathcal{P}}

\newcommand{\cA}{\mathcal{A}}
\newcommand{\cB}{\mathcal{B}}
\newcommand{\cF}{\mathcal{F}}

\newcommand{\cL}{\mathcal{L}}

\renewcommand{\L}{\mathrm{L}}
\renewcommand{\H}{\mathrm{H}}
\renewcommand{\S}{\mathrm{S}}

\newcommand{\W}{\mathrm{W}}
\newcommand{\C}{\mathrm{C}}

\newcommand{\Lloc}{\L_{\mathrm{loc}}}
\newcommand{\Llocs}{\L_{\mathrm{loc},\sigma}}

\newcommand{\esssup}{\mathrm{ess\, sup}}

\renewcommand{\d}{\mathrm{d}}

\newcommand{\loc}{\mathrm{loc}}
\renewcommand\Re{\operatorname{Re}}

\renewcommand{\div}{\operatorname{div}}
\DeclareMathOperator{\supp}{supp}
\DeclareMathOperator{\Id}{Id}

\DeclareMathOperator{\dom}{\mathcal{D}}

\numberwithin{equation}{section}

\title[$\L^p$-extrapolation via non-local decay estimates]{$\L^p$-extrapolation via non-local decay estimates: theory and applications to the generalized Stokes operator with rough coefficients}
 
\author{Luca Haardt}

\address{Karlsruhe Institute of Technology, Department of Mathematics, 76131 Karlsruhe, Germany}
\email{luca.haardt@kit.edu}



\date{\today}
 
\begin{document}
\begin{abstract}
    In this article, we establish $\L^p$-mapping properties for the generalized Stokes operator with bounded measurable coefficients on the full range of exponents that is known to be sharp for elliptic systems with rough coefficients. These include $\L^p$-estimates for the corresponding semigroup, its gradient and the $\H^\infty$-calculus, as well as $\L^q$-$\L^p$ smoothing estimates. As a key ingredient, we develop a non-local framework for $\L^p$-extrapolation that captures the relationship between hypercontractivity, non-local decay estimates and uniform $\L^p$-bounds.
\end{abstract}

\allowdisplaybreaks
\maketitle

\section{Introduction}
 
For second-order elliptic operators of the form $L=-\div(\mu \nabla \cdot )$ with rough coefficients $\mu$, the $\L^p$-theory is well-understood, see for example \cites{Auscher-Lp,Auscher-Egert, Bechtel} and references therein. In particular, there are two intervals of exponents $p$ that rule the $\L^p$-behavior of $L$. More precisely, it is known that there exists $\varepsilon > 0$ such that for all $p \in (1, \infty)$ satisfying
\begin{align}
\label{eq: elliptic range I}
    \left|\frac{1}{2}-\frac{1}{p}\right| < \frac{1}{d} + \varepsilon,
\end{align}
the resolvent, the semigroup and the functional calculus of $L$ can be successfully extrapolated from $\L^2(\IR^d)$ to $\L^p(\IR^d)$.
However, if the gradient is involved, it is known that if $p\in(1,\infty)$ satisfies
\begin{align}
\label{eq: elliptic range II}
    \frac{2d}{d+2}-\varepsilon <p<2+\varepsilon,
\end{align}
then the gradient of the resolvent family, the gradient of the semigroup, and the Riesz transform $\nabla L^{-\frac{1}{2}}$ are uniformly bounded on $\L^p$ as well. Here, $\varepsilon>0$ denotes a (possibly small) constant depending on the dimension $d$ and the coefficients $\mu$, which stems from a stability theorem of \u{S}ne\u{\i}berg (see Proposition~\ref{prop: Sneiberg}). We point out that both intervals above are known to be sharp; see \cite{HMM} and the references therein.
 
The core machinery driving these elliptic results relies heavily on so-called off-diagonal estimates, which are generalizations of pointwise kernel bounds. For the semigroup of $L$, they state that for all $p\in[2,\infty)$ satisfying \eqref{eq: elliptic range I}, there exist constants $C,c>0$ such that
\begin{align}
\label{eq: local off-diag}
    \|e^{-zL}f\|_{\L^p(B)} \leq C |z|^{-\frac{d}{2}(\frac{1}{2}-\frac{1}{p})} e^{-\frac{c4^kr^2}{|z|}}\|f\|_{\L^2(C_k(B))}
\end{align}
for appropriate $z\in\IC\setminus\{0\}$ in some open sector, all balls $B=B(x,r)\subset\IR^d$, all $k\geq 2$ and all $f\in\L^2(\IR^d)$ with $\supp(f)\subset C_k(B)=2^kB\setminus \overline{2^{k-1}B}$; see \cite{Auscher-Lp}. These estimates are the key input for the $\L^p$-theory in the ranges \eqref{eq: elliptic range I} and \eqref{eq: elliptic range II}. In particular, these form a stronger notion than uniform $\L^p$-bounds and are classically derived in two steps. First, local Caccioppoli-type inequalities, obtained by testing the resolvent equation with its solution multiplied by suitable cut-off functions, yielding the case $p=2$. Second, one uses $\L^2$-$\L^p$ bounds, also called hypercontractivity~\cite{Auscher-Lp}, which state that for all $p\in[2,\infty)$ satisfying \eqref{eq: elliptic range I} there exists $C>0$ such that
\begin{align*}
    \|e^{-zL}f\|_{\L^p} \leq C |z|^{-\frac{d}{2}(\frac{1}{2}-\frac{1}{p})}\|f\|_{\L^2}
\end{align*}
for appropriate $z\in\IC\setminus\{0\}$ and $f\in\L^2(\IR^d)$. Given $p\in(2,\infty)$ satisfying \eqref{eq: elliptic range I}, one chooses $p_1>p$ in the same range and interpolates the $\L^2$-off-diagonal estimates with the $\L^2$-$\L^{p_1}$ bounds. If $\theta\in(0,1)$ is given by $\frac{1}{p}=\frac{1-\theta}{2}+\frac{\theta}{p_1}$, this yields the $\L^2$-$\L^p$ off-diagonal estimates above with $c$ replaced by $(1-\theta)c$. Hence, only the constant in the exponential factor gets worse, and no decay is lost.
 
For operators with non-local structure a derivation of local Caccioppoli inequalities are not possible, since testing with a cut-off function produces non-local terms that cannot be estimated by local data. The main example of this paper is the generalized Stokes operator, where these terms are caused by the pressure; see below. For such operators, the solution on a ball $B$ is no longer controlled by the data on a single annulus, but only by the data on all annuli around $B$ at the same time. Moreover, the non-locality can limit the decay rate, as it happens for the Stokes operator. The first aim of this paper is to provide a notion of decay estimates that captures this situation, together with an abstract framework that turns such estimates into uniform $\L^p$-bounds. This framework is developed in Section~\ref{sec: abstract machinery}.
 
More precisely, a family $(T(z))_{z\in\Omega}$ of bounded operators between closed subspaces of $\L^2$, where $\Omega\subset\IC\setminus\{0\}$ is open, satisfies $\L^2$-$\L^p$ decay estimates of order $\nu\geq 0$ for $p\in[2,\infty]$ if there exist $M,N\geq 0$ such that
\begin{align*}
    \|T(z)f\|_{\L^p(B)} 
    \lesssim r^{-d(\frac{1}{2}-\frac{1}{p})}\max\Big\{\Big(\frac{r^2}{|z|}\Big)^{M}, \Big(\frac{r^2}{|z|}\Big)^{-N}\Big\}\bigg(\sum\limits_{n=0}^\infty 2^{-\nu n} \|f\|^2_{\L^2(C_{n}(B))}\bigg)^\frac{1}{2}
\end{align*}
for all $z\in\Omega$ and all balls $B=B(x,r)\subset \IR^d$; see Definition~\ref{def: decay est L2-Lp}. Compared with the elliptic off-diagonal estimates \eqref{eq: local off-diag}, the estimates above are non-local, since the data on all annuli enter at the same time and their decay rate is limited to the polynomial factor $2^{-\nu n}$. Nevertheless, they imply uniform $\L^p$-bounds as in the elliptic case if $\nu>d$ (Proposition~\ref{prop: L2-Lp off imply Lp bdd}). Moreover, they interpolate with hypercontractivity too (Proposition~\ref{prop: int principles}). However, interpolation reduces the decay rate, so that uniform bounds can only be concluded as long as the reduced order is still larger than $d$ (Corollary~\ref{cor: decay, hyper, unif}). We stress that the results of Section~\ref{sec: abstract machinery} do not use any structure of the underlying operator and believe that this general framework is of independent interest for non-local operators such as integrodifferential operators and parabolic operators, see \cites{Tolksdorf, Baadi_Egert_Kosmala}.
 
As the second aim of this paper, we will apply this framework to the generalized Stokes operator $A$ associated with the Stokes system
\begin{align}
\label{eq: system}
    \left\{ \begin{aligned}
    - \div(\mu\nabla u) +\nabla \phi &= f + \div(F),&& \text{in }\IR^d\\
    \div u &=0 ,&& \text{in }\IR^d,
    \end{aligned} \right.
\end{align}
with rough coefficients $\mu$. More precisely, $\mu$ is assumed to be merely essentially bounded and complex-valued, and to satisfy a G\r{a}rding-type inequality that enforces ellipticity; see Assumption~\ref{Ass: Coefficients}. Under these assumptions, the generalized Stokes operator $A$, formally given by 
\begin{align*}
    Au = -\div(\mu \nabla u) + \nabla \phi, \quad \div u = 0,
\end{align*}
can be realized on the space of divergence-free vector fields $\L^2_\sigma(\IR^d)$ via form methods as a densely defined, sectorial operator of angle $\omega_0 \in [0, \frac{\pi}{2})$, see for instance \cite{Kato}. Consequently, its spectrum $\sigma(A)$ is contained in the closed sector $\overline{\S_{\omega_0}}$ in the complex plane that is symmetric about the positive real axis with opening angle $2\omega_0$, and for every $\theta \in (0, \pi-\omega_0)$ there exists $C>0$ such that the resolvent satisfies the uniform estimate
\begin{align}
\label{eq: L2 eq resolvent}
    \|\lambda(\lambda+A)^{-1}f\|_{\L^2} \leq C \|f\|_{\L^2}
\end{align}
for all $\lambda \in \S_{\theta}$ and $f \in \L^2_\sigma(\IR^d)$.
 
The equations in \eqref{eq: system} appear, for example, in the description of flows through porous materials~\cites{periodic, Gu_Shen}, or as linearizations of non-Newtonian fluids, which generally exhibit non-constant viscosity~\cites{Diening, Pruss_Simonett}. 
In the study of such non-linear problems, $\L^p$-mapping properties of the linearization are often crucial. The question of whether the $\L^2$-bounds in \eqref{eq: L2 eq resolvent} can be extended to the $\L^p$-scale, $p\neq 2$, serves as a starting point for the investigation of further functional analytic objects of $A$, such as semigroups, maximal $\L^q$-regularity and functional calculus.
 
When one tries to transfer the elliptic $\L^p$-theory to $A$, the pressure $\phi$ is the major obstacle. Indeed, the derivation of off-diagonal estimates, and in particular Caccioppoli inequalities, relies on test the resolvent equation by its solution $u$ multiplied by a localized cut-off function. Because the cut-off inherently destroys the divergence-freeness of the testing function, the pressure $\phi$ explicitly appears in the equations. Since the pressure is non-local, it cannot be directly estimated by local data, creating a technical bottleneck. This obstacle was directly addressed by Tolksdorf in~\cite{Tolksdorf-Caccioppoli}, where he circumvented this issue by establishing a framework of non-local Caccioppoli inequalities. He used these inequalities in combination with an $\L^p$-extrapolation argument to extend the maximal regularity of $A$ from $\L^2_\sigma(\IR^d)$ to $\L^p_\sigma(\IR^d)$, and upgraded this to a bounded $\H^\infty$-calculus on $\L^p_\sigma(\IR^d)$ by means of a transference principle due to Kunstmann and Weis~\cite{Kunstmann-Weis}. In this way, he obtained $\L^p$-bounds for the resolvent, semigroup and the bounded $\H^\infty$-functional calculus provided that $p\in(1,\infty)$ satisfies the strict condition
\begin{align}
\label{eq: Tolksdorf range I}
    \left|\frac{1}{2}-\frac{1}{p}\right| < \frac{1}{d}.
\end{align}
Moreover, if $p\in(1,\infty)$ satisfies
\begin{align}
\label{eq: Tolksdorf range II}
    \frac{2d}{d+2}<p\leq 2
\end{align}
then he managed to show uniform $\L^p$-boundedness of the gradient of the resolvent and semigroup family.
However, when compared with the pure elliptic setting \eqref{eq: elliptic range I} and \eqref{eq: elliptic range II}, the intervals \eqref{eq: Tolksdorf range I} and \eqref{eq: Tolksdorf range II} stop precisely at the standard Sobolev threshold and miss the endpoint and the additional $\varepsilon$-window. This gap leaves open the question of whether the non-local nature of $A$ restricts its integrability range. In this work, we answer this question definitively by closing the gap to the optimal elliptic range; see Table~\ref{table: ranges}.
 
\begin{table}[h]
\centering
\begin{tabular}{lccc}
\toprule
& \textbf{Elliptic systems}
& \textbf{Stokes~\cite{Tolksdorf-Caccioppoli}}
& \textbf{Stokes, this paper} \\
\midrule
\makecell[l]{Resolvent, semigroup,\\ $\H^\infty$-calculus}
& $\Big|\dfrac{1}{2}-\dfrac{1}{p}\Big|<\dfrac{1}{d}+\varepsilon$
& $\Big|\dfrac{1}{2}-\dfrac{1}{p}\Big|<\dfrac{1}{d}$
& $\Big|\dfrac{1}{2}-\dfrac{1}{p}\Big|<\dfrac{1}{d}+\varepsilon$ \\
\addlinespace
\makecell[l]{Gradient of resolvent\\ and semigroup}
& $\dfrac{2d}{d+2}-\varepsilon<p<2+\varepsilon$
& $\dfrac{2d}{d+2}<p\le 2$
& $\dfrac{2d}{d+2}-\varepsilon<p<2+\varepsilon$ \\
\bottomrule
\end{tabular}
\vspace{0.2cm}
\caption{Ranges of $p\in(1,\infty)$ for which uniform $\L^p$-bounds are known.}
\label{table: ranges}
\end{table}
 
To this end, we take a different route, namely the one of the elliptic theory: we derive decay estimates and hypercontractivity for the Stokes resolvent and semigroup and combine them with the non-local abstract framework of Section~\ref{sec: abstract machinery}. This reproves the $\L^p$-bounds of~\cite{Tolksdorf-Caccioppoli} in the ranges \eqref{eq: Tolksdorf range I} and \eqref{eq: Tolksdorf range II} with different techniques and yields the extended sharp elliptic ranges \eqref{eq: elliptic range I} and \eqref{eq: elliptic range II}.
 
Decay estimates for the Stokes resolvent have been studied before. For the case $p=2$, a first attempt to establish off-diagonal-type estimates for the Stokes operator was made by Tolksdorf in~\cite{Tolksdorf-off-diagonal}, where a decay of order $0 < \nu < 2$ was achieved. Since this order is below $d$, it is not sufficient for the approach described above. A first step in this direction was taken by Tolksdorf and the author in \cite{Haardt_Tolksdorf}, where decay estimates of order $0<\nu<d+2$ were established in the special case $|\lambda| r^2 \simeq 1$, enabling the proof of the Kato square root property for the Stokes operator $A$. Our first main result, proved in Section~\ref{sec: off-diag est}, provides non-local $\L^2$-$\L^p$ decay estimates for the Stokes resolvent of every order $0<\nu<d+2$ without any constraint on $\lambda$ and $r$. We refer to Section~\ref{sec: notation} for precise notation and background.
 
\begin{theorem}[$\L^2$-$\L^p$ decay estimates for Stokes resolvents]
\label{thm: L2-Lp off diag stokes resolv}
    Let $\mu$ satisfy Assumption~\ref{Ass: Coefficients} and let $\omega_0$ be given by \eqref{eq: def omega}. Then for all $\theta \in [0,\pi-\omega_0)$, all $\nu \in (0, d+2)$ and all $p\in [2,\infty)$ satisfying $\frac{1}{2}-\frac{1}{p}\leq \frac{1}{d}$ there exists a constant $C >0$ such that for all balls $B= B(x_0,r)$, $\lambda \in \S_\theta$, $f\in \L_\sigma^2(\IR^d)$ and $F \in \L^2 (\IR^d ; \IC^{d \times d})$ the unique solution $u\in \W^{1,2}_\sigma(\IR^d)$ to
    \begin{align}
    \label{eq: Stokes system}
        \left\{ \begin{aligned}
        \lambda u- \div(\mu\nabla u) +\nabla \phi &= f + \div(F),&& \text{in }\IR^d\\
        \div u &=0 ,&& \text{in }\IR^d,
        \end{aligned} \right.
    \end{align}
    satisfies
    \begin{align*}
        \|u\|_{\L^{p}(B)}^2 &\leq  \frac{C}{|\lambda|r^{d(1-\frac{2}{p})}}\max\Big\{|\lambda|r^2, (|\lambda|r^2)^{-(d+1)}\Big\}\sum\limits_{n=0}^\infty 2^{-\nu n}\bigg(\frac{1}{|\lambda|}\int\limits_{C_{n}(B)}|f|^2\, \d x +\int\limits_{C_{n}(B)}|F|^2\,\d x\bigg).
    \end{align*}
    If, in addition, $\supp(f),\supp (F) \subset \IR^d\setminus 2B$, then the unique solution $u\in \W^{1,2}_\sigma(\IR^d)$ to \eqref{eq: Stokes system} satisfies
    \begin{align*}
        \|u\|_{\L^{p}(B)}^2 &\leq  \frac{C }{|\lambda|r^{d(1-\frac{2}{p})}}\max\Big\{1, (|\lambda|r^2)^{-(d+2)}\Big\}\sum\limits_{n=0}^\infty 2^{-\nu n}\bigg(\frac{1}{|\lambda|}\int\limits_{C_{n}(B)}|f|^2\, \d x +\int\limits_{C_{n}(B)}|F|^2\,\d x\bigg).
    \end{align*}
    In both cases the constant $C$ depends only on $\theta$, $\nu$, $d$ and $\mu_\bullet,\mu^\bullet$.
\end{theorem}
 
The decay rate in Theorem~\ref{thm: L2-Lp off diag stokes resolv} is limited by $d+2$, but $\nu$ can be chosen larger than $d$, as required by Proposition~\ref{prop: L2-Lp off imply Lp bdd}. For the proof, we refine the iteration procedure of~\cite{Haardt_Tolksdorf} based on non-local Caccioppoli inequalities by an additional hole-filling argument. This gives sharper constants and thus quantitative decay with respect to $|\lambda|$, $r$ and the distance to the ball $B$. Since the procedure yields $\L^2$-decay estimates for $u$ and $\nabla u$ simultaneously, Sobolev's embedding then gives Theorem~\ref{thm: L2-Lp off diag stokes resolv} for $p\neq 2$.
 
In Sections~\ref{sec: Lp-theory sg} and \ref{sec: gradient est}, we transfer Theorem~\ref{thm: L2-Lp off diag stokes resolv} to the Stokes semigroup and its gradient and apply the framework of Section~\ref{sec: abstract machinery}. Together with hypercontractivity slightly beyond the Sobolev threshold, which we obtain from a stability result of \u{S}ne\u{\i}berg, this yields uniform $\L^p$-bounds for all $p \in (1, \infty)$ satisfying \eqref{eq: elliptic range I} and \eqref{eq: elliptic range II}, respectively. Our main theorem for the semigroup reads as follows.
 
\begin{theorem}[$\L^p$-extrapolation of the semigroup]
\label{thm: Lp-extrapolation sg II}
Let $\mu$ satisfy Assumption~\ref{Ass: Coefficients}, $\omega_0$ be given by \eqref{eq: def omega} and $\beta \in [0,\frac{\pi}{2}-\omega_0)$. There exists $\varepsilon>0$ depending on $\beta,d$ and $\mu_\bullet,\mu^\bullet$ such that for all $p\in (1,\infty)$ satisfying
\begin{align*}
    \Big|\frac{1}{2}-\frac{1}{p}\Big|<  \frac{1}{d}+\varepsilon
\end{align*}
there exists a constant $C>0$ such that for all $z \in \S_\beta$ and $f\in \L_\sigma^2(\IR^d)\cap \L_\sigma^{p}(\IR^d)$  we have
\begin{align*}
    \|e^{-zA}f \|_{\L^{p}}\leq C \|f\|_{\L^{p}}.
\end{align*}
The constant $C$ depends only on $\beta,d,p$ and $\mu_\bullet,\mu^\bullet$.
\end{theorem}

In particular, for $d=2$ the full range $p\in(1,\infty)$ is covered.
The corresponding gradient estimates in $\L^p$ are obtained by a similar approach combined with a duality argument, which we summarize in the following theorem.
 
\begin{theorem}[$\W^{1,p}$-estimates of the semigroup]
\label{thm: Lp bound for grad fam}
    Let $\mu$ satisfy Assumption~\ref{Ass: Coefficients}, $\omega_0$ be given by \eqref{eq: def omega} and $\beta \in [0,\frac{\pi}{2}-\omega_0)$. There exists $\varepsilon>0$ depending on $\beta,d$ and $\mu_\bullet,\mu^\bullet$ such that for all $p\in (1,\infty)$ satisfying
    \begin{align*}
        \frac{2d}{d+2}-\varepsilon < p < 2+\varepsilon
    \end{align*}
    there exists a constant $C>0$ such that for all $z \in \S_\beta$ and $f\in \L_\sigma^2(\IR^d)\cap \L_\sigma^{p}(\IR^d)$ we have
    \begin{align*}
      \||z|^\frac{1}{2}\nabla e^{-zA}f \|_{\L^{p}}\leq C \|f\|_{\L^{p}}.
    \end{align*}
    The constant $C$ depends only on $\beta,d,p$ and $\mu_\bullet,\mu^\bullet$.
\end{theorem}

We briefly remark that both $\L^p$-estimates for the semigroup and its gradient imply analogous $\L^p$-estimates for the resolvent and its gradient by the Laplace transform representation of the resolvent; see \cite[Prop.~3.4.4]{Haase}. In particular, the $\L^2$-bounds in \eqref{eq: L2 eq resolvent} can be extended to the $\L^p$-scale for $p \neq 2$. 
Combining the uniform estimates from Theorems~\ref{thm: Lp-extrapolation sg II} and~\ref{thm: Lp bound for grad fam} with their hypercontractivity properties (see Sections~\ref{sec: Lp-theory sg} and~\ref{sec: gradient est}), we obtain the following smoothing estimates, that are standard inputs for fixed-point arguments for the Navier--Stokes equations, see for example~\cite{Giga, Kato-Lp}.
 
\begin{theorem}[$\L^q$-$\L^p$ smoothing estimates]
\label{thm: smoothing}
Let $\mu$ satisfy Assumption~\ref{Ass: Coefficients}, $\omega_0$ be given by \eqref{eq: def omega} and $\beta \in [0,\frac{\pi}{2}-\omega_0)$. There exists $\varepsilon>0$ depending on $\beta,d$ and $\mu_\bullet,\mu^\bullet$ such that for all $1<q\leq p<\infty$ satisfying
\begin{align*}
    \Big|\frac{1}{2}-\frac{1}{q}\Big|<  \frac{1}{d}+\varepsilon \quad \text{and} \quad  \Big|\frac{1}{2}-\frac{1}{p}\Big|<  \frac{1}{d}+\varepsilon,
\end{align*}
there exists a constant $C>0$ such that for all $z \in \S_\beta$ and $f\in \L_\sigma^2(\IR^d)\cap \L_\sigma^{q}(\IR^d)$  we have
\begin{align}
\label{eq: smoothing sg}
    \|e^{-zA}f \|_{\L^{p}}\leq C |z|^{-\frac{d}{2}(\frac{1}{q}-\frac{1}{p})}\|f\|_{\L^{q}}.
\end{align}
If moreover $q$ and $p$ satisfy
\begin{align*}
    \frac{2d}{d+2}-\varepsilon < q\leq p < 2+\varepsilon,
\end{align*}
then there exists a constant $C>0$ such that for all $z \in \S_\beta$ and $f\in \L_\sigma^2(\IR^d)\cap \L_\sigma^{q}(\IR^d)$ we have
    \begin{align}
    \label{eq: smoothing grad}
      \||z|^\frac{1}{2}\nabla e^{-zA}f \|_{\L^{p}}\leq C |z|^{-\frac{d}{2}(\frac{1}{q}-\frac{1}{p})}\|f\|_{\L^{q}}.
    \end{align}
Both constants depend only on $\beta,d,p,q$ and $\mu_\bullet,\mu^\bullet$.
\end{theorem}
 
Finally, in Section~\ref{sec: Hinfty calc}, we employ an abstract $\L^p$-extrapolation result well-suited to the $\L^2$-$\L^p$ decay estimates of the Stokes semigroup, allowing us to extrapolate the bounded $\H^\infty$-calculus of $A$ to $\L^p_\sigma(\IR^d)$ for all $p \in (1, \infty)$ satisfying \eqref{eq: elliptic range I}.
 
\begin{theorem}[$\L^p$-extrapolation of $\H^\infty$-calculus]
\label{thm: bdd calc on Lp}
Let $\mu$ satisfy Assumption~\ref{Ass: Coefficients},  $\omega_0$ be given by \eqref{eq: def omega} and $\varrho\in (\omega_0,\pi)$. There exists $\varepsilon>0$ depending on $d,\varrho$ and $\mu_\bullet,\mu^\bullet$ such that for all $p\in(1,\infty)$ satisfying 
\begin{align*}
    \Big|\frac{1}{2}-\frac{1}{p}\Big|< \frac{1}{d}+\varepsilon
\end{align*}
there exists a constant $C>0$ such that for all $\varphi\in \H^\infty(\S_\varrho)$ and $f\in \L^2_\sigma(\IR^d)\cap \L^p_\sigma(\IR^d)$ we have
\begin{align*}
    \|\varphi(A)f\|_{\L^p} \leq C\|\varphi\|_{\infty,\varrho}\|f\|_{\L^p}.
\end{align*}
The constant $C$ depends only on $\varrho, d,p$ and $\mu_\bullet,\mu^\bullet$.
\end{theorem}

Summarizing, these results demonstrate that the non-locality caused by the pressure does not restrict the $\L^p$-theory. In particular, the Stokes operator with rough coefficients shares the same $\L^p$-mapping properties as its elliptic counterpart. In a forthcoming joint work with Tolksdorf, we will also treat the Riesz transform associated with the Stokes operator $A$ on $\L^p$-spaces for $p \neq 2$.

\subsection{Acknowledgments}
The author would like to thank Sebastian Bechtel and Patrick Tolksdorf for enriching discussions and valuable feedback. The author was supported by \textit{Studienstiftung des deutschen Volkes}.

\section{Notation and background}
\label{sec: notation}

We use the following notation throughout the paper.
We denote by $\IN,\IN_0,\IZ, \IR$ and $\IC$ the sets of all positive integers, all non-negative integers, all integers, all real and all complex numbers, respectively. Throughout, $d\geq 2$ denotes the dimension of the underlying Euclidean space. The Euclidean norms of $\IC$, $\IC^d$ and $\IC^{d \times d}$ will be denoted by $|\cdot|$, all other norms will be labeled accordingly. 
We denote the open ball centered at $x\in \IR^d$ with radius $r>0$ by $B(x,r)$. Given an open ball $B=B(x,r)$ and $\ell \in \IN_0$, we define $B_\ell=B(x,2^\ell r)$ and the $\ell$-th dyadic annulus around $B$ by
\begin{align*}
 C_{\ell} (B) \coloneqq B_\ell \setminus \overline{B_{\ell-1}} \quad \text{if} \quad \ell \geq 1 \quad \text{and} \quad C_0 (B) \coloneqq B.
\end{align*}
The open sector in the complex plane that is symmetric about the positive real axis with opening angle $2 \omega$, $\omega\in(0,\pi)$, is defined as
\begin{align*}
    \S_\omega \coloneqq \bigl\{z\in\IC\setminus \{0\}: |\operatorname{arg}(z)|<\omega\bigr\}.
\end{align*}
In the special case $\omega = 0$ we write $\S_{\omega} \coloneqq (0 , \infty)$.

In estimates it will be convenient to write $\alpha \lesssim \beta$ or $\beta\gtrsim \alpha$ if there exists $C > 0$, depending only on parameters not at stake, such that $\alpha \leq C \beta$. We will write $\alpha \simeq \beta$ if both $\alpha \lesssim \beta$ and $\alpha \gtrsim \beta$ hold. In some situations it will be more convenient to keep the notation $\alpha \leq C \beta$ or $C\alpha \geq \beta$. In this case $C$ is a generic constant and may change from line to line. \par

The characteristic function of a set $E \subseteq \IR^d$ is denoted by $\mathbf{1}_E$. If $E$ is measurable we denote its Lebesgue measure by $\lvert E \rvert$. In the case of $E$ being a bounded set with $0 < |E| < \infty$, we define the mean value of a locally integrable function $f \in \L^1_{\loc} (\IR^d)$ over $E$ by
\begin{align*}
    f_E\coloneqq \fint_E f \, \d x =\frac{1}{\lvert E \rvert} \int_E f \, \d x.
\end{align*}

\subsection{Function spaces}
For a target space $X\in\{\IC,\IC^d,\IC^{d \times d}\}$, we denote by $\C_{c}^{\infty}(\IR^d;X)$ the vector space of all smooth, compactly supported functions. For $1<p<\infty$, we denote by $\L^p (\IR^d;X)$ and $\W^{1,p} (\IR^d;X)$ the classical Lebesgue and Sobolev spaces equipped with their usual norms $\|\cdot \|_{\L^p}$ and $\|\cdot \|_{\W^{1,p}}$, respectively. To streamline notation, we often omit the target space and simply write $\L^p$ or $\W^{1,p}$ when the context is clear. We recall that $\C_{c}^{\infty} (\IR^d;X)$ is dense in both spaces with respect to their underlying norms.

To introduce standard spaces in the theory of mathematical fluid dynamics, we define the space of smooth, compactly supported and divergence-free vector fields by
\begin{align*}
    \C_{c , \sigma}^{\infty} (\IR^d) \coloneqq \{\varphi \in \C_c^{\infty} (\IR^d ; \IC^d) : \div(\varphi) = 0 \}.
\end{align*}
For $1<p<\infty$, we define the corresponding Lebesgue and Sobolev spaces
\begin{align*}
    \L^p_{\sigma} (\IR^d) &\coloneqq\overline{\C_{c , \sigma}^{\infty} (\IR^d)}^{\|\cdot\|_{\L^p}} \\
    \W^{1,p}_{\sigma} (\IR^d) &\coloneqq\overline{\C_{c , \sigma}^{\infty} (\IR^d)}^{\|\cdot\|_{\W^{1,p}}}.
\end{align*}
These spaces are also characterized by
\begin{align*}
    \L^p_{\sigma} (\IR^d) & =\{f\in\L^p(\IR^d;\IC^d) :  \div (f) = 0\} \\
    \W^{1,p}_{\sigma} (\IR^d) &=\{f\in\W^{1,p}(\IR^d;\IC^d) : \div(f) = 0\}.
\end{align*}
Moreover, we also define
    \begin{align*}
        \Llocs^p(\IR^d) \coloneqq \{f\in \Lloc^p(\IR^d;\IC^d):  \div (f) = 0\}.
    \end{align*}
We refer to the monograph of Galdi~\cite{Galdi} for a more comprehensive treatment of these function spaces. We recall the Leray projection $\IP$, which is given on the whole space $\IR^d$ by
\begin{align*}
    \IP = \cF^{-1}\Big(\Id_{\IC^{d\times d}} - \frac{\xi\xi^\top}{|\xi|^2}\Big)\cF,
\end{align*}
where $\cF$ denotes the Fourier transform. It is known that $\IP$ is a bounded linear projection from $\L^p(\IR^d;\IC^d)$ and $\W^{1,p}(\IR^d;\IC^{d})$ onto $\L^p_{\sigma} (\IR^d)$ and $\W^{1,p}_{\sigma} (\IR^d)$, respectively. The space of all bounded antilinear functionals $\W^{1,p}_{\sigma} (\IR^d) \to \IC$ is denoted by $\W^{-1,p'}_{\sigma} (\IR^d)$, where $p'\in(1,\infty)$ is the Hölder conjugate defined by $\frac{1}{p}+\frac{1}{p'}=1$. For $p\in[1,\infty)$, we define the lower Sobolev conjugate $p_*\coloneqq \frac{dp}{d+p}$ and the upper Sobolev conjugate $p^*\coloneqq \frac{dp}{d-p}$ if $1\leq p<d$, and $p^*=\infty$ if $p\geq d$. Given a vector field $u \in \W^{1,p} (\IR^d ; \IC^d)$, we regard its gradient as the matrix given as the transpose of the Jacobian of $u$, that is $\nabla u = (\partial_{\alpha} u_i)_{\alpha , i = 1}^d$. For a matrix-valued function $F \in \L^p (\IR^d ; \IC^{d \times d})$, we define its weak divergence as an element in $\W^{-1,p} (\IR^d ; \IC^d) \coloneqq (\W^{1,p'} (\IR^d ; \IC^d))^{\prime}$ via the duality pairing
\begin{align*}
 \langle \div (F) , w \rangle \coloneqq - \sum_{\alpha , i = 1}^d \int_{\IR^d} F_{\alpha i} \overline{\partial_{\alpha} w_i} \, \d x \qquad (w \in \W^{1,p'} (\IR^d ; \IC^d)).
\end{align*}
To interpret the space $\W^{1,p}_{\sigma} (\IR^d)$ as a subspace of $\W^{1,p} (\IR^d ; \IC^d)$, we introduce the canonical inclusion $\iota : \W^{1,p}_{\sigma} (\IR^d) \to \W^{1,p} (\IR^d ; \IC^d)$.
Its adjoint operator then defines the generalized Leray projection on the dual spaces, which we denote by $\cP \coloneqq \iota' : \W^{-1,p'} (\IR^d ; \IC^d) \to \W^{-1,p'}_{\sigma} (\IR^d)$. Notice that $\cP = \IP$ on $\L^{p}$ and $\W^{1,p}$.

\subsection{Sectorial operators and semigroups}

Let $X$ be a Banach space and $\omega \in [0,\pi)$. We say that a closed linear operator $T$ on $X$ is sectorial of angle $\omega$ if its spectrum $\sigma(T)$ is contained in the closed sector $\overline{\S_\omega}$ and for every $\theta\in (0, \pi-\omega)$ there exists a constant $C>0$ such that
\begin{align*}
    \|\lambda (\lambda +T)^{-1}\|_{\cL(X)} \leq C
\end{align*}
for all $\lambda \in \S_\theta$, see \cite[Sec.~2.1]{Haase} for further details. We denote by
\begin{align*}
    \omega_T \coloneqq \inf\{\omega  : \text{ $T$ is sectorial of angle $\omega$}\}
\end{align*}
the smallest of such angles. If $\omega_T\in[0,\frac{\pi}{2})$, then $-T$ generates a bounded holomorphic semigroup $(e^{-zT})_{z\in\S_{\frac{\pi}{2}-\omega_T}}$ on $X$. For any $\beta\in[0,\frac{\pi}{2}-\omega_T)$ and $z\in \S_{\beta}$, we can represent the semigroup via the Cauchy integral formula
\begin{align}
\label{eq: rep semigroup}
    e^{-zT}= \frac{1}{2\pi i} \int\limits_{\gamma_z} e^{z\lambda} (\lambda+T)^{-1}\,\d \lambda,
\end{align}
where the path $\gamma_z$ parameterizes $\partial \big(\S_\theta \cup B(0,1/|z|) \big)$ for $\theta\in (\pi/2+|\arg(z)|,\pi - \omega_T)$. We refer to the monographs \cite{Haase} and \cite{Engel_Nagel} for further information.

\subsection{Holomorphic functional calculus}

For $\vartheta\in [0,\pi)$ define the class of holomorphic functions with some decay in zero and at infinity by
\begin{align*}
    \H^\infty_0(\S_\vartheta) \coloneqq \big\{ \varphi:\S_\vartheta\to \IC \text{ holomorphic } :\exists C,s>0 \text{ s.t.\@ } |\varphi(z)|\leq C\min\{|z|^s,|z|^{-s}\} \big\}.
\end{align*}
Let $T$ be a sectorial operator of angle $\omega\in[0,\pi)$ on a Banach space $X$. For angles satisfying $\omega<\theta<\vartheta <\pi$ and any function $\varphi\in \H^\infty_0(\S_\vartheta)$, we define the bounded linear operator $\varphi(T)$ on $X$ via
\begin{align*}
     \varphi(T) \coloneqq \frac{1}{2\pi i} \int\limits_{\partial \S_\theta} \varphi(z)(z-T)^{-1}\,\d z,
\end{align*}
where $\partial \S_{\theta}$ is parameterized such that it surrounds the spectrum of $T$ counter-clockwise in the extended complex plane; we refer to \cite[Sec.~2.3]{Haase} for details of this construction. If $T$ is additionally injective, we extend the functional calculus to the set
\begin{align*}
    \H^\infty(\S_\vartheta) \coloneqq \big\{ \varphi:\S_\vartheta\to \IC \text{ holomorphic and bounded}\big\}
\end{align*}
via regularization. More precisely, for $\varphi\in \H^\infty(\S_\vartheta)$, the operator $\varphi(T)$ is defined as a closed linear operator through
\begin{align*}
    \varphi(T) \coloneqq [T(1+T)^{-2}]^{-1} [z(1+z)^{-2}\varphi(z)](T).
\end{align*}
In this case, we say that $T$ possesses a bounded $\H^\infty(\S_\vartheta)$-calculus if there exists $C_\vartheta>0$ such that for all $\varphi \in \H^\infty(\S_\vartheta)$ we have $\varphi(T)\in \cL(X)$ and
\begin{align*}
    \|\varphi(T)\|_{\cL(X)} \leq C_\vartheta \|\varphi\|_{\infty,\vartheta},
\end{align*}
where $\|\cdot\|_{\infty,\vartheta}$ denotes the supremum norm over the sector $\S_\vartheta$.
Finally, if $T$ has a bounded $\H^\infty$-calculus of some angle $\vartheta$, so does its adjoint $T^*$ and $(\varphi(T))^* = \varphi^*(T^*)$ holds for all $\varphi\in\H^\infty(\S_\vartheta)$, where $\varphi^*(z) = \overline{\varphi(\overline{z})}$.

\subsection{Generalized Stokes operator}

In this section we introduce generalized Stokes operators via the form method, following the standard procedure for elliptic operators in divergence form.

To simplify our expressions, given $x \in \IR^d$ and a matrix $G \in \IC^{d \times d}$, we define the action of our coefficients $\mu = (\mu_{\alpha \beta}^{i j})_{\alpha , \beta , i , j = 1}^d$ on matrices by
\begin{align*}
    \mu(x) G \coloneqq \Big( \sum_{j , \beta = 1}^d \mu_{\alpha \beta}^{i j} (x) G_{\beta j} \Big)_{\alpha, i = 1}^d
\end{align*}
and we denote the matrix inner product by $G \cdot \overline{H} = \sum_{\alpha , i = 1}^d G_{\alpha i} \overline{H_{\alpha i}}$ for $G , H \in \IC^{d \times d}$.

\begin{assumption}
\label{Ass: Coefficients}
The coefficients $\mu = (\mu_{\alpha \beta}^{i j})_{\alpha , \beta , i , j = 1}^d$ with $\mu_{\alpha \beta}^{i j} \in \L^{\infty} (\IR^d ; \IC)$ for all $1 \leq \alpha , \beta , i , j \leq d$ satisfy for some $\mu_{\bullet} , \mu^{\bullet} > 0$ the inequalities
\begin{align*}
 \Re \int_{\IR^d} \mu \nabla u \cdot \overline{\nabla u} \, \d x \geq \mu_{\bullet} \| \nabla u \|_{\L^2}^2 \qquad (u \in \W^{1,2} (\IR^d ; \IC^d))
\end{align*}
and
\begin{align*}
    \esssup_{x \in \IR^d} \|\mu(x)\|_{\cL(\IC^{d \times d})} \leq \mu^{\bullet}.
\end{align*}
\end{assumption} 

These assumptions imply that the sesquilinear form
\begin{align*}
    \a: \W^{1,2}_\sigma(\IR^d)\times \W^{1,2}_\sigma(\IR^d) \to \IC, \qquad \a(u,v) \coloneqq \int_{\IR^d}  \mu \nabla u \cdot \overline{\nabla v} \, \d x
\end{align*}
is bounded and coercive. Its numerical range
\begin{align*}
    N(\a) \coloneqq \big\{ \a (u,u) : u\in \W^{1,2}_\sigma(\IR^d), \|u\|_{\W^{1,2}}=1 \big\}
\end{align*}
is contained in the closed sector $\overline{\S_{\omega_0 }}$, where we set
\begin{align}
    \label{eq: def omega}
        \omega_0 \coloneqq \arctan\Big(\frac{\mu^{\bullet}}{\mu_{\bullet}}\Big) \in[0,\frac{\pi}{2}).
\end{align}

By the Lax–Milgram lemma, the form $\a$ induces a bounded linear isomorphism $\cA$, which we call the \textit{weak generalized Stokes operator} 
\begin{align*}
    \mathcal{A}:\W^{1,2}_\sigma(\IR^d)\subset\W^{-1,2}_\sigma(\IR^d)\to  \W^{-1,2}_\sigma(\IR^d), \quad (\mathcal{A}u)(v)\coloneqq \a (u,v).
\end{align*}
With the natural inclusion $\iota$ and its adjoint $\cP$ introduced above, we can rewrite the weak generalized Stokes operator via
\begin{align*}
    \langle \cA u , v \rangle_{\W^{-1,2}_{\sigma},\W^{1,2}_{\sigma}} &= \int\limits_{\IR^d} \mu \nabla u \cdot \overline{\nabla \iota(v)} \, \d x \\
    &= \langle -\div(\mu \nabla u) , \iota (v) \rangle_{\W^{-1,2} , \W^{1,2}} = \langle -\cP \div(\mu \nabla u) , v \rangle_{\W^{-1,2}_{\sigma} , \W^{1,2}_{\sigma}}
\end{align*}
for all $u , v \in \W^{1,2}_{\sigma} (\IR^d)$.

We realize the generalized Stokes operator $A$ in $\L^2_\sigma(\IR^d)$ as the part of $\cA$ in $\L^2_\sigma(\IR^d)$. Equivalently, $A$ can be directly defined by
\begin{align*}
    \begin{cases}
        \dom(A) &\coloneqq \big\{u \in \W^{1,2}_{\sigma} (\IR^d) : \, \exists f \in \L^2_{\sigma} (\IR^d) \text{ s.t. } \a (u , v) = \langle f, v\rangle_{\L^2},  \forall v \in \W^{1,2}_{\sigma} (\IR^d)\big\},\\
        \hfill A u &\coloneqq f.
    \end{cases}
\end{align*}

For $x\in\IR^d$ let $\mu^*(x)$ denote the adjoint of $\mu(x)$ with respect to the matrix inner product defined above. Then, for all $u,v\in\W^{1,2}(\IR^d;\IC^d)$ we have
\begin{align*}
    \int\limits_{\IR^d} \mu^*\nabla u\cdot\overline{\nabla v} \,\d x
    = \overline{\int\limits_{\IR^d} \mu\nabla v\cdot\overline{\nabla u }\, \d x},
\end{align*}
so that $\mu^*$ again satisfies Assumption~\ref{Ass: Coefficients} with the same constants $\mu_\bullet, \mu^\bullet$. Restricting to $\W^{1,2}_\sigma(\IR^d)$ induces the adjoint form
\begin{align*}
    \a^*(u,v) \coloneqq \overline{\a(v,u)}
    = \int\limits_{\IR^d} \mu^*\nabla u\cdot\overline{\nabla v}\, \d x
    \qquad (u,v\in\W^{1,2}_\sigma(\IR^d)),
\end{align*}
and we denote by $\cA^*$ and $A^*$ the (weak) generalized Stokes operators associated with $\a^*$. These are the adjoints of $\cA$ and $A$, see \cite[Sec.~6.2]{Bechtel} for elliptic systems. In particular, $\cA^*$ and $A^*$ have the same sectoriality angle $\omega_0$, and every result of this paper applies verbatim to them. We use this in duality arguments without further mention.

The operator $A$ is linked to the generalized Stokes system~\eqref{eq: system} containing the pressure by a standard argument (see \cite[Lem.~II.2.2.1]{Sohr}), which implies that a function $u\in \W^{1,2}_\sigma(\IR^d)$ satisfies
\begin{align*}
     \int\limits_{\IR^d} \mu \nabla u \cdot \overline{\nabla v} \, \d x = \langle f, v\rangle_{\L^2} , \quad (\forall v\in \C^\infty_{c,\sigma}(\IR^d))
\end{align*}
for a given $f\in\L^2_\sigma(\IR^d)$ if and only if there exists $\phi\in \L^2_\loc(\IR^d)$, unique up to an additive constant, such that $(u,\phi)$ satisfies
\begin{align*}
     \int\limits_{\IR^d} \mu \nabla u \cdot \overline{\nabla v} \, \d x + \int\limits_{\IR^d} \phi\, \overline{\div (v)}\,\d x = \langle f, v\rangle_{\L^2} , \quad (\forall v\in \C_c^\infty(\IR^d;\IC^d)).
\end{align*}

Despite the roughness of the coefficients in Assumption~\ref{Ass: Coefficients}, both $A$ and $\cA$ inherit structural properties. In particular, they are sectorial operators of angle $ \omega_0<\frac{\pi}{2}$, admitting uniform resolvent estimates. We recall the following quantitative bounds from \cite[Prop.~2.4]{Haardt_Tolksdorf}.

\begin{proposition}
\label{Prop: L2 resolvent bounds}
    Let $\mu$ satisfy Assumption~\ref{Ass: Coefficients} and let $\omega_0$ be given by \eqref{eq: def omega}. Then, we have $\sigma(A) \cup \sigma( \cA)\subset \overline{\S_{\omega_0}}$ and for all $\theta\in [0,\pi-\omega_0)$ there exists $C>0$ such that for all $\lambda \in \S_\theta$ and all $f\in \L^2_\sigma(\IR^d)$ we have
    \begin{align*}
        \|\lambda(\lambda + A)^{-1}f\|_{\L^2} + \||\lambda|^\frac{1}{2}\nabla(\lambda + A)^{-1}f\|_{\L^2} \leq C\|f\|_{\L^2}.
    \end{align*}
    Moreover, there exists $C>0$ such that for all $\lambda \in  \S_\theta$ and all $F\in \L^2(\IR^d;\IC^{d\times d})$ we have
    \begin{align*}
        \||\lambda|^\frac{1}{2}(\lambda + \cA)^{-1}\cP\div(F)\|_{\L^2} + \|\nabla(\lambda + \cA)^{-1}\cP\div(F)\|_{\L^2} \leq C\|F\|_{\L^2}.
    \end{align*}
    The constant $C$ only depends on $\theta$, $d$, $\mu^{\bullet}$ and $\mu_{\bullet}$.
\end{proposition}

Because the sectoriality angle is strictly less than $\frac{\pi}{2}$, both operators $-A$ and $-\cA$ generate bounded holomorphic semigroups. Applying the standard Cauchy integral formula~\eqref{eq: rep semigroup} for semigroups along with the resolvent bounds of Proposition~\ref{Prop: L2 resolvent bounds}, we obtain the following uniform bounds.

\begin{proposition}
    \label{prop: L2 semigroup bounds}
    Let $\mu$ satisfy Assumption~\ref{Ass: Coefficients}, $\omega_0$ be given by \eqref{eq: def omega} and $\beta\in [0,\frac{\pi}{2}-\omega_0)$. There exists $C>0$ such that for all $z \in \S_\beta$ and all $f\in \L^2_\sigma(\IR^d)$ we have
    \begin{align*}
        \|e^{-zA}f\|_{\L^2}+\|zAe^{-zA}f\|_{\L^2} + \||z|^\frac{1}{2}\nabla e^{-zA}f\|_{\L^2} \leq C\|f\|_{\L^2}.
    \end{align*}
    Moreover, there exists $C>0$ such that for all $z \in  \S_\beta$ and all $F\in \L^2(\IR^d;\IC^{d\times d})$ we have
    \begin{align*}
        \||z|^\frac{1}{2}e^{-z\cA}\cP\div(F)\|_{\L^2} + \|z\nabla e^{-z\cA}\cP\div(F)\|_{\L^2} \leq C\|F\|_{\L^2}.
    \end{align*}
    The constant $C$ only depends on $\beta$, $d$, $\mu^{\bullet}$ and $\mu_{\bullet}$.
\end{proposition}

Finally, we remark that Assumption~\ref{Ass: Coefficients} yields even deeper functional analytic properties. Specifically, $A$ is a maximal $\omega$-accretive operator for any $\omega \in (\omega_0, \frac{\pi}{2})$, a condition that is strictly stronger than mere sectoriality; see \cite[Sec.~7.1.1]{Haase}. By virtue of this property, $A$ automatically possesses a bounded $\H^\infty$-calculus on $\L^2_\sigma$ (see for example \cite[Cor.~7.1.17]{Haase}).

\begin{proposition}
\label{prop: bdd Hinf on L2}
    Let $\mu$ satisfy Assumption~\ref{Ass: Coefficients} and let $\omega_0$ be given by \eqref{eq: def omega}.
    For any $\vartheta\in (\omega_0,\pi)$ the operator $A$ has a bounded $\H^\infty(\S_\vartheta)$-calculus on $\L^2_\sigma(\IR^d)$.
\end{proposition}

\subsection{Transformed Stokes operator}
\label{sec: rescaling}

In this section we introduce generalized Stokes operators associated with rescaled and rotated coefficients. In addition, we establish representation formulas for their corresponding resolvents and semigroups. Since Assumption~\ref{Ass: Coefficients} is preserved under these transformations, norm estimates will remain uniform. This will be advantageous in rescaling arguments. \par

For $s > 0$ let $\Phi_sf$ be defined by $(\Phi_sf)(x) = f(sx)$ for all $x\in \IR^d$. We regard $\Phi_s$ as a map on both $\L^2_\sigma(\IR^d)$ and $\W^{1,2}_\sigma(\IR^d)$ noting that it preserves the divergence-freeness of vector fields. Furthermore, we extend $\Phi_s$ to the space $\W^{-1,2}_\sigma(\IR^d)$ via the identity $\langle\Phi_s u, v\rangle = \langle u, s^{-d}\Phi_{s^{-1}}v\rangle$. The rescaled operators are then associated to the following rescaled sesquilinear forms
\begin{align*}
    \a_s: \W^{1,2}_\sigma(\IR^d)\times \W^{1,2}_\sigma(\IR^d) \to \IC, \qquad \a_s(u,v) = \int\limits_{\IR^d}(\Phi_s\mu)\nabla u\cdot \overline{\nabla v}\,\d x.
\end{align*}
The following lemma shows that all these forms share the same quantitative bounds.

\begin{lemma}
\label{lem: scaling}
    For all $s>0$ we have
    \begin{align*}
        \a_s(u,v) = s^{2-d}\a_1(\Phi_{s^{-1}}u, \Phi_{s^{-1}}v) \qquad (u,v \in  \W^{1,2}_\sigma(\IR^d))
    \end{align*}
    as well as
    \begin{align*}
        |\a_s(u,v)|\leq \mu^\bullet \|\nabla u\|_{\L^2} \|\nabla v\|_{\L^2} \quad \text{and} \quad \Re (\a_s(u,u))\geq \mu_\bullet\|\nabla u\|^2_{\L^2} \qquad (u,v \in  \W^{1,2}_\sigma(\IR^d)).
    \end{align*}
\end{lemma}

\begin{proof}
    The upper bound readily follows from H\"older's inequality via
    \begin{align*}
        |\a_s(u,v)| \leq  \int\limits_{\IR^d}|\mu(sx)\nabla u(x)\cdot \overline{\nabla v(x)}|\,\d x \leq \mu^\bullet\int\limits_{\IR^d}|\nabla u(x)| |\nabla v(x)|\,\d x \leq \mu^\bullet\|\nabla u\|_{\L^2} \|\nabla v\|_{\L^2}.
    \end{align*}
    To establish the lower bound, observe first that for $u,v \in  \W^{1,2}_\sigma(\IR^d)$ the substitution $sx=y$ yields
    \begin{align*}
         \a_s(u,v) = \int\limits_{\IR^d}\mu(sx)\nabla u(x)\cdot \overline{\nabla v(x)}\,\d x
         = s^{2-d}\a_1(\Phi_{s^{-1}}u, \Phi_{s^{-1}}v).
    \end{align*}
    Invoking Assumption~\ref{Ass: Coefficients}, we conclude
    \begin{align*}
        \Re\bigl(\a_s(u,u)\bigr) = s^{2-d}\Re\bigl(\a_1(\Phi_{s^{-1}}u, \Phi_{s^{-1}}u)\bigr) &\geq \mu_\bullet s^{2-d}\|\nabla (\Phi_{s^{-1}}u)\|^2_{\L^2} = \mu_\bullet \|\nabla u\|^2_{\L^2},
    \end{align*}
    completing the proof.
\end{proof}

Let $\cA_s$ and $A_s$ denote the (weak) generalized Stokes operators associated to the form $\a_s$. Then the following lemma provides a connection between the rescaled operators and the original ones.

\begin{lemma}
\label{lem: rescaled resolvent and semigroup}
    For all $s > 0$ we have
    \begin{align*}
        (1+A_s)^{-1} = \Phi_s (1+s^2A)^{-1}\Phi_{s^{-1}} \quad \text{and} \quad (1+\cA_s)^{-1} = \Phi_s (1+s^2\cA)^{-1}\Phi_{s^{-1}},
    \end{align*}
    as well as
    \begin{align*}
        e^{-A_s} = \Phi_s e^{-s^2A}\Phi_{s^{-1}} \quad \text{and} \quad e^{-\cA_s} = \Phi_s e^{-s^2\cA}\Phi_{s^{-1}}.
    \end{align*}
\end{lemma}

\begin{proof}
    Because the corresponding identities for $A$ follow directly by restriction, it suffices to establish them for the operator $\cA$. Applying the previous lemma, we obtain for every $u,v\in\W^{1,2}_\sigma(\IR^d)$
    \begin{align*}
        \langle \cA_su, v\rangle = s^{2-d}\a_1(\Phi_{s^{-1}}u, \Phi_{s^{-1}}v) = s^{2-d}\langle \cA\Phi_{s^{-1}}u, \Phi_{s^{-1}}v\rangle =\langle s^2 \Phi_{s}\cA\Phi_{s^{-1}}u, v\rangle.
    \end{align*}
    Since $\Phi_s$ is an automorphism on both $\W^{1,2}_\sigma(\IR^d)$ and $ \W^{-1,2}_\sigma(\IR^d)$, we have $\dom(\cA_s) = \dom(\cA)$ and
    \begin{align*}
        (1+\cA_s)^{-1} = (1+s^{2} \Phi_{s} \cA \Phi_{s^{-1}})^{-1} =\Phi_{s}(1+s^{2}  \cA )^{-1}\Phi_{s^{-1}}.
    \end{align*}
    The corresponding identity for the semigroup follows now from the representation~\eqref{eq: rep semigroup}.
\end{proof}

Next, we study the behavior of the form $\a$ under rotation. Given some angle $\theta\in (-(\frac{\pi}{2}-\omega_0),\frac{\pi}{2}-\omega_0)$, we define the rotated sesquilinear form
\begin{align*}
    e^{i\theta}\a: \W^{1,2}_\sigma(\IR^d)\times \W^{1,2}_\sigma(\IR^d) \to \IC, \qquad e^{i\theta}\a(u,v) \coloneqq \int\limits_{\IR^d}( e^{i\theta}\mu)(x)\nabla u(x)\cdot \overline{\nabla v(x)}\,\d x.
\end{align*}

\begin{lemma}
\label{lem: rotation}
    Let $\theta\in (-(\frac{\pi}{2}-\omega_0),\frac{\pi}{2}-\omega_0)$. Then, the rotated form $e^{i\theta}\a$ satisfies
    \begin{align*}
        |e^{i\theta}\a(u,v)|\leq \mu^\bullet \|\nabla u \|_{\L^2} \|\nabla v\|_{\L^2}
    \end{align*}
    and 
    \begin{align*}
        \Re\big( e^{i\theta}\a(u,u)\big) \geq\mu_\bullet \cos(|\theta|+\omega_0) \|\nabla u\|_{\L^2}^2
    \end{align*}
    for all $u,v\in \W^{1,2}_\sigma(\IR^d)$.
\end{lemma}

\begin{proof}
    The first inequality is clear. For the second, assume without loss of generality that $\theta> 0$ and $u\in \W^{1,2}_\sigma(\IR^d)$ satisfies $\|\nabla u\|_{\L^2}=1$. By Assumption~\ref{Ass: Coefficients}, we have
    \begin{align*}
        e^{i\theta}\a(u,u)\in \overline{\S_{\theta + \omega_0}} \quad \text{and} \quad \Re (\a(u,u))\geq \mu_\bullet.
    \end{align*}
    Consequently, we calculate
    \begin{align*}
        \Re(e^{i\theta}\a(u,u)) &= |e^{i\theta}\a(u,u)|\cos(\theta+\arg(\a(u,u)))\\
    &\geq \Re(\a(u,u) ) \cos(\theta+\omega_0)\\
    &\geq \mu_\bullet \cos(\theta+\omega_0),
    \end{align*}
    which proves the claim.
\end{proof}

We conclude this section with the following important observation, which will be used in Lemma~\ref{lem: rescaling lemma} to reduce arguments to the case $z = 1$.

\begin{observation}
\label{observation}
    By Lemma~\ref{lem: scaling} and~\ref{lem: rotation}, the rescaled and rotated coefficients $e^{i\theta}\Phi_s\mu$ again satisfy Assumption~\ref{Ass: Coefficients}, with the same upper bound $\mu^\bullet$ and with $\mu_\bullet$ replaced by $\mu_\bullet\cos(|\theta|+\omega_0)$. Consequently, all operators $e^{i\theta}A_s$, for $s>0$ and $\theta\in (-(\frac{\pi}{2}-\omega_0),\frac{\pi}{2}-\omega_0)$, belong to the same class of maximal $\omega$-accretive operators as $A$, and every estimate whose constant depends only on $d$ and the ellipticity constants therefore holds for all of them, uniformly in $s>0$ and locally uniformly in $\theta$. We use this in rescaling and rotation arguments without further mention.    
\end{observation}

\subsection{Bogovski\u{\i} operator}
\label{sec: Bogovskii correction}

In this section, we introduce so-called Bogovski\u{\i} operators. They will be used to restore divergence-freeness of functions in a localized sense.

Fix the ball $B_0=B(0,1)$ and recall $C_1(B_0) = 2B_0\setminus\overline{B_0}$. We define the average free $\L^2$-space over $C_1(B_0)$ by
\begin{align*}
    \L^2_0(C_1(B_0)) \coloneqq  \bigg\{g\in \L^2(C_1(B_0)) : \int\limits_{C_1(B_0)} g\,\d x =0 \bigg\}.
\end{align*}
A Bogovski\u{\i} operator $\cB_{C_1(B_0)} :\L^2_0 (C_1(B_0)) \to \W^{1,2}_0 (C_1(B_0) ; \IC^d)$ denotes a solution operator to the divergence equation
\begin{align*}
    \begin{cases}
        \div(u) = f & \text{in } C_1(B_0), \\
    \hfill u = 0 & \text{on } \partial C_1(B_0),
    \end{cases}
\end{align*}
where $f \in \L^2_0 (C_1(B_0))$. Here, we use the linear operator constructed in \cite[Sec.~III.3]{Galdi}, which satisfies
\begin{align*}
    \div(\cB_{C_1(B_0)} f) = f \quad \text{and} \quad \| \cB_{C_1(B_0)} f \|_{\W^{1,2} (C_1(B_0))} \leq C_{\mathrm{Bog}} \| f \|_{\L^2(C_1(B_0))}
\end{align*}
for all $f \in \L^2_0 (C_1(B_0))$ and some constant depending only on $d$. By scaling, we construct Bogovski\u{\i} operators on $\alpha C_1(B_0) = C_1(\alpha B_0)$ for $\alpha > 0$ as follows: if $f \in \L^2_0 (C_1(\alpha B_0))$, then $f_{\alpha} (x) \coloneqq \alpha f (\alpha x)$ belongs to $\L^2_0 (C_1(B_0))$. For $x \in C_1(\alpha B_0)$ and $f \in \L^2_0 (C_1(\alpha B_0))$ set
\begin{align*}
    [\cB_{C_1(\alpha B_0)} f] (x) \coloneqq [\cB_{C_1(B_0)} f_{\alpha}] (\alpha^{-1} x).
\end{align*}
Clearly, $\cB_{C_1(\alpha B_0)}$ is bounded from $\L^2_0 (C_1(\alpha B_0))$ into $\W^{1,2}_0 (C_1(\alpha B_0) ; \IC^d)$ and satisfies $\div (\cB_{C_1(\alpha B_0)} f) = f$. Furthermore, we have
\begin{align}
\label{Eq: Inhomogeneous estimate Bogovskii}
    \| \cB_{C_1(\alpha B_0)} f \|_{\L^2 (C_1(\alpha B_0))} \leq \alpha C_{\mathrm{Bog}} \| f \|_{\L^2 (C_1(\alpha B_0))}
\end{align}
and
\begin{align}
\label{Eq: Homogeneous estimate Bogovskii}
    \| \nabla \cB_{C_1(\alpha B_0)} f \|_{\L^2 (C_1(\alpha B_0))} \leq C_{\mathrm{Bog}} \| f \|_{\L^2 ( C_1(\alpha B_0))} 
\end{align}
for all $f \in \L^2_0(C_1(\alpha B_0))$ and the same constant $C_{\mathrm{Bog}} > 0$ as above. Finally, by translation we may define Bogovski\u{\i} operators for translated annuli as well, and these will again satisfy~\eqref{Eq: Inhomogeneous estimate Bogovskii} and~\eqref{Eq: Homogeneous estimate Bogovskii} with the same constant. 

The following lemma details how the Bogovski\u{\i} operator can be used to localize functions without destroying their divergence-freeness.

\begin{lemma}
\label{lem: Bogovskii lemma}
    Let $B=B(x_0,r)\subset \IR^d$ be a ball, $f\in \Llocs^2(\IR^d)$ and $\eta\in \C_c^\infty(\IR^d)$ with $\supp \eta \subset 2B$, $\eta \equiv 1$ on $B$ and $\|\eta\|_{\L^\infty}+r\|\nabla \eta\|_{\L^\infty}\lesssim 1$. Then, $\nabla\eta \cdot f \in \L^2_0 (C_1(B))$ and we have
    \begin{align}
    \label{ineq: grad Bogovskii}
        \|r\nabla \cB_{C_1(B)}(\nabla\eta\cdot f)\|_{\L^2(\IR^d)}\lesssim  \|f\|_{\L^2(C_1(B))},
    \end{align}
    as well as
    \begin{align}
    \label{ineq: Bogovskii}
        \| \cB_{C_1(B)}(\nabla\eta\cdot f)\|_{\L^2(\IR^d)}\lesssim \|f\|_{\L^2(C_1(B))},
    \end{align}
    where $\cB_{C_1(B)}(\nabla\eta\cdot f)$ is understood as a vector field in $\W^{1,2}(\IR^d;\IC^d)$ via extension by zero. Moreover, if we define
    \begin{align*}
        f_{2r} \coloneqq \eta f -  \cB_{C_1(B)}(\nabla\eta\cdot f),
    \end{align*}
    then $f_{2r} \in \L^2_\sigma(\IR^d)$ with $\supp f_{2r}\subset 2B$, $f_{2r}=f$ on $B$ and satisfies $\|f_{2r}\|_{\L^2(\IR^d)} \lesssim \|f\|_{\L^2(2B)}$.
\end{lemma}

\begin{proof}
    It is clear that $\nabla \eta \cdot f \in \L^2(2B)$. Because $\supp \eta \subset 2B$ and $\eta=1$ on $B$, we have $\supp (\nabla\eta\cdot f) \subset 2B\setminus B$. Since $\div f = 0$ holds in the sense of distributions, it follows
    \begin{align*}
        \int\limits_{C_1(B)} (\nabla\eta)(x)\cdot f(x)\, \d x = -\langle \div (f), \eta \rangle=0.
    \end{align*}
    Thus, $\nabla \eta \cdot f \in \L^2_0(C_1(B))$. Applying \eqref{Eq: Inhomogeneous estimate Bogovskii} and \eqref{Eq: Homogeneous estimate Bogovskii}, as well as the estimate on $\eta$, we obtain
    \begin{align*}
        \|r\nabla \cB_{C_1(B)}(\nabla\eta\cdot f)\|_{\L^2(\IR^d)} \lesssim r \|\nabla\eta \cdot f\|_{\L^2(C_1(B))}\lesssim  \|f\|_{\L^2(C_1(B))},
    \end{align*}
    as well as
    \begin{align*}
        \| \cB_{C_1(B)}(\nabla\eta\cdot f)\|_{\L^2(\IR^d)}\lesssim  r\|\nabla\eta \cdot f\|_{\L^2(C_1(B))} \lesssim\|f\|_{\L^2(C_1(B))}.
    \end{align*}
    Finally, we see that $f_{2r}\in \L^2(\IR^d;\IC^d)$ is divergence-free since
    \begin{align*}
        \div f_{2r} = \div (\eta f) -\div \cB_{C_1(B)}(\nabla\eta\cdot f)=\nabla \eta \cdot f - \nabla \eta \cdot f=0
    \end{align*}
    holds in the sense of distributions. Moreover, we observe that $f_{2r} =f$ on $B$ and
    \begin{align*}
        \supp f_{2r} \subset \supp (\eta f) \cup \supp (\cB_{C_1(B)}(\nabla\eta\cdot f)) \subset 2B.
    \end{align*}
    Finally, we have
    \begin{align*}
        \|f_{2r}\|_{\L^2} \leq \|\eta f\|_{\L^2} + \|\cB_{C_1(B)}(\nabla\eta\cdot f)\|_{\L^2}\lesssim \|f\|_{\L^2(2B)}
    \end{align*}
    by the estimate on $\eta$ and \eqref{ineq: Bogovskii}.
\end{proof}

\section{Decay estimates for the Stokes resolvent}
\label{sec: off-diag est}

In this section, we investigate decay properties of the generalized Stokes resolvent. More precisely, given a ball $B\subset \IR^d$ and its $k$-th annulus $C_k(B)$, we aim to quantify how much the size of $f$ and/or $F$ on $C_k(B)$ affects the size of the solution $u$ and/or its gradient on $B$. In the elliptic setting $L=-\div(\mu\nabla \cdot)$, analogous estimates are provided by so-called off-diagonal estimates, see for example \cite[Lem.~2.1]{Auscher-Kato}: there exist constants $C , c > 0$ such that for all functions $f,F \in \L^2$ with $\supp(f), \supp(F) \subset C_k(B)$ one has
\begin{align*}
 \|t (1 + t^2 \cL)^{-1} \div(F)\|_{\L^2(B)} + \| (1 + t^2 L)^{-1} f \|_{\L^2 (B)} \leq C e^{- \frac{c 2^k r}{t}} (\| f \|_{\L^2 (C_k(B))}+ \| F\|_{\L^2 (C_k(B))}).
\end{align*}
Recalling the explicit formulas for Bessel kernels~\cite[Sec.~V.3]{Stein_singular}, one sees that the estimates above are optimal already for the Laplacian.

So far, analogous estimates for the generalized Stokes resolvent remain largely open; see the discussion in the introduction.

In \cite[Prop.~2.5]{Haardt_Tolksdorf} decay estimates were obtained in the regime $|\lambda|r^2\simeq 1$, which sufficed for the Kato square root property. However, for the $\L^p$-theory, we need estimates that are uniform in $|\lambda|r^2$ and that can be integrated along the Cauchy contour. We therefore proceed in three steps. First, we prove an $\L^2$-decay estimate that also covers non-zero $f$ and has simpler decay coefficients, streamlining subsequent calculations in applications (Theorem~\ref{thm: off-diag resolvent}). Second, we improve the behavior in $|\lambda|$ and $r$ when the data are supported away from $2B$ (Proposition~\ref{prop: imp off-diag est}). Third, we upgrade to $\L^2$-$\L^p$ estimates by Sobolev's embedding (Theorem~\ref{thm: L2-Lp off diag stokes resolv}).

\begin{theorem}[$\L^2$-decay estimates]
\label{thm: off-diag resolvent}
    Let $\mu$ satisfy Assumption~\ref{Ass: Coefficients} and let $\omega_0$ be given by \eqref{eq: def omega}. For all $\theta \in [0,\pi-\omega_0)$ and all $\nu \in (0, d+2)$ there exist constants $C_1 , C_2 >0$ such that for all balls $B= B(x_0,r)$, $\lambda \in \S_\theta$, $f\in \L_\sigma^2(\IR^d)$ and $F \in \L^2 (\IR^d ; \IC^{d \times d})$ the solution $u\in \W^{1,2}_\sigma(\IR^d)$ to the Stokes system \eqref{eq: Stokes system} satisfies
    \begin{align}
    \label{eq: off-diag resolvent}
        &\sum\limits_{k=0}^\infty 2^{-\nu k } \int\limits_{B_k}|u|^2\,\d x +  \sum\limits_{k=0}^\infty 2^{-\nu k }\frac{2^{2k}r^2}{C_1+ |\lambda| r^2 2^{2k}} \int\limits_{B_k}|\nabla u|^2\,\d x \nonumber\\
        &\leq \frac{C_2}{|\lambda|} \max\Big\{1, (|\lambda|r^2)^{-(d+1)}\Big\}\sum\limits_{n=0}^\infty 2^{-\nu n}\bigg(\frac{1}{|\lambda|}\int\limits_{C_n(B)}|f|^2\, \d x +\int\limits_{C_n(B)}|F|^2\,\d x\bigg).
    \end{align}
    The constants $C_1 , C_2$ depend only on $\theta$, $\nu$, $d$ and $\mu_\bullet,\mu^\bullet$.
\end{theorem}

 We start with a non-local Caccioppoli-type inequality (Lemma~\ref{lem: iteration setup}) that will subsequently be iterated (Lemma~\ref{lem: iteration procedure}).

\begin{lemma}[Non-local Caccioppoli-type inequality]
\label{lem: iteration setup}
    Let $\mu$ satisfy Assumption~\ref{Ass: Coefficients} and let $\omega_0$ be given by \eqref{eq: def omega}. For all $\theta \in [0,\pi-\omega_0)$ and all $\nu \in (0, d+2)$ there exist constants $C_1 , C_2 >0$ such that for all balls $B= B(x_0,r)$, $\lambda \in \S_\theta$, $f\in \L_\sigma^2(\IR^d)$ and $F \in \L^2 (\IR^d ; \IC^{d \times d})$ the solution $u\in \W^{1,2}_\sigma(\IR^d)$ to the Stokes system \eqref{eq: Stokes system} satisfies
    \begin{align}
    \label{eq: iteration equation}
        &\sum\limits_{k=0}^\infty 2^{-\nu k}\int\limits_{B_k}|u|^2\,\d x + \frac{1}{2}\sum\limits_{k=0}^\infty 2^{-\nu k}\frac{2^{2k}r^2}{C_1 + |\lambda|r^2 2^{2k}}\int\limits_{B_k}|\nabla u|^2\,\d x  \nonumber\\
        & \leq \sum\limits_{k=0}^\infty 2^{-\nu k}\frac{C_1}{C_1 + |\lambda|r^2 2^{2k}}\int\limits_{B_{k+1}}|u|^2\,\d x  + \frac{C_2}{|\lambda|}\sum\limits_{k=0}^\infty  2^{-\nu k} \bigg(\frac{1}{|\lambda|}\int\limits_{B_{k+1}}|f|^2\, \d x +\int\limits_{B_{k+1}}|F|^2\,\d x\bigg).
    \end{align}
    The constants $C_1 , C_2$ depend only on $\theta$, $\nu$, $d$ and $\mu_\bullet,\mu^\bullet$.
\end{lemma}

\begin{proof}
    Let $u\in \W^{1,2}_\sigma(\IR^d)$ be the unique solution to the Stokes resolvent problem \eqref{eq: Stokes system} and $\phi\in \Lloc^2(\IR^d)$ its associated pressure. Then, \cite[Lem.~2.2]{Tolksdorf-Caccioppoli} states that there exists $C > 0$ depending only on $\theta$, $d$ and $\mu_\bullet,\mu^\bullet$ such that
    \begin{align}
    \label{Eq: First Caccioppoli}
        |\lambda|\int\limits_{B_k}|u|^2\,\d x + \int\limits_{B_k}|\nabla u|^2\,\d x
        & \leq \frac{C(1+\frac{1}{\varepsilon})}{2^{2k}r^2}\int\limits_{C_{k+1}(B)}|u|^2\,\d x + \varepsilon\int\limits_{C_{k+1}(B)}|\phi - \phi_{C_{k+1}(B)}|^2\, \d x \notag\\
        &\quad+ \frac{C}{|\lambda|}\int\limits_{B_{k+1}}|f|^2\, \d x + C \int\limits_{B_{k+1}}|F|^2\, \d x 
    \end{align}
    for all $k\in\IN_0$ and $\varepsilon>0$ to be chosen later. Using \cite[Lem.~2.1]{Tolksdorf-Caccioppoli} and $C_{\ell}(B) \subseteq B_{\ell}$ allows us to control the pressure by
    \begin{align*}
        \int\limits_{C_{k+1}(B)}|\phi - \phi_{C_{k+1}(B)}|^2\, \d x 
        &\leq C\bigg(\sum\limits_{\ell=0}^{k} 2^{\frac{d}{2}(\ell-k)} \bigl( \|\nabla u\|_{\L^2(C_\ell(B))} +
        \| F \|_{\L^2(C_\ell(B))} \bigr)\\
        &\quad +\sum\limits_{\ell=k+1}^{\infty} 2^{(\frac{d}{2}+1)(k-\ell)} \bigl( \|\nabla u\|_{\L^2(C_\ell(B))} + \| F \|_{\L^2(C_\ell(B))} \bigr) \bigg)^2 \\
        & \leq C\bigg( \sum\limits_{\ell=k+1}^{\infty} 2^{(\frac{d}{2}+1)(k-\ell)} \bigl( \|\nabla u\|_{\L^2(B_\ell)} + \| F \|_{\L^2(B_\ell)} \bigr) \bigg)^2.
    \end{align*}
    Applying the Cauchy--Schwarz inequality, we arrive at
    \begin{align}
    \label{eq: pressure estimate}
        \int\limits_{C_{k+1}(B)}|\phi - \phi_{C_{k+1}(B)}|^2\, \d x \leq C\sum\limits_{\ell=k}^{\infty} 2^{\nu'(k-\ell)} \bigl(\|\nabla u\|_{\L^2(B_\ell)}^2 + \| F \|_{\L^2 (B_{\ell})}^2 \bigr),
    \end{align}
    where we choose $\nu'= \frac{d+2+\nu}{2}\in (\nu,d+2)$. Inserting this into~\eqref{Eq: First Caccioppoli} leads to
    \begin{align*}
        |\lambda|\int\limits_{B_k}|u|^2\,\d x + \int\limits_{B_k}|\nabla u|^2\,\d x &
         \leq \frac{C(1+\frac{1}{\varepsilon})}{2^{2k}r^2}\int\limits_{C_{k+1}(B)}|u|^2\,\d x + C\varepsilon\sum\limits_{\ell=k}^{\infty} 2^{\nu'(k-\ell)} \int\limits_{B_{\ell}} \lvert \nabla u \rvert^2 \, \d x \\
        &\quad+ \frac{C}{|\lambda|}\int\limits_{B_{k+1}}|f|^2\, \d x + C(1 + \varepsilon) \sum\limits_{\ell=k}^{\infty} 2^{\nu'(k-\ell)} \int\limits_{B_{\ell}} \lvert F \rvert^2 \, \d x.
    \end{align*}
    Setting $\varepsilon^{\prime} \coloneqq 1+\frac{1}{\varepsilon}$, we add $\frac{C\varepsilon'}{2^{2k}r^2}\int_{B_{k}}|u|^2\,\d x$ on both sides and thereby `fill the hole' on the right-hand side to get
    \begin{align*}
        &\Big(|\lambda|+\frac{C\varepsilon'}{2^{2k}r^2}\Big) \int\limits_{B_k}|u|^2\,\d x + \int\limits_{B_k}|\nabla u|^2\,\d x\\
        & \leq \frac{C\varepsilon'}{2^{2k}r^2}\int\limits_{B_{k+1}}|u|^2\,\d x  + C\varepsilon\sum\limits_{\ell=k}^{\infty} 2^{\nu'(k-\ell)} \int_{B_\ell} \lvert \nabla u \rvert^2 \, \d x + \frac{C}{|\lambda|}\int\limits_{B_{k+1}}|f|^2\, \d x\\
        &\quad + C(1 + \varepsilon) \sum\limits_{\ell=k}^{\infty} 2^{\nu'(k-\ell)} \int\limits_{B_{\ell}} \lvert F \rvert^2 \, \d x.
    \end{align*}
    From now on, $C$ denotes one fixed constant for which the previous inequality holds. Its precise value is not important; however, it depends only on $\theta,\nu, d$ and $\mu_\bullet,\mu^\bullet$. Dividing the inequality above by the factor $(|\lambda|+\frac{C\varepsilon'}{2^{2k}r^2})$, multiplying all terms by $2^{-\nu k}$ and summing over all $k\in \IN_0$ leads to
    \begin{align*}
        &\sum\limits_{k=0}^\infty  2^{-\nu k} \int\limits_{B_k}|u|^2\,\d x + \sum\limits_{k=0}^\infty 2^{-\nu k}\frac{2^{2k}r^2}{C\varepsilon' + |\lambda|r^2 2^{2k}}\int\limits_{B_k}|\nabla u|^2\,\d x \\
        & \leq \sum\limits_{k=0}^\infty  2^{-\nu k}\frac{C\varepsilon'}{C\varepsilon' + |\lambda|r^2 2^{2k}}\int\limits_{B_{k+1}}|u|^2\,\d x + C\varepsilon \sum\limits_{k=0}^\infty  2^{-\nu k}\frac{2^{2k}r^2}{C\varepsilon' + |\lambda|r^2 2^{2k}}\sum\limits_{\ell=k}^{\infty} 2^{\nu'(k-\ell)}\int\limits_{B_\ell}|\nabla u|^2\,\d x\\
        &\quad+ \frac{C}{|\lambda|}\sum\limits_{k=0}^\infty  2^{-\nu k}\frac{2^{2k}r^2}{C\varepsilon' + |\lambda|r^2 2^{2k}}\int\limits_{B_{k+1}}|f|^2\, \d x\\
        &\quad + C(1 + \varepsilon) \sum\limits_{k=0}^\infty  2^{-\nu k}\frac{2^{2k}r^2}{C\varepsilon' + |\lambda|r^2 2^{2k}}\sum\limits_{\ell=k}^{\infty} 2^{\nu'(k-\ell)}\int\limits_{B_\ell}|F|^2\,\d x.
    \end{align*}
    To calculate the terms including a double series, we first observe that
    \begin{align*}
        \Big( \frac{2^{2k}r^2}{C\varepsilon' + |\lambda|r^2 2^{2k}} \Big)_{k \in \IN_0}
    \end{align*}
    is monotonically increasing in $k$. Since $\nu^{\prime} > \nu$, it follows
    \begin{align*}
        &C\varepsilon \sum\limits_{k=0}^\infty  2^{-\nu k}\frac{2^{2k}r^2}{C\varepsilon' + |\lambda|r^2 2^{2k}}\sum\limits_{\ell=k}^{\infty} 2^{\nu'(k-\ell)}\int\limits_{B_\ell}|\nabla u|^2\,\d x\\
        &\leq C\varepsilon \sum\limits_{\ell=0}^\infty  \bigg(\sum\limits_{k=0}^{\ell} 2^{(\nu'-\nu)k}\bigg) 2^{-\nu' \ell}\frac{2^{2\ell}r^2}{C\varepsilon' + |\lambda|r^2 2^{2\ell}} \int\limits_{B_\ell}|\nabla u|^2\,\d x\\
        &\leq  C C_{\nu',\nu} \varepsilon \sum\limits_{\ell=0}^\infty  2^{-\nu \ell}\frac{2^{2\ell}r^2}{C\varepsilon' + |\lambda|r^2 2^{2\ell}} \int\limits_{B_\ell}|\nabla u|^2\,\d x,
    \end{align*}
    where $C_{\nu',\nu} = \frac{2^{\nu^{\prime} - \nu}}{2^{\nu'-\nu}-1}$. The double series involving $\L^2$-norms of $F$ is treated analogously. Choosing $\varepsilon= \frac{1}{2C\cdot C_{\nu',\nu}}$, we can absorb the series involving $\L^2$-norms of $\nabla u$ onto the left-hand side. In addition, we set $C_1=C \varepsilon^{\prime}$ for the remaining terms to find
    \begin{align*}
        &\sum\limits_{k=0}^\infty 2^{-\nu k}\int\limits_{B_k}|u|^2\,\d x + \frac{1}{2}\sum\limits_{k=0}^\infty 2^{-\nu k}\frac{2^{2k}r^2}{C_1 + |\lambda|r^2 2^{2k}}\int\limits_{B_k}|\nabla u|^2\,\d x  \\
        & \leq \sum\limits_{k=0}^\infty  2^{-\nu k}\frac{C_1}{C_1 + |\lambda|r^2 2^{2k}}\int\limits_{B_{k+1}}|u|^2\,\d x\\
        &\quad + C_2\sum\limits_{k=0}^\infty  2^{-\nu k}\frac{2^{2k}r^2}{C_1 + |\lambda|r^2 2^{2k}} \bigg(\frac{1}{|\lambda|}\int\limits_{B_{k+1}}|f|^2\, \d x +\int\limits_{B_{k+1}}|F|^2\,\d x\bigg)\\
        & \leq \sum\limits_{k=0}^\infty  2^{-\nu k}\frac{C_1}{C_1 + |\lambda|r^2 2^{2k}}\int\limits_{B_{k+1}}|u|^2\,\d x + \frac{C_2}{|\lambda|}\sum\limits_{k=0}^\infty  2^{-\nu k} \bigg(\frac{1}{|\lambda|}\int\limits_{B_{k+1}}|f|^2\, \d x +\int\limits_{B_{k+1}}|F|^2\,\d x\bigg)
    \end{align*}
    for some constant $C_2>0$ depending only on $\theta,d,\nu,\mu_\bullet,\mu^\bullet$.
\end{proof}

To provide a lean proof of the iteration procedure, we introduce the following coefficients.
Let $C_1>0$ be the constant from Lemma~\ref{lem: iteration setup}. Then for  $\lambda\in \IC$, $r>0$ and $m\in \IN_0$ we define
\begin{align}
\label{def: coefficients}
    \kappa_{\lambda,r}(0) \coloneqq 1 \quad \text{and} \quad \kappa_{\lambda,r}(m)\coloneqq \prod\limits_{s= 0}^{m-1} \frac{C_1}{C_1+ |\lambda|r^2 2^{2s}}.
\end{align}
Notice that for every $k\in \IN_0$, we have
\begin{align}
\label{rem: A(n-k) coefficients}
    \kappa_{\lambda,r}(m)  \frac{C_1}{C_1+ |\lambda|r^2 2^{2(k+m)}} &\leq  \kappa_{\lambda,r}(m+1).
\end{align}

\begin{lemma}[Iteration procedure]
\label{lem: iteration procedure}
    Let $\mu$ satisfy Assumption~\ref{Ass: Coefficients} and let $\omega_0$ be given by \eqref{eq: def omega}. For all $\theta \in [0,\pi-\omega_0)$ and all $\nu \in (0, d+2)$ there exist constants $C_1 , C_2>0$ such that for all balls $B= B(x_0,r)$, $\lambda \in \S_\theta$, $f\in \L_\sigma^2(\IR^d)$ and $F \in \L^2 (\IR^d ; \IC^{d \times d})$ the solution $u\in \W^{1,2}_\sigma(\IR^d)$ to the Stokes system \eqref{eq: Stokes system} satisfies
    \begin{align*}
        \sum\limits_{k=0}^\infty 2^{-\nu k}\int\limits_{B_k}|u|^2\,\d x &+ \frac{1}{2}\sum\limits_{k=0}^\infty 2^{-\nu k}\frac{2^{2k}r^2}{C_1 + |\lambda|r^2 2^{2k}}\int\limits_{B_k}|\nabla u|^2\,\d x \nonumber\\
        &\leq \frac{C_2}{|\lambda|} \sum\limits_{k=0}^\infty\sum\limits_{n=k}^\infty  \kappa_{\lambda,r}(n-k) 2^{-\nu k}\bigg(\frac{1}{|\lambda|}\int\limits_{B_{n+1}}|f|^2\, \d x +\int\limits_{B_{n+1}}|F|^2\,\d x\bigg).
    \end{align*}
    The constants $C_1 , C_2$ depend only on $\theta$, $\nu$, $d$ and $\mu_\bullet,\mu^\bullet$.
\end{lemma}

\begin{proof}
    For $m\in \IN_0$, $\lambda\in\S_\theta$ and $r>0$, denote by $\kappa_{\lambda,r}(m)$ the coefficients introduced in~\eqref{def: coefficients} with respect to the constant $C_1>0$ from Lemma~\ref{lem: iteration setup}. Take \eqref{eq: iteration equation} but with radius $2^mr$ instead of $r$, and multiply it with $\kappa_{\lambda,r}(m)$ to get
    \begin{align}
    \label{eq: iteration equation 1}
        &\sum\limits_{k=0}^\infty \kappa_{\lambda,r}(m)2^{-\nu k}\int\limits_{B_{k+m}}|u|^2\,\d x + \frac{1}{2}\sum\limits_{k=0}^\infty  \kappa_{\lambda,r}(m)2^{-\nu k}\frac{2^{2(k+m)}r^2}{C_1 + |\lambda|r^2 2^{2(k+m)}}\int\limits_{B_{k+m}}|\nabla u|^2\,\d x  \nonumber\\
        & \leq \sum\limits_{k=0}^\infty \kappa_{\lambda,r}(m)2^{-\nu k}\frac{C_1}{C_1 + |\lambda|r^2 2^{2(k+m)}}\int\limits_{B_{k+m+1}}|u|^2\,\d x \\
        &\quad+ \frac{C_2}{|\lambda|}\sum\limits_{k=0}^\infty   \kappa_{\lambda,r}(m)2^{-\nu k} \bigg(\frac{1}{|\lambda|}\int\limits_{B_{k+m+1}}|f|^2\, \d x +\int\limits_{B_{k+m+1}}|F|^2\,\d x\bigg).\notag
    \end{align}
    Starting with $m=0$, we estimate with~\eqref{rem: A(n-k) coefficients}
    \begin{align}
    \label{eq: iteration 0}
        &\sum\limits_{k=0}^\infty 2^{-\nu k}\int\limits_{B_k}|u|^2\,\d x + \frac{1}{2}\sum\limits_{k=0}^\infty 2^{-\nu k}\frac{2^{2k}r^2}{C_1 + |\lambda|r^2 2^{2k}}\int\limits_{B_k}|\nabla u|^2\,\d x  \nonumber\\
        &\leq \sum\limits_{k=0}^\infty \kappa_{\lambda,r}(0)2^{-\nu k}\frac{C_1}{C_1 + |\lambda|r^2 2^{2k}}\int\limits_{B_{k+1}}|u|^2\,\d x \notag\\
        &\quad+ \frac{C_2}{|\lambda|}\sum\limits_{k=0}^\infty  \kappa_{\lambda,r}(0)2^{-\nu k} \bigg(\frac{1}{|\lambda|}\int\limits_{B_{k+1}}|f|^2\, \d x +\int\limits_{B_{k+1}}|F|^2\,\d x\bigg)\\
        &\leq \sum\limits_{k=0}^\infty\kappa_{\lambda,r}(1)2^{-\nu k}\int\limits_{B_{k+1}}|u|^2\,\d x  + \frac{C_2}{|\lambda|}\sum\limits_{k=0}^\infty  \kappa_{\lambda,r}(0)2^{-\nu k} \bigg(\frac{1}{|\lambda|}\int\limits_{B_{k+1}}|f|^2\, \d x +\int\limits_{B_{k+1}}|F|^2\,\d x\bigg). \nonumber
    \end{align}
    We repeat this step by taking inequality~\eqref{eq: iteration equation 1} (but we abandon the term involving $\nabla u$) with $m=1$ instead of $m=0$, so that with~\eqref{rem: A(n-k) coefficients} it follows
    \begin{align}
    \label{eq: iteration 1}
         \sum\limits_{k=0}^\infty  \kappa_{\lambda,r}(1) 2^{-\nu k}\int\limits_{B_{k+1}}|u|^2\,\d x 
        &\leq \sum\limits_{k=0}^\infty   \kappa_{\lambda,r}(1) 2^{-\nu k}\frac{C_1}{C_1 + |\lambda|r^2 2^{2(k+1)}}\int\limits_{B_{k+2}}|u|^2\,\d x   \nonumber\\
        &\quad+\frac{C_2}{|\lambda|}\sum\limits_{k=0}^\infty  \kappa_{\lambda,r}(1)2^{-\nu k} \bigg(\frac{1}{|\lambda|}\int\limits_{B_{k+2}}|f|^2\, \d x +\int\limits_{B_{k+2}}|F|^2\,\d x\bigg)  \nonumber\\
        &\leq\sum\limits_{k=0}^\infty   \kappa_{\lambda,r}(2) 2^{-\nu k}\int\limits_{B_{k+2}}|u|^2\,\d x  \\
        &\quad+  \frac{C_2}{|\lambda|}\sum\limits_{k=0}^\infty  \kappa_{\lambda,r}(1)2^{-\nu k} \bigg(\frac{1}{|\lambda|}\int\limits_{B_{k+2}}|f|^2\, \d x +\int\limits_{B_{k+2}}|F|^2\,\d x\bigg). \nonumber
    \end{align}
    Plugging~\eqref{eq: iteration 1} into~\eqref{eq: iteration 0}, and repeating this procedure by increasing $m$ in each step, we will end up with
    \begin{align*}
        &\sum\limits_{k=0}^\infty 2^{-\nu k}\int\limits_{B_k}|u|^2\,\d x + \frac{1}{2}\sum\limits_{k=0}^\infty 2^{-\nu k}\frac{2^{2k}r^2}{C_1 + |\lambda|r^2 2^{2k}}\int\limits_{B_k}|\nabla u|^2\,\d x \\
        &\leq \frac{C_2}{|\lambda|}\sum\limits_{m=0}^\infty \sum\limits_{k=0}^\infty \kappa_{\lambda,r}(m)2^{-\nu k}\bigg(\frac{1}{|\lambda|}\int\limits_{B_{k+m+1}}|f|^2\, \d x +\int\limits_{B_{k+m+1}}|F|^2\,\d x\bigg)\\
        &=\frac{C_2}{|\lambda|} \sum\limits_{k=0}^\infty\sum\limits_{m=0}^\infty  \kappa_{\lambda,r}(m)2^{-\nu k}\bigg(\frac{1}{|\lambda|}\int\limits_{B_{k+m+1}}|f|^2\, \d x +\int\limits_{B_{k+m+1}}|F|^2\,\d x\bigg) \\
        &= \frac{C_2}{|\lambda|} \sum\limits_{k=0}^\infty\sum\limits_{n=k}^\infty  \kappa_{\lambda,r}(n-k) 2^{-\nu k}\bigg(\frac{1}{|\lambda|}\int\limits_{B_{n+1}}|f|^2\, \d x +\int\limits_{B_{n+1}}|F|^2\,\d x\bigg),
    \end{align*}
    where we set $n=k+m$ and used Fubini's theorem.
\end{proof}

We are now in a position to prove Theorem~\ref{thm: off-diag resolvent}.

\begin{proof}[Proof of Theorem~\ref{thm: off-diag resolvent}]
    By Lemma~\ref{lem: iteration procedure}, it suffices to estimate
    \begin{align}
    \label{eq: 187}
         &\sum\limits_{k=0}^\infty\sum\limits_{n=k}^\infty  \kappa_{\lambda,r}(n-k)2^{-\nu k}\bigg(\frac{1}{|\lambda|}\int\limits_{B_{n+1}}|f|^2\, \d x +\int\limits_{B_{n+1}}|F|^2\,\d x\bigg) \nonumber\\
         &\lesssim \frac{1}{|\lambda|} \max\Big\{1, (|\lambda|r^2)^{-(d+1)}\Big\}\sum\limits_{n=0}^\infty 2^{-\nu n}\bigg(\frac{1}{|\lambda|}\int\limits_{C_n(B)}|f|^2\, \d x +\int\limits_{C_n(B)}|F|^2\,\d x\bigg).
    \end{align}
    For this purpose, we define the functions
    \begin{align*}
        G(j) \coloneqq \kappa_{\lambda,r}(-j) 2^{-\nu j}\mathbf{1}_{\IN_0}(-j) \qquad (j\in\IZ),
    \end{align*}
    where we set $\kappa_{\lambda,r}(-j) = 0$ for $j> 0$, and
    \begin{align*}
        H(j) \coloneqq 2^{-\nu j}\bigg(\frac{1}{|\lambda|}\int\limits_{B_{j+1}}|f|^2\, \d x +\int\limits_{B_{j+1}}|F|^2\,\d x\bigg)\mathbf{1}_{\IN_0}(j) \qquad (j\in\IZ).
    \end{align*}
    Then, we can rewrite the left-hand side of \eqref{eq: 187} as a discrete convolution of the form
    \begin{align*}
         \sum\limits_{k=0}^\infty\sum\limits_{n=k}^\infty  \kappa_{\lambda,r}(n-k) 2^{-\nu k}\bigg(\frac{1}{|\lambda|}\int\limits_{B_{n+1}}|f|^2\, \d x +\int\limits_{B_{n+1}}|F|^2\,\d x\bigg) = \sum\limits_{k\in\IN_0} (G\ast H)(k).
    \end{align*}
    Applying the discrete Young inequality gives us the estimate
    \begin{align}
    \label{eq: B}
        \sum\limits_{k=0}^\infty\sum\limits_{n=k}^\infty  \kappa_{\lambda,r}(n-k) 2^{-\nu k}\bigg(\frac{1}{|\lambda|}\int\limits_{B_{n+1}}|f|^2\, \d x +\int\limits_{B_{n+1}}|F|^2\,\d x\bigg)\leq  \|G\|_{\ell^1(\IZ)}\|H\|_{\ell^1(\IZ)}.
    \end{align}
    For $\|H\|_{\ell^1(\IZ)}$, we calculate with Fubini's theorem
    \begin{align}
    \label{eq: G}
        \|H\|_{\ell^1(\IZ)} &= \sum\limits_{n=0}^\infty 2^{-\nu n}\bigg(\frac{1}{|\lambda|}\int\limits_{B_{n+1}}|f|^2\, \d x +\int\limits_{B_{n+1}}|F|^2\,\d x\bigg)\nonumber\\
        &=\sum\limits_{n=0}^\infty 2^{-\nu n}\sum\limits_{k=0}^{n+1}\bigg(\frac{1}{|\lambda|}\int\limits_{C_{k}(B)}|f|^2\, \d x +\int\limits_{C_{k}(B)}|F|^2\,\d x\bigg) \nonumber\\
        &= \sum\limits_{k=0}^\infty\sum\limits_{n=\max\{0,k-1\}}^{\infty} 2^{-\nu n}\bigg(\frac{1}{|\lambda|}\int\limits_{C_{k}(B)}|f|^2\, \d x +\int\limits_{C_{k}(B)}|F|^2\,\d x\bigg) \nonumber\\
        &\lesssim\sum\limits_{k=0}^\infty  2^{-\nu k}\bigg(\frac{1}{|\lambda|}\int\limits_{C_{k}(B)}|f|^2\, \d x +\int\limits_{C_{k}(B)}|F|^2\,\d x\bigg).
    \end{align}
    For $\|G\|_{\ell^1(\IZ)}$, we split the sum into the two parts
    \begin{align}
    \label{eq: F}
        \|G\|_{\ell^1(\IZ)} =\sum\limits_{n=0}^\infty  \kappa_{\lambda,r}(n)2^{\nu n}  =  \sum\limits_{n=0}^d   \kappa_{\lambda,r}(n) 2^{\nu n} +\sum\limits_{n=d+1}^\infty  \kappa_{\lambda,r}(n) 2^{\nu n}
    \end{align}
    and estimate both sums separately. For the finite sum, we use
    \begin{align*}
        \kappa_{\lambda,r}(n)= \prod\limits_{s= 0}^{n-1} \frac{C_1}{C_1+ |\lambda|r^22^{2s}} \leq 1
    \end{align*}
    to deduce the bound
    \begin{align*}
        \sum\limits_{n=0}^d   \kappa_{\lambda,r}(n)2^{\nu n}  \leq  \sum\limits_{n=0}^d   2^{\nu n} \lesssim 1.
    \end{align*}
    For the infinite sum, we use
    \begin{align*}
        \kappa_{\lambda,r}(n) =\prod\limits_{s= 0}^{n-1} \frac{C_1}{C_1+ |\lambda|r^22^{2s}}
        \leq \prod\limits_{s= n-1-d}^{n-1} \frac{C_1}{|\lambda|r^22^{2s}}
        \lesssim 2^{-2d n}(|\lambda|r^2)^{-(d+1)},
    \end{align*}
    with an implicit constant depending only on $C_1, \nu, d$, such that
    \begin{align*}
        \sum\limits_{n=d+1}^\infty   \kappa_{\lambda,r}(n) 2^{\nu n} \lesssim\sum\limits_{n=d+1}^\infty   2^{(\nu-2d) n}(|\lambda|r^2)^{-(d+1)} \lesssim(|\lambda|r^2)^{-(d+1)}
    \end{align*}
    since $\nu<d+2\leq 2d$. Both estimates combined in \eqref{eq: F} yield the estimate
    \begin{align*}
        \|G\|_{\ell^1(\IZ)} \lesssim \max\Big\{1, (|\lambda|r^2)^{-(d+1)}\Big\}.
    \end{align*}
    Plugging this estimate and \eqref{eq: G} into \eqref{eq: B} yields
    \begin{align*}
        &\sum\limits_{k=0}^\infty\sum\limits_{n=k}^\infty \kappa_{\lambda,r}(n-k) 2^{-\nu k}\bigg(\frac{1}{|\lambda|}\int\limits_{B_{n+1}}|f|^2\, \d x +\int\limits_{B_{n+1}}|F|^2\,\d x\bigg)\\
        &\lesssim \max\Big\{1, (|\lambda|r^2)^{-(d+1)}\Big\}\sum\limits_{k=0}^\infty 2^{-\nu k}\bigg(\frac{1}{|\lambda|}\int\limits_{C_k(B)}|f|^2\, \d x +\int\limits_{C_k(B)}|F|^2\,\d x\bigg),
    \end{align*}
    which proves \eqref{eq: 187} and, thus, the claim. 
\end{proof}

Next, we improve the decay of the prefactor in \eqref{eq: off-diag resolvent} with respect to $|\lambda|$, which is what we need in Section~\ref{sec: Hinfty calc}. If the data vanish on $2B$, we can gain this decay by first running a local Caccioppoli inequality on $B$ and then feeding \eqref{eq: off-diag resolvent} into the resulting pressure term.

\begin{lemma}[Local Caccioppoli inequality]
\label{lem: local Caccioppoli}
    Let $\mu$ satisfy Assumption~\ref{Ass: Coefficients} and let $\omega_0$ be given by \eqref{eq: def omega}. For all $\theta \in [0,\pi-\omega_0)$ there exists a constant $C>0$ such that for all balls $B=B(x_0,r)$, $c\in\IC$, $\lambda \in \S_\theta$, $f\in \L_\sigma^2(\IR^d)$ and $F\in\L^2(\IR^d;\IC^{d\times d})$ with $\supp(f),\supp (F) \subset \IR^d\setminus 2B$ the solutions $u\in \W^{1,2}_\sigma(\IR^d)$ and $\phi\in\Lloc^2(\IR^d)$ to \eqref{eq: Stokes system} satisfy
    \begin{align*}
        |\lambda|\int\limits_{B} |u|^2\,\d x + \int\limits_{B} |\nabla u|^2\,\d x \leq \frac{C}{r^2}\int\limits_{C_1(B)} |u|^2\,\d x + \frac{C}{|\lambda|r^2}\int\limits_{C_1(B)} |\phi - c|^2\,\d x.
    \end{align*}
    The constant $C$ depends only on $\theta,d$ and $\mu_\bullet,\mu^\bullet$.
\end{lemma}

\begin{proof}
    Let $\eta\in \C_c^\infty(\IR^d)$ be as in Lemma~\ref{lem: Bogovskii lemma} with respect to the ball $B$. Then, by \cite[Lem.~5.1]{Tolksdorf-Caccioppoli} we have
    \begin{align*}
        &|\lambda| \int\limits_{2B}|u|^2\eta^2 \,\d x + \int\limits_{2B}|\nabla(u\eta)|^2\,\d x \\
        &\leq \frac{C}{r^2}\int\limits_{C_1(B)}|u|^2\,\d x\quad+ \frac{C}{r}\bigg(\int\limits_{C_1(B)}|\phi-c|^2\,\d x \bigg)^\frac{1}{2}\bigg(\int\limits_{C_1(B)} |u\eta|^2\,\d x\bigg)^\frac{1}{2},
    \end{align*}
    where $C>0$ depends only on $ d, \theta$ and $\mu_\bullet,\mu^\bullet$. Finally, by Young's inequality, we have
    \begin{align*}
        \frac{C}{r}\bigg(\int\limits_{C_1(B)}|\phi-c|^2\,\d x \bigg)^\frac{1}{2}\bigg(\int\limits_{C_1(B)} |u\eta|^2\,\d x\bigg)^\frac{1}{2}
        &\leq  \frac{C^2}{2|\lambda|r^2}\int\limits_{C_1(B)}|\phi-c|^2\,\d x + \frac{|\lambda|}{2}\int\limits_{C_1(B)} |u\eta|^2\,\d x
    \end{align*}
    such that
    \begin{align*}
        |\lambda| \int\limits_{2B}|u|^2\eta^2 \,\d x + \int\limits_{2B}|\nabla(u\eta)|^2\,\d x &\leq \frac{C}{r^2}\int\limits_{C_1(B)}|u|^2\,\d x + \frac{C^2}{2|\lambda|r^2}\int\limits_{C_1(B)}|\phi-c|^2\,\d x\\
        &\quad + \frac{|\lambda|}{2}\int\limits_{C_1(B)} |u\eta|^2\,\d x.
    \end{align*}
    Absorbing the $u\eta$ term onto the left-hand side yields the claim.
\end{proof}

\begin{proposition}[Improved $\L^2$-decay estimates]
\label{prop: imp off-diag est}
Let $\mu$ satisfy Assumption~\ref{Ass: Coefficients} and let $\omega_0$ be given by \eqref{eq: def omega}. For all $\theta \in [0,\pi-\omega_0)$ and for all $\nu \in (0, d+2)$ there exist $C_1,C_2>0$ such that for all balls $B=B(x_0,r)$, $\lambda \in \S_\theta$, $f\in \L_\sigma^2(\IR^d)$ and $F\in\L^2(\IR^d;\IC^{d\times d})$ with $\supp(f),\supp (F) \subset \IR^d\setminus 2B$ the solution $u\in \W^{1,2}_\sigma(\IR^d)$ to \eqref{eq: Stokes system} satisfies
\begin{align*}
    \int\limits_{B}|u|^2\,\d x &+  \frac{r^2}{C_1+|\lambda|r^2}\int\limits_{B}|\nabla u|^2\,\d x \\
    &\leq \frac{C_2}{|\lambda|^2r^2} \max\Big\{1, (|\lambda|r^2)^{-(d+1)}\Big\}\sum\limits_{n=0}^\infty 2^{-\nu n}\bigg(\frac{1}{|\lambda|}\int\limits_{C_n(B)}|f|^2\, \d x +\int\limits_{C_n(B)}|F|^2\,\d x\bigg).
\end{align*}
The constants $C_1,C_2$ depend only on $\theta,\nu,d$ and $\mu_\bullet,\mu^\bullet$.
\end{proposition}

\begin{proof}
    Since $\supp (f),\supp (F)\subset \IR^d\setminus 2B$, we can apply Lemma~\ref{lem: local Caccioppoli} to obtain
    \begin{align*}
        |\lambda|\int\limits_{B}|u|^2\,\d x +  \int\limits_{B}|\nabla u|^2\,\d x\leq \frac{C}{r^2}\int\limits_{C_1(B)}|u|^2\,\d x + \frac{C}{|\lambda|r^2} \int\limits_{C_1(B)}|\phi-\phi_{C_1(B)}|^2\,\d x
    \end{align*}
    for some generic constant $C>0$. Add the term
    \begin{align*}
        \frac{C_1}{r^2}\int\limits_{B}|u|^2\,\d x
    \end{align*}
    on both sides of the inequality, where $C_1>0$ is the constant from Theorem~\ref{thm: off-diag resolvent}, to get
    \begin{align*}
        &\Big(|\lambda|+\frac{C_1}{r^2}\Big)\int\limits_{B}|u|^2\,\d x +  \int\limits_{B}|\nabla u|^2\,\d x \\
        &\leq \frac{C+C_1}{r^2}\int\limits_{2B}|u|^2\,\d x + \frac{C}{|\lambda|r^2} \int\limits_{C_1(B)}|\phi-\phi_{C_1(B)}|^2\,\d x.
    \end{align*}
    Dividing both sides by $(|\lambda|+\frac{C_1}{r^2})$ yields
    \begin{align*}
        &\int\limits_{B}|u|^2\,\d x +  \frac{r^2}{C_1+|\lambda|r^2}\int\limits_{B}|\nabla u|^2\,\d x \\
        &\leq \frac{C+C_1}{C_1+|\lambda|r^2}\int\limits_{2B}|u|^2\,\d x + \frac{C}{|\lambda|r^2} \frac{r^2}{C_1+|\lambda|r^2}\int\limits_{C_1(B)}|\phi-\phi_{C_1(B)}|^2\,\d x\\
        &\leq \frac{C+C_1}{|\lambda|r^2}\bigg(\int\limits_{2B}|u|^2\,\d x +  \frac{r^2}{C_1+|\lambda|r^2}\int\limits_{C_1(B)}|\phi-\phi_{C_1(B)}|^2\,\d x\bigg).
    \end{align*}
    Applying the pressure estimate from \eqref{eq: pressure estimate} yields
    \begin{align*}
        &\int\limits_{B}|u|^2\,\d x +  \frac{r^2}{C_1+|\lambda|r^2}\int\limits_{B}|\nabla u|^2\,\d x \\
        &\leq \frac{C+C_1}{|\lambda|r^2}\bigg(\int\limits_{2B}|u|^2\,\d x +  \frac{r^2}{C_1+|\lambda|r^2}\sum\limits_{n=0}^\infty 2^{-\nu n} \bigg(\int\limits_{B_n} |\nabla u|^2\,\d x + \int\limits_{B_n} |F|^2\,\d x\bigg) \bigg)\\
        &\leq \frac{C+C_1}{|\lambda|r^2}\bigg( \int\limits_{2B}|u|^2\,\d x +  \sum\limits_{n=0}^\infty2^{-\nu n} \frac{2^{2n}r^2}{C_1+|\lambda|r^22^{2n}}\int\limits_{B_n} |\nabla u|^2\,\d x\bigg)\\
        &\quad + \frac{C+C_1}{|\lambda|^2r^2}\sum\limits_{n=0}^\infty 2^{-\nu n}  \int\limits_{C_n(B)} |F|^2\,\d x,
    \end{align*}
    where we used \eqref{eq: G} and the monotonicity of $\frac{2^{2n}r^2}{C_1+|\lambda|r^22^{2n}}$ in $n\in\IN_0$ in the last step. Finally, using Theorem~\ref{thm: off-diag resolvent} yields
    \begin{align*}
        &\int\limits_{B}|u|^2\,\d x +  \frac{r^2}{C_1+|\lambda|r^2}\int\limits_{B}|\nabla u|^2\,\d x \\
        &\leq \frac{C_2}{|\lambda|^2r^2} \max\Big\{1, (|\lambda|r^2)^{-(d+1)}\Big\}\sum\limits_{n=0}^\infty 2^{-\nu n}\bigg(\frac{1}{|\lambda|}\int\limits_{C_n(B)}|f|^2\, \d x +\int\limits_{C_n(B)}|F|^2\,\d x\bigg)
    \end{align*}
    for some new generic constant $C_2>0$.
\end{proof}

We conclude this section by presenting the proof of Theorem~\ref{thm: L2-Lp off diag stokes resolv}. It is based on the previous decay estimates combined with Sobolev's embedding theorem.

\begin{proof}[Proof of Theorem~\ref{thm: L2-Lp off diag stokes resolv}]
    For simplicity, we present only the proof of the first assertion since the second follows by the same line of reasoning with  Proposition~\ref{prop: imp off-diag est}. Moreover, we assume $d\geq 3$. The case $d=2$ follows analogously, using the Sobolev embedding $\W^{1,2}\hookrightarrow \L^p$ for $2\leq p<\infty$. 
    
    We start by applying Theorem~\ref{thm: off-diag resolvent} to get the estimate
    \begin{align}
    \label{eq: 287}
         &\int\limits_{B}|u|^2\,\d x +  \frac{r^2}{C_1+ |\lambda| r^2 }\int\limits_{B}|\nabla u|^2\,\d x \notag\\
        &\lesssim \frac{1}{|\lambda|} \max\Big\{1, (|\lambda|r^2)^{-(d+1)}\Big\}\sum\limits_{n=0}^\infty 2^{-\nu n}\bigg(\frac{1}{|\lambda|}\int\limits_{C_n(B)}|f|^2\, \d x +\int\limits_{C_n(B)}|F|^2\,\d x\bigg).
    \end{align}
    We estimate the left-hand side by
    \begin{align*}
        \int\limits_{B}|u|^2\,\d x + \frac{r^2}{C_1+ |\lambda| r^2 } \int\limits_{B}|\nabla u|^2\,\d x &\gtrsim \frac{1}{\max\{1, |\lambda| r^2\} }\bigg(\int\limits_{B}|u|^2\,\d x +r^2\int\limits_{B}|\nabla u|^2\,\d x\bigg),
    \end{align*}
    where the implicit constant depends only on $\theta$, $\nu$, $d$ and $\mu_\bullet,\mu^\bullet$. Using Sobolev's embedding on the right-hand side, we obtain the estimate
    \begin{align*}
        \int\limits_{B}|u|^2\,\d x + \frac{r^2}{C_1+ |\lambda| r^2 } \int\limits_{B}|\nabla u|^2\,\d x&\gtrsim \frac{r^{d(1-\frac{2}{p})}}{\max\{1, |\lambda| r^2\} }\bigg(\int\limits_{B}|u|^{p}\,\d x\bigg)^\frac{2}{p}.
    \end{align*}
    Plugging this estimate into \eqref{eq: 287} leads to the bound
    \begin{align*}
        &\|u\|_{\L^{p}(B)}^2 \\
        &\lesssim \frac{1}{|\lambda|r^{d(1-\frac{2}{p})}}\max\Big\{|\lambda|r^2, (|\lambda|r^2)^{-(d+1)}\Big\}\sum\limits_{n=0}^\infty 2^{-\nu n}\bigg(\frac{1}{|\lambda|}\int\limits_{C_{n}(B)}|f|^2\, \d x +\int\limits_{C_{n}(B)}|F|^2\,\d x\bigg),
    \end{align*}
    which proves the claim.
\end{proof}

\section{Abstract principles}
\label{sec: abstract machinery}

The purpose of this section is to develop an abstract toolbox connecting uniform $\L^p$-bounds, $\L^2$-$\L^p$ bounds (or hypercontractivity), and $\L^2$-$\L^p$ decay estimates for families of operators acting on $\L^2$-spaces similar to the one for elliptic operators as discussed in \cite[Chap.~3]{Auscher-Lp}. In particular, no differential structure is used, and the exponent $2$ is not essential; everything below carries over to $\L^q$-based spaces for $q\leq p$. We fix $q=2$ to spare notation, since this is the form in which the decay estimates of Section~\ref{sec: off-diag est} are established and applied in Sections~\ref{sec: Lp-theory sg} and~\ref{sec: gradient est}.
We begin by generalizing the concept of $\L^2$-$\L^p$ decay estimates. Here, the specific decay estimates of the Stokes resolvent established in Section~\ref{sec: off-diag est} serve as our motivating example. 
We then demonstrate how these generalized estimates, in turn, imply uniform $\L^p$-boundedness.
Finally, we introduce the concept of $\L^2$-$\L^p$ bounds and establish their interpolation relations with decay estimates.

\subsection{$\L^2$-$\L^p$ decay estimates}

We start with the following definition.
\begin{definition}
\label{def: decay est L2-Lp}
    Let $\Omega\subset \IC\setminus\{0\}$ be open, $V,W\subset \L^2$ be closed subspaces and $(T(z))_{z\in\Omega}\subset \mathcal{L}(V,W)$. We say that the family $(T(z))_{z\in\Omega}$ satisfies $\L^2$-$\L^p$ decay estimates of order $\nu\geq 0$ for $p\in[2,\infty]$ if there exist $M,N\geq 0$ such that
    \begin{align*}
        \|T(z)f\|_{\L^p(B)} 
        \lesssim r^{-d(\frac{1}{2}-\frac{1}{p})}\max\Big\{\Big(\frac{r^2}{|z|}\Big)^{M}, \Big(\frac{r^2}{|z|}\Big)^{-N}\Big\}\bigg(\sum\limits_{n=0}^\infty 2^{-\nu n} \|f\|^2_{\L^2(C_{n}(B))}\bigg)^\frac{1}{2}
    \end{align*}
    holds for all $z\in\Omega$, all balls $B=B(x,r)\subset \IR^d$ and all $f\in V$.
\end{definition}

\begin{remark}
    \begin{enumerate}
        \item Compared to the decay estimates of the Stokes resolvent of Theorem~\ref{thm: off-diag resolvent}, we have formulated the definition here in terms of the operator family $((1+zA)^{-1})_{z\in \S_{\pi-\omega_0}}$ rather than $(\lambda(\lambda+A)^{-1})_{\lambda\in \S_{\pi-\omega_0}}$. This scaling also aligns with the semigroup generated by $-A$ (see Proposition~\ref{prop: off-diag semigroup} below).

        \item In applications, we have $V,W \in\{ \L^2(\IR^d;\IC^{d\times d}), \L^2_\sigma(\IR^d)\}$ in mind.
    \end{enumerate}
\end{remark}

The next result shows that $\L^2$-$\L^p$ decay estimates are a stronger notion than uniform $\L^p$-boundedness provided the decay rate is fast enough.

\begin{proposition}
\label{prop: L2-Lp off imply Lp bdd}
    Let $\Omega\subset \IC\setminus\{0\}$ be open, $V,W\subset \L^2$ be closed subspaces and $(T(z))_{z\in\Omega}\subset \cL(V,W)$ satisfy $\L^2$-$\L^p$ decay estimates of order $\nu>d$ for $p\in[2,\infty]$. 
    Then, there exists $C>0$ such that for all $z\in\Omega$ and $f\in V\cap \L^p$ we have
    \begin{align*}
        \|T(z)f\|_{\L^p} \leq C\|f\|_{\L^p}.
    \end{align*}
\end{proposition}

\begin{proof}
    Fix $f\in V\cap \L^p$ and $z\in \Omega$. We cover $\IR^d$ by the balls $B^x=B(x|z|^{\frac{1}{2}}, (d|z|)^{\frac{1}{2}})$ for $x\in\IZ^d$. Using the $\L^2$-$\L^p$ decay estimates of $T(z)$ yields
    \begin{align*}
        \|T(z)f\|^p_{\L^p} &\leq  \sum\limits_{x\in\IZ^d}\|T(z)f\|^p_{\L^p(B^x)}\\
        &\lesssim  \sum\limits_{x\in\IZ^d}\bigg( |z|^{-\frac{d}{2}(1-\frac{2}{p})}\sum\limits_{n=0}^\infty 2^{-\nu n} \|f\|^2_{\L^2(C_{n}(B^x))}\bigg)^\frac{p}{2}.
    \intertext{Since $p> 2$, we can apply Hölder's inequality to get}
        &\lesssim  \sum\limits_{x\in\IZ^d}\bigg( |z|^{-\frac{d}{2}(1-\frac{2}{p})}\sum\limits_{n=0}^\infty 2^{-\nu n} |C_{n}(B^x)|^{1-\frac{2}{p}}\|f\|^2_{\L^p(C_{n}(B^x))}\bigg)^\frac{p}{2}\\
        &\simeq  \sum\limits_{x\in\IZ^d}\bigg( \sum\limits_{n=0}^\infty 2^{(d(1-\frac{2}{p})-\nu) n} \|f\|^2_{\L^p(C_{n}(B^x))}\bigg)^\frac{p}{2}.
    \intertext{Next, applying Hölder's inequality on the inner sum yields}
        &\lesssim \sum\limits_{x\in\IZ^d} \sum\limits_{n=0}^\infty 2^{\frac{p}{2}(d(1-\frac{2}{p})-\tilde \nu) n} \|f\|^p_{\L^p(C_{n}(B^x))}
    \intertext{for some $d<\tilde \nu<\nu$. Finally, Fubini's theorem followed by a counting argument implies}
        &= \sum\limits_{n=0}^\infty 2^{\frac{p}{2}(d -\tilde \nu) n - dn} \int\limits_{\IR^d}\sum\limits_{x\in\IZ^d}\mathbf{1}_{C_{n}(B^x)}(y) |f(y)|^p \, \d y\\
        &\lesssim \sum\limits_{n=0}^\infty 2^{\frac{p}{2}(d -\tilde \nu) n - dn} \int\limits_{\IR^d}2^{dn} |f(y)|^p \, \d y\\
        &\lesssim \|f\|_{\L^p}^p\sum\limits_{n=0}^\infty 2^{\frac{p}{2}(d -\tilde \nu)  n} .
    \end{align*}
    Since $d<\tilde \nu<\nu$, the value of the remaining sum is finite. Hence, we proved
    \begin{align*}
         \|T(z)f\|_{\L^p}\lesssim \|f\|_{\L^p}
    \end{align*}
    for all $f\in V\cap \L^p$ and $z\in \Omega$.
\end{proof}

\subsection{Hypercontractivity and interpolation principles}
\label{subsec: interpolation principles}

In this subsection, we introduce the notion of $\L^2$-$\L^p$ bounds, which we call hypercontractivity, following~\cite{Auscher-Lp}. This property can be regarded as a special case of $\L^2$-$\L^p$ decay estimates corresponding to the order $\nu=0$ (see Remark~\ref{rem: hypercontractivity and decay estimates}), making it a structurally weaker condition. Nevertheless, we will demonstrate how hypercontractivity interacts with decay estimates via interpolation.

\begin{definition}
\label{def: L2-Lp bounds}
    Let $\Omega\subset \IC\setminus\{0\}$ be open and $V,W\subset \L^2$ be closed subspaces. For $2\leq p\leq \infty$, we say that the family $(T(z))_{z\in\Omega}\subset \mathcal{L}(V,W)$ satisfies $\L^2$-$\L^p$ bounds if
    \begin{align*}
        \|T(z)f\|_{\L^p}\lesssim |z|^{-\frac{d}{2}(\frac{1}{2}-\frac{1}{p})}\|f\|_{\L^2}
    \end{align*}
    for every $z\in \Omega$ and $f\in V$.
\end{definition}

\begin{remark}
\label{rem: hypercontractivity and decay estimates}
In the situation of Definition~\ref{def: L2-Lp bounds}, we have
    \begin{align*}
        \|T(z)f\|_{\L^p(B)}\lesssim |z|^{-\frac{d}{2}(\frac{1}{2}-\frac{1}{p})}\|f\|_{\L^2} = r^{-d(\frac{1}{2}-\frac{1}{p})}\Big(\frac{r^2}{|z|}\Big)^{\frac{d}{2}(\frac{1}{2}-\frac{1}{p})}\bigg(\sum\limits_{n=0}^\infty \|f\|^2_{\L^2(C_n(B))}\bigg)^\frac{1}{2}
    \end{align*}
    for every $z\in \Omega$, balls $B=B(x,r)\subset \IR^d$ and $f\in V$. Hence, $\L^2$-$\L^p$ bounds imply $\L^2$-$\L^p$ decay estimates of order $\nu =0$ with $M=\frac{d}{2}(\frac{1}{2}-\frac{1}{p})$ and $N=0$.
\end{remark}

To establish an interpolation relation between decay estimates and hypercontractivity, we define the following auxiliary function spaces.

\begin{definition}
    For $B=B(x,r)\subset \IR^d$, $0\leq \theta\leq 1$ and $\nu\geq 0$ we define
    \begin{align*}
        X^\nu_\theta(B) \coloneqq \Big\{ f\in\Llocs^2(\IR^d) :  \|f\|_{X^\nu_\theta(B) }\coloneqq \bigg(\sum\limits_{n=0}^\infty 2^{-\nu(1-\theta) n}\|f\|^2_{\L^2(C_n(B))}\bigg)^\frac{1}{2}<\infty\Big\} 
    \end{align*}
    and
    \begin{align*}
        Y^\nu_\theta(B)  \coloneqq \Big\{ F\in\Lloc^2(\IR^d;\IC^{d\times d}) :\|F\|_{Y^\nu_\theta(B) }\coloneqq \bigg(\sum\limits_{n=0}^\infty 2^{-\nu(1-\theta) n}\|F\|^2_{\L^2(C_n(B))}\bigg)^\frac{1}{2}<\infty\Big\}.
    \end{align*}
\end{definition}
This notation allows us to formulate $\L^2$-$\L^p$ decay estimates from Definition~\ref{def: decay est L2-Lp} and $\L^2$-$\L^p$ bounds from Definition~\ref{def: L2-Lp bounds} on a common scale of spaces.

\begin{lemma}
\label{lem: facts of X and Y}
    For $B=B(x,r)\subset \IR^d$, $0\leq \theta\leq 1$ and $\nu\geq 0$, the spaces $X^\nu_\theta(B) $ and $Y^\nu_\theta(B) $ are Banach spaces. Moreover, we have the continuous and dense embeddings
    \begin{align*}
        \L^2_\sigma(\IR^d) \hookrightarrow X^\nu_\theta(B)  \quad \text{and} \quad \L^2(\IR^d;\IC^{d\times d}) \hookrightarrow Y^\nu_\theta(B) .
    \end{align*}
\end{lemma}

\begin{proof}
    First, observe that the norms $\|\cdot \|_{X_\theta^\nu(B) }$ and $\|\cdot \|_{Y_\theta^\nu(B) }$ are equivalent to the weighted $\L^2$-norm $\|\cdot (1+|\cdot-x|/r)^{-\frac{\nu(1-\theta)}{2}}\|_{\L^2}$. Thus, the completeness of $X_\theta^\nu$ and $Y_\theta^\nu$ follows from the completeness of weighted $\L^2$-spaces. Moreover, we see that
    \begin{align*}
        \|f \|_{X_\theta^\nu(B) }\leq  \|f\|_{\L^2} \quad \text{and}\quad \|F \|_{Y_\theta^\nu(B) }\leq \|F\|_{\L^2}.
    \end{align*}
    Thus, the canonical embeddings $\L^2_\sigma(\IR^d) \hookrightarrow X^\nu_\theta(B)  $ and $\L^2(\IR^d;\IC^{d\times d}) \hookrightarrow Y^\nu_\theta(B)  $ are continuous. Finally, for $Y_\theta^\nu(B) $, fix $F\in Y_\theta^\nu(B) $ and define $F_k\coloneqq \mathbf{1}_{2^kB} F\in \L^2(\IR^d;\IC^{d\times d})$. Then, we have
    \begin{align*}
        \|F-F_k\|_{Y_\theta^\nu(B) } &= \bigg(\sum\limits_{n=0}^\infty 2^{-\nu(1-\theta) n}\|\mathbf{1}_{\IR^d\setminus 2^kB}F\|^2_{\L^2(C_n(B))}\bigg)^\frac{1}{2}\\
        &=\bigg(\sum\limits_{n=k+1}^\infty 2^{-\nu(1-\theta) n}\|F\|^2_{\L^2(C_n(B))}\bigg)^\frac{1}{2}\to 0
    \end{align*}
    as $k\to \infty$ since $F\in Y_\theta^\nu(B) $. For $f\in X_\theta^\nu(B) $ we define $f_k \coloneqq f_{2^{k+1}r} \in\L^2_\sigma(\IR^d)$, where 
    \begin{align*}
        f_{2^{k+1}r} =  f_{2(2^{k}r)}  = \eta f - \cB_{C_1(2^kB)}(\nabla \eta \cdot f)
    \end{align*}
    is the divergence-free truncation of $f$ on the ball $2^{k}B$ from Lemma~\ref{lem: Bogovskii lemma} with $\eta \in \C_c^\infty(2^{k+1}B)$ satisfying $\eta = 1$ on $2^kB$ and $\|\eta\|_{\L^\infty} + 2^kr\|\nabla \eta\|_{\L^\infty}\lesssim 1$. Using \eqref{ineq: Bogovskii} yields the bound
    \begin{align*}
          \|f-f_k\|_{X_\theta^\nu(B)  } &=  \bigg(\sum\limits_{n=k}^\infty 2^{-\nu(1-\theta) n}\|(1-\eta)f+\cB_{C_1(2^kB)}(\nabla \eta \cdot f) \|^2_{\L^2(C_n(B))} \bigg)^\frac{1}{2} \\
          &\lesssim \bigg(\sum\limits_{n=k}^\infty 2^{-\nu(1-\theta) n}\|f \|^2_{\L^2(C_n(B))}\bigg)^\frac{1}{2}\to 0
    \end{align*}
    as $k\to \infty$ since $f\in X_\theta^\nu(B) $. This finishes the proof.
\end{proof}

Our goal is to show that $(X^\nu_\theta(B) )_{0\leq \theta\leq 1}$ and $(Y^\nu_\theta(B) )_{0\leq \theta\leq 1}$ form real interpolation scales, thereby enabling us to interpolate decay estimates and hypercontractivity. 
To this end, we provide a brief overview of real interpolation theory to fix the relevant notation. For a more comprehensive treatment, we refer the reader to the monograph \cite{Lunardi}.

Let $(X_0,X_1)$ be an interpolation couple, that is, a pair of Banach spaces $X_0$ and $X_1$ that are continuously embedded into a Hausdorff topological vector space $Z$. Then we define the Banach space $X_0+X_1 = \{x_0+x_1 \in Z: x_0\in X_0, x_1\in X_1\}$  equipped with the norm
\begin{align*}
    \|x\|_{X_0+X_1} \coloneqq \inf\limits_{\substack{x_0\in X_0, x_1\in X_1\\ x=x_0+x_1}} \|x_0\|_{X_0} + \|x_1\|_{X_1}.
\end{align*}
For every $x\in X_0 + X_1$ and $t>0$ we define the $K$-functional
\begin{align*}
    K(x,t,X_0,X_1) \coloneqq \inf\limits_{\substack{x_0\in X_0, x_1\in X_1\\ x=x_0+x_1}} \|x_0\|_{X_0} + t\|x_1\|_{X_1}.
\end{align*}
Finally, for $0<\theta<1$ and $1\leq q< \infty$, we define the real interpolation space $(X_0,X_1)_{\theta,q}$ as the set of all $x\in X_0+X_1$ such that
\begin{align*}
    \|x\|_{(X_0,X_1)_{\theta,q}} \coloneqq \bigg(\int\limits_0^\infty t^{-\theta q}  \big(K(x,t,X_0,X_1)\big)^q \, \frac{\d t}{t}\bigg)^\frac{1}{q}<\infty.
\end{align*}
Moreover, we define $(X_0,X_1)_{0,q} \coloneqq X_0$ and $(X_0,X_1)_{1,q} \coloneqq X_1$. It is known that $(X_0,X_1)_{\theta,q}$ is a Banach space.

\begin{lemma}
\label{lem: interpolation result}
    Let $B=B(x,r)\subset \IR^d$, $0\leq \theta\leq 1$ and $\nu\geq 0$. Then, we have
    \begin{align*}
        X^\nu_\theta (B) = (X^\nu_0(B) ,X^\nu_1(B) )_{\theta,2}  \quad \text{and} \quad Y^\nu_\theta(B)  = (Y^\nu_0(B) ,Y^\nu_1(B) )_{\theta,2} 
    \end{align*}
    with equivalence of norms. The implicit constants are independent of the underlying ball.
\end{lemma}

\begin{remark}
    One might be tempted to prove the result first for the $Y $-spaces and deduce the corresponding result for the $X$-spaces by a retraction-coretraction argument using the Leray projection $\IP$. However, as we saw in the proof of Lemma~\ref{lem: facts of X and Y}, the spaces are equivalent to weighted $\L^2$-spaces with weight $(1+|\cdot-x|/r)^{-\nu(1-\theta)}$. Because $\IP$ is a Calderón--Zygmund operator, its boundedness on the weighted $\L^2$-space above is only guaranteed if the weight belongs to the Muckenhoupt class $A_2$, which is only true if $-d<\nu(1-\theta)<d$, see for example \cite[Chap.~7]{Grafakos}. Consequently, if $\nu(1-\theta) > d$, the projection $\IP$ might fail to act as a bounded operator on these spaces because it does not preserve the required decay rate.
\end{remark}

\begin{proof}[Proof of Lemma~\ref{lem: interpolation result}]
    We only provide a proof for the $X$-spaces since the analogous result for $Y$-spaces follows by the same reasoning or classical interpolation theorems for weighted $\L^2$-spaces. Moreover, the claim holds in the endpoints $\theta=0,1$ by definition; thus we assume $\theta\in(0,1)$. Furthermore, for $\nu =0$ the spaces coincide with $\L^2_\sigma(\IR^d)$ and the claim is trivial; thus we assume $\nu>0$. 
    Now, we use the so-called $L$-method for real interpolation spaces, which is equivalent to the above construction, see \cite[Chap.~1]{Triebel}. In particular, we choose $p_0 = p_1 = 2$, the parameter $\theta $ and $q=1$ in \cite[Sec.~1.4.2]{Triebel} to get the equivalent description of the interpolation norm
    \begin{align}
    \label{eq: int 0}
        \|f\|^2_{(X^\nu_0(B) ,X^\nu_1(B) )_{\theta,2}} &\simeq  \int\limits_{0}^\infty \tau^{-2\theta}\inf\limits_{f=f_0+f_1} \sum\limits_{n=0}^\infty 2^{-\nu n}\|f_0\|^2_{\L^2(C_n(B))}+ \tau^2\|f_1\|^2_{\L^2(C_n(B))}\,\frac{\d \tau}{\tau}.
    \end{align}
    From here, we divide the proof into two steps.\\

    \noindent \textbf{Step 1: $X^\nu_\theta(B)  \subset (X^\nu_0(B) ,X^\nu_1(B) )_{\theta,2} $}.
    We split the integral in \eqref{eq: int 0} into
    \begin{align}
    \label{eq: split int}
        &\int\limits_{0}^1 \tau^{-2\theta}\inf\limits_{f=f_0+f_1} \sum\limits_{n=0}^\infty 2^{-\nu n}\|f_0\|^2_{\L^2(C_n(B))}+ \tau ^2\|f_1\|^2_{\L^2(C_n(B))}\,\frac{\d \tau}{\tau} \nonumber\\
        &+ \int\limits_{1}^\infty \tau^{-2\theta}\inf\limits_{f=f_0+f_1} \sum\limits_{n=0}^\infty 2^{-\nu n}\|f_0\|^2_{\L^2(C_n(B))}+ \tau ^2\|f_1\|^2_{\L^2(C_n(B))}\,\frac{\d \tau}{\tau}
    \end{align}
    and estimate both parts separately. For the second integral, we choose the decomposition $f_0 = f$ and $f_1=0$ such that
    \begin{align}
    \label{eq: int 3}
        &\int\limits_{1}^\infty \tau^{-2\theta}\inf\limits_{f=f_0+f_1} \sum\limits_{n=0}^\infty 2^{-\nu n}\|f_0\|^2_{\L^2(C_n(B))}+ \tau ^2\|f_1\|^2_{\L^2(C_n(B))}\,\frac{\d \tau}{\tau} \notag\\
        &\leq \int\limits_{1}^\infty \tau^{-2\theta} \sum\limits_{n=0}^\infty 2^{-\nu n}\|f\|^2_{\L^2(C_n(B))}\,\frac{\d \tau}{\tau} \\
        &\lesssim\sum\limits_{n=0}^\infty  2^{ -\nu(1-\theta)n}\|f\|^2_{\L^2(C_n(B))}.\notag
    \end{align}
    For the first integral in \eqref{eq: split int}, we divide the integral into
    \begin{align}
    \label{eq: int 2}
        &\int\limits_{0}^1 \tau^{-2\theta}\inf\limits_{f=f_0+f_1} \sum\limits_{n=0}^\infty 2^{-\nu n}\|f_0\|^2_{\L^2(C_n(B))}+ \tau ^2\|f_1\|^2_{\L^2(C_n(B))}\,\frac{\d \tau}{\tau}\nonumber\\
        &= \sum\limits_{k=0}^\infty \int\limits_{2^{-\frac{\nu}{2} (k+1)}}^{2^{-\frac{\nu}{2} k}} \tau^{-2\theta}\inf\limits_{f=f_0+f_1} \sum\limits_{n=0}^\infty 2^{-\nu n}\|f_0\|^2_{\L^2(C_n(B))}+ \tau ^2\|f_1\|^2_{\L^2(C_n(B))}\,\frac{\d \tau}{\tau}.
    \end{align}
    Now, for each $k\in\IN_0$, we choose the decomposition $f_1 = f_{2^{k+1}r}$ and $f_0 = f-f_{2^{k+1}r}$, where
    \begin{align*}
        f_{2^{k+1}r} = \eta f - \cB_{C_1(B_k)}(\nabla\eta\cdot f)
    \end{align*}
    from Lemma~\ref{lem: Bogovskii lemma} associated to the ball $B_k = 2^{k}B$. Since we have
    \begin{align*}
        &\supp(f_{2^{k+1}r}) \subset \overline{B(x,2^{k+1}r)},\\
        &\supp(f-f_{2^{k+1}r})\subset \IR^d\setminus  B(x,2^{k}r),
    \end{align*}
    and 
    \begin{align*}
        \|f_{2^{k+1}r}\|_{\L^2(C_n(B))}+\|f-f_{2^{k+1}r}\|_{\L^2(C_n(B))}\lesssim\|f\|_{\L^2(C_n(B))}
    \end{align*}
    for all $n\in\IN_0$, it follows
    \begin{align*}
        &\inf\limits_{f=f_0+f_1} \sum\limits_{n=0}^\infty 2^{-\nu n}\|f_0\|^2_{\L^2(C_n(B))}+ \tau ^2\|f_1\|^2_{\L^2(C_n(B))}\\
        &\leq \sum\limits_{n=0}^{k+1} \tau ^2\|f_{2^{k+1}r}\|^2_{\L^2(C_n(B))} +  \sum\limits_{n=k+1}^\infty 2^{-\nu n}\|f-f_{2^{k+1}r}\|^2_{\L^2(C_n(B))}\\
        &\lesssim \sum\limits_{n=0}^{k+1} \tau ^2\|f\|^2_{\L^2(C_n(B))} +  \sum\limits_{n=k+1}^\infty 2^{-\nu n}\|f\|^2_{\L^2(C_n(B))}.
    \end{align*}
    Plugging this estimate into \eqref{eq: int 2} yields
    \begin{align*}
        &\int\limits_{0}^1 \tau^{-2\theta}\inf\limits_{f=f_0+f_1} \sum\limits_{n=0}^\infty 2^{-\nu n}\|f_0\|^2_{\L^2(C_n(B))}+ \tau ^2\|f_1\|^2_{\L^2(C_n(B))}\,\frac{\d \tau}{\tau}\\
        &\lesssim\sum\limits_{k=0}^\infty \int\limits_{2^{-\frac{\nu}{2} (k+1)}}^{2^{-\frac{\nu}{2} k}}  \tau^{-2\theta}\sum\limits_{n=0}^{k+1} \tau ^2\|f\|^2_{\L^2(C_n(B))} +  \sum\limits_{n=k+1}^\infty 2^{-\nu n}\|f\|^2_{\L^2(C_n(B))}\,\frac{\d \tau}{\tau}\\
        &=\sum\limits_{k=0}^\infty\sum\limits_{n=0}^{k+1} \int\limits_{2^{-\frac{\nu}{2} (k+1)}}^{2^{-\frac{\nu}{2} k}} \tau^{1-2\theta}\,\d \tau \|f\|^2_{\L^2(C_n(B))} +  \sum\limits_{k=0}^\infty\sum\limits_{n=k+1}^{\infty} 2^{-\nu n}\int\limits_{2^{-\frac{\nu}{2} (k+1)}}^{2^{-\frac{\nu}{2} k}} \tau^{-1-2\theta}\,\d \tau \|f\|^2_{\L^2(C_n(B))}\\
        &\lesssim\sum\limits_{k=0}^\infty\sum\limits_{n=0}^{k+1} 2^{-\nu (1-\theta)k}\|f\|^2_{\L^2(C_n(B))} +  \sum\limits_{k=0}^\infty  \sum\limits_{n=k+1}^{\infty} 2^{-\nu n}2^{\nu \theta (k+1)} \|f\|^2_{\L^2(C_n(B))}.
    \intertext{Using Fubini's theorem and the structure of the coefficients $2^{-\nu k}$, we further estimate}
        &=\sum\limits_{n=0}^\infty\sum\limits_{k=\max\{0,n-1\}}^{\infty} 2^{-\nu(1-\theta)k}\|f\|^2_{\L^2(C_n(B))} + \sum\limits_{n=1}^\infty 2^{-\nu n} \sum\limits_{k=0}^{n-1} 2^{\nu \theta(k+1)} \|f\|^2_{\L^2(C_n(B))}\\
        &\lesssim\sum\limits_{n=0}^\infty 2^{-\nu (1-\theta)n}\|f\|^2_{\L^2(C_n(B))} + \sum\limits_{n=0}^\infty 2^{-\nu n}  2^{\nu\theta  n} \|f\|^2_{\L^2(C_n(B))}\\
        &= \sum\limits_{n=0}^\infty 2^{-\nu (1-\theta)n}\|f\|^2_{\L^2(C_n(B))}.
    \end{align*}
    Combining this estimate with \eqref{eq: int 3} yields
    \begin{align*}
        \int\limits_{0}^\infty \tau^{-2\theta}\inf\limits_{f=f_0+f_1} \sum\limits_{n=0}^\infty 2^{-\nu n}\|f_0\|^2_{\L^2(C_n(B))}+ \tau ^2\|f_1\|^2_{\L^2(C_n(B))}\,\frac{\d \tau}{\tau}\lesssim \sum\limits_{n=0}^\infty 2^{-\nu(1-\theta)n}\|f\|^2_{\L^2(C_n(B))}.
    \end{align*}
    Finally, using this estimate in \eqref{eq: int 0}, we conclude
    \begin{align*}
         \|f\|^2_{(X^\nu_0(B),X^\nu_1(B))_{\theta,2}}&\lesssim \sum\limits_{n=0}^\infty  2^{-\nu(1-\theta)n}\|f\|^2_{\L^2(C_n(B))}
    \end{align*}
    where the implicit constant depends only on $d,\theta$ and $\nu$.\\

    \noindent \textbf{Step 2: $X^\nu_\theta(B)  \supset (X^\nu_0(B) ,X^\nu_1(B) )_{\theta,2} $}.
    We observe that
    \begin{align}
    \label{eq: 1}
        \|f\|^2_{X_\theta^\nu(B)} &=\sum\limits_{n=0}^\infty  2^{-\nu (1-\theta)n}\|f\|^2_{\L^2(C_n(B))} \nonumber\\
        &\simeq \sum\limits_{n=0}^\infty \int\limits_{0}^{2^{-\frac{\nu}{2} n}} \tau^{-2\theta} \min\{2^{-\nu n},\tau ^2\}\|f\|^2_{\L^2(C_n(B))}\,\frac{\d \tau}{\tau}\\
        &\leq \int\limits_{0}^{\infty} \tau^{-2\theta} \sum\limits_{n=0}^\infty\min\{2^{-\nu n},\tau ^2\}\|f\|^2_{\L^2(C_n(B))}\,\frac{\d \tau}{\tau}.\nonumber
    \end{align}
    Since we have
    \begin{align*}
        \|f\|_{\L^2(C_n(B))}^2 \lesssim \|f_0\|^2_{\L^2(C_n(B))} + \|f_1\|^2_{\L^2(C_n(B))}
    \end{align*}
    for all $n\in\IN_0$, $f_0\in X^\nu_0(B) $ and $f_1 \in  X^\nu_1(B) $ with $f= f_0 + f_1$, we obtain
    \begin{align*}
        \|f\|^2_{\L^2(C_n(B))}\lesssim  \inf\limits_{f=f_0+f_1}\|f_0\|_{\L^2(C_n(B))}^2 + \|f_1\|_{\L^2(C_n(B))}^2.
    \end{align*}
    Plugging this estimate into \eqref{eq: 1}, we conclude
    \begin{align*}
        \|f\|^2_{X^\nu_\theta(B) } &\lesssim  \int\limits_{0}^{\infty} \tau^{-2\theta} \sum\limits_{n=0}^\infty\min\{2^{-\nu n},\tau ^2\}\inf\limits_{f=f_0+f_1}\|f_0\|_{\L^2(C_n(B))}^2 + \|f_1\|_{\L^2(C_n(B))}^2\,\frac{\d \tau}{\tau}\\
        &\leq  \int\limits_{0}^{\infty} \tau^{-2\theta} \inf\limits_{f=f_0+f_1}\sum\limits_{n=0}^\infty 2^{-\nu n}\|f_0\|^2_{\L^2(C_n(B))} + \tau ^2\|f_1\|^2_{\L^2(C_n(B))}\,\frac{\d \tau}{\tau}\\
        &= \int\limits_{0}^\infty \tau^{-2\theta}\inf\limits_{f=f_0+f_1} \|f_0\|^2_{X^\nu_0(B)} + \tau^2\|f_1\|^2_{X^\nu_1(B)}\,\frac{\d \tau}{\tau}\\
        &\simeq \|f\|^2_{(X^\nu_0(B) ,X^\nu_1(B) )_{\theta,2}},
    \end{align*}
    where we used \eqref{eq: int 0} in the last step. This proves the claim.
\end{proof}

Using a reiteration theorem for real interpolation, see for example \cite[Cor.~1.24]{Lunardi}, we obtain the following result.

\begin{lemma}
\label{lem: reiterated int lem}
     Let $B=B(x,r)\subset \IR^d$, $0\leq \theta_0,\theta_1,\theta\leq 1$ and $\nu\geq 0$. Then, we have
    \begin{align*}
        X^\nu_{(1-\theta)\theta_0 + \theta \theta_1}(B)  = (X^\nu_{\theta_0}(B) ,X^\nu_{\theta_1}(B) )_{\theta,2}   \quad \text{and} \quad Y^\nu_{(1-\theta)\theta_0 + \theta \theta_1} (B) = (Y^\nu_{\theta_0}(B) ,Y^\nu_{\theta_1}(B) )_{\theta,2}
    \end{align*}
    with equivalence of norms. The implicit constants are independent of the underlying ball.
\end{lemma}

We use this abstract result to interpolate the $\L^2$-$\L^{q}$ decay estimates with the $\L^2$-$\L^{p}$ boundedness.

\begin{proposition}
\label{prop: int principles}
     Let $2\leq p_0,p_1\leq \infty$ and $\Omega\subset \IC\setminus\{0\}$ be open. Furthermore, let $V,W\in\{\L^2_\sigma(\IR^d), \L^2(\IR^d;\IC^{d\times d})\}$ and $(T(z))_{z\in\Omega}\subset \mathcal{L}(V,W)$. For $0\le \theta\leq 1$ define
     \begin{align*}
         \frac{1}{p_\theta} \coloneqq \frac{1-\theta}{p_0} + \frac{\theta}{p_1}.
     \end{align*}
     Then we have the following assertions.
     \begin{enumerate}
        \item\label{item: (1)} If $(T(z))_{z\in\Omega}$ satisfies $\L^2$-$\L^{p_0}$ and $\L^2$-$\L^{p_1}$ decay estimates of order $\nu_0\geq0$ and $\nu_1\geq 0$ respectively, then $(T(z))_{z\in\Omega}$ satisfies $\L^2$-$\L^{p_\theta}$ decay estimates of order $\nu_0(1-\theta) +\nu_1\theta\geq 0$.
        
        \item \label{item: (2)} If $(T(z))_{z\in\Omega}$ satisfies $\L^2$-$\L^{p_0}$ decay estimates of order $\nu\geq 0$ and $\L^2$-$\L^{p_1}$ bounds, then $(T(z))_{z\in\Omega}$ satisfies $\L^2$-$\L^{p_\theta}$ decay estimates of order $\nu(1-\theta)\geq 0$.
     \end{enumerate}
     
\end{proposition}

\begin{proof}
    For simplicity, we will only show assertion~\eqref{item: (1)} for $V=W =\L^2_\sigma(\IR^d)$. The other cases follow the same line of reasoning. Assertion~\eqref{item: (2)} is a direct consequence of~\eqref{item: (1)}, since by Remark~\ref{rem: hypercontractivity and decay estimates} the $\L^2$-$\L^{p_1}$ bounds are $\L^2$-$\L^{p_1}$ decay estimates of order $\nu_1=0$.
    
    For \eqref{item: (1)}, by symmetry, we may assume that $\nu_1\geq \nu_0$. Then, there exists $\rho\in[0,1)$ such that $\nu_0=\nu_1(1-\rho)$. Fix a ball $B=B(x_0,r)\subset \IR^d$ and $z\in\Omega$. For $f\in \L^2_\sigma(\IR^d)$, the assumptions imply
    \begin{align*}
        \|T(z)f\|_{\L^{p_0}(B)} 
        &\lesssim r^{-d(\frac{1}{2}-\frac{1}{{p_0}})}\max\Big\{\Big(\frac{r^2}{|z|}\Big)^{M_0}, \Big(\frac{r^2}{|z|}\Big)^{-N_0}\Big\}\|f\|_{X^{\nu_1}_\rho(B) }
    \end{align*}
    as well as
    \begin{align*}
        \|T(z)f\|_{\L^{p_1}(B)} 
        \lesssim r^{-d(\frac{1}{2}-\frac{1}{{p_1}})}\max\Big\{\Big(\frac{r^2}{|z|}\Big)^{M_1}, \Big(\frac{r^2}{|z|}\Big)^{-N_1}\Big\} \|f\|_{X^{\nu_1}_0(B) }
    \end{align*}
    for some $N_0,N_1,M_0,M_1\geq 0$.
    Since $\L^2_\sigma(\IR^d)$ is a dense subspace of $X^{\nu_1}_\rho(B)$ and $X^{\nu_1}_0(B)$ by Lemma~\ref{lem: facts of X and Y}, we see that $T(z)\in\cL(X^{\nu_1}_\rho(B),\L^{p_0}(B))\cap \cL(X^{\nu_1}_0(B),\L^{p_1}(B))$. Applying \cite[Thm.~1.6]{Lunardi} and Lemma~\ref{lem: reiterated int lem}, we obtain the bound
    \begin{align*}
        &\|T(z)f\|_{(\L^{p_0}(B),\L^{p_1}(B))_{\theta,2}}\\
        &\lesssim r^{-d(1-\theta)(\frac{1}{2}-\frac{1}{{p_0}})-d\theta(\frac{1}{2}-\frac{1}{p_1})}\max\Big\{\Big(\frac{r^2}{|z|}\Big)^{M_\theta}, \Big(\frac{r^2}{|z|}\Big)^{-N_\theta}\Big\}\|f\|_{X^{\nu_1}_{(1-\theta)\rho}(B) }\\
        &=r^{-d(\frac{1}{2}-\frac{1}{{p_\theta}})}\max\Big\{\Big(\frac{r^2}{|z|}\Big)^{M_\theta}, \Big(\frac{r^2}{|z|}\Big)^{-N_\theta}\Big\}\|f\|_{X^{\nu_0(1-\theta) +\nu_1\theta}_{0}(B) },
    \end{align*}
    where we set $M_\theta = (1-\theta)M_0+\theta M_1$ and $N_\theta =(1-\theta)N_0+\theta N_1$, and used the relation
    \begin{align*}
        \nu_1(1-(1-\theta)\rho)= \nu_1(1-\rho)+\nu_1\theta \rho = \nu_0+(\nu_1-\nu_0)\theta = \nu_0(1-\theta) +\nu_1\theta.
    \end{align*}
    Finally, since $p_\theta \geq 2$, applying \cite[Thm.~3.4.1 and 5.2.1]{Berg_Lofstrom} yields the continuous embedding
    \begin{align*}
        (\L^{p_0}(B), \L^{p_1}(B))_{\theta,2}  \hookrightarrow (\L^{p_0}(B), \L^{p_1}(B))_{\theta,p_\theta } = \L^{p_\theta}(B).
    \end{align*}
    Hence, we conclude
    \begin{align*}
        &\|T(z)f\|_{\L^{p_\theta}(B)}\\
        &\lesssim r^{-d(\frac{1}{2}-\frac{1}{{p_\theta}})}\max\Big\{\Big(\frac{r^2}{|z|}\Big)^{M_\theta}, \Big(\frac{r^2}{|z|}\Big)^{-N_\theta}\Big\}\bigg(\sum\limits_{n=0}^\infty 2^{-(\nu_0(1-\theta) +\nu_1\theta) n} \|f\|^2_{\L^2(C_{n}(B))}\bigg)^\frac{1}{2},
    \end{align*}
    proving the assertion~\eqref{item: (1)}.
\end{proof}

As a consequence, we obtain the following interplay of $\L^2$-$\L^p$ decay estimates, hypercontractivity and uniform $\L^p$-boundedness.

\begin{corollary}
\label{cor: decay, hyper, unif}
     Let $\Omega\subset \IC\setminus\{0\}$ be open, $V,W\in\{\L^2_\sigma(\IR^d), \L^2(\IR^d;\IC^{d\times d})\}$ and $(T(z))_{z\in\Omega}\subset \mathcal{L}(V,W)$.
     Assume $(T(z))_{z\in\Omega}$ satisfies $\L^2$-$\L^{q}$ decay estimates of order $\nu>d$ for some $q\in[2,\infty]$ and $\L^2$-$\L^{p}$ bounds for some $p>q$. Then $(T(z))_{z\in\Omega}$ is uniformly $\L^r$-bounded for all $r\in [q, p_\theta)$, where
     \begin{align*}
         \frac{1}{p_\theta} \coloneqq \frac{1-\theta}{q} + \frac{\theta}{p}
     \end{align*}
     and $\theta\coloneqq \sup\{\vartheta\in(0,1): \nu(1-\vartheta)>d\}$.
\end{corollary}

\begin{proof}
    Let $p_\theta\in (q,p)$ and $\theta\in (0,1)$ be as defined in the assertion. By Proposition~\ref{prop: int principles}, for all $0\leq \vartheta<\theta<1$, we obtain $\L^2$-$\L^{p_\vartheta}$ decay estimates of order $\nu(1-\vartheta)>d$, where
    \begin{align*}
        \frac{1}{p_\vartheta} =  \frac{1-\vartheta}{q} + \frac{\vartheta}{p}.
    \end{align*}
    Applying Proposition~\ref{prop: L2-Lp off imply Lp bdd} yields uniform $\L^{p_\vartheta}$-boundedness of $(T(z))_{z\in\Omega}$ for all  $0\leq \vartheta<\theta$, which proves the claim.
\end{proof}

\section{$\L^p$-theory for the Stokes semigroup}
\label{sec: Lp-theory sg}
The purpose of this section is to apply the results of the preceding sections to the Stokes semigroup and establish its corresponding $\L^p$-theory. First, using Cauchy's integral formula \eqref{eq: rep semigroup}, we transfer the decay estimates from the Stokes resolvent to the analytic semigroup. Subsequently, we establish hypercontractivity properties that extend beyond the classical Sobolev range $[2_*, 2^*]$ for dimensions $d \geq 3$, while for $d = 2$ we recover boundedness for the full range $p \in (1, \infty)$. Finally, we employ the interpolation principles derived in Section~\ref{subsec: interpolation principles} to deduce $\L^2$-$\L^p$ decay estimates and uniform $\L^p$-boundedness of the Stokes semigroup also outside the Sobolev range, see Theorem~\ref{thm: Lp-extrapolation sg II}.

\subsection{Decay estimates for the semigroup}

We start by transferring $\L^2$-decay estimates.

\begin{proposition}
\label{prop: off-diag semigroup}
    Let $\mu$ satisfy Assumption~\ref{Ass: Coefficients}, $\omega_0$ be given by \eqref{eq: def omega} and $\beta\in [0,\frac{\pi}{2}-\omega_0)$.  Then for all $\nu \in (0,d+2)$ there exists $C>0$ such that for all balls $B=B(x_0,r)$, $z\in \S_{\beta}$ and $f\in\L^2_\sigma(\IR^d)$ we have
    \begin{align*}
        \|e^{-zA}f\|_{\L^2(B)} \leq C\max\Big\{1, \Big(\frac{r^2}{|z|}\Big)^{-\frac{d+1}{2}}\Big\}\bigg(\sum\limits_{n=0}^\infty 2^{-\nu n}\|f\|_{\L^2(C_n(B))}^2\bigg)^\frac{1}{2},
    \end{align*}
    as well as
    \begin{align*}
        \||z|^\frac{1}{2}\nabla e^{-zA}f\|_{\L^2(B)} \leq C \max\Big\{1, \Big(\frac{r^2}{|z|}\Big)^{-\frac{d+2}{2}}\Big\}\bigg(\sum\limits_{n=0}^\infty 2^{-\nu n}\|f\|_{\L^2(C_n(B))}^2\bigg)^\frac{1}{2}.
    \end{align*}
    The constant $C$ depends only on $\beta,\nu,d$ and $\mu_\bullet,\mu^\bullet$.
\end{proposition}

\begin{proof}
    We only provide a proof for the semigroup, since the estimate for its gradient follows by similar reasoning. By Cauchy's formula, we have
    \begin{align*}
        \|e^{-zA}f\|_{\L^2(B)}
        &\lesssim \int\limits_{\gamma_z} e^{\Re (z\lambda)}\| (\lambda +A)^{-1}f\|_{\L^2(B)}\,|\d \lambda|.
    \intertext{Applying Theorem~\ref{thm: off-diag resolvent} yields}
        &\lesssim\int\limits_{\gamma_z} e^{\Re (z\lambda)} \frac{1}{|\lambda|} \max\Big\{1, (|\lambda|r^2)^{-\frac{d+1}{2}}\Big\}\bigg(\sum\limits_{n=0}^\infty 2^{-\nu n}\|f\|_{\L^2(C_n(B))}^2\bigg)^\frac{1}{2}\,|\d \lambda|.
    \intertext{By the substitution $|z|\lambda = \zeta$, we obtain}
        &=\int\limits_{\gamma_1} e^{\Re (\frac{\zeta z}{|z|})} \frac{|z|}{|\zeta|} \max\Big\{1, \Big(\frac{|\zeta|r^2}{|z|}\Big)^{-\frac{d+1}{2}}\Big\}\bigg(\sum\limits_{n=0}^\infty 2^{-\nu n}\|f\|_{\L^2(C_n(B))}^2\bigg)^\frac{1}{2}\,\frac{|\d \zeta|}{|z|}.
    \intertext{Finally, since $|\zeta|\geq 1$, we estimate}
        &\lesssim   \max\Big\{1, \Big(\frac{r^2}{|z|}\Big)^{-\frac{d+1}{2}}\Big\}\bigg(\sum\limits_{n=0}^\infty 2^{-\nu n}\|f\|_{\L^2(C_n(B))}^2\bigg)^\frac{1}{2}\int\limits_{\gamma_1} e^{\Re (\frac{\zeta z}{|z|})}\,\frac{|\d \zeta|}{|\zeta|}\\
        &\lesssim \max\Big\{1, \Big(\frac{r^2}{|z|}\Big)^{-\frac{d+1}{2}}\Big\}\bigg(\sum\limits_{n=0}^\infty 2^{-\nu n}\|f\|_{\L^2(C_n(B))}^2\bigg)^\frac{1}{2},
    \end{align*}
    proving the claim.
\end{proof}

Using the improved decay estimates from Proposition~\ref{prop: imp off-diag est}, the same argument gives the following result.

\begin{proposition}
\label{prop: imp off-diag semigroup}
    Let $\mu$ satisfy Assumption~\ref{Ass: Coefficients}, $\omega_0$ be given by \eqref{eq: def omega} and $\beta\in [0,\frac{\pi}{2}-\omega_0)$.  Then for all $\nu \in (0,d+2)$, there exists $C>0$ such that for all balls $B=B(x_0,r)$, $z\in \S_{\beta}$ and $f\in\L^2_\sigma(\IR^d)$ with $\supp f \subset \IR^d\setminus 2B$ we have
    \begin{align*}
        \|e^{-zA}f\|_{\L^2(B)} \leq \frac{C|z|^\frac{1}{2}}{r} \max\Big\{1, \Big(\frac{r^2}{|z|}\Big)^{-\frac{d+1}{2}}\Big\}\bigg(\sum\limits_{n=0}^\infty 2^{-\nu n}\|f\|_{\L^2(C_n(B))}^2\bigg)^\frac{1}{2}.
    \end{align*}
    The constant $C$ depends only on $\beta,\nu,d$ and $\mu_\bullet,\mu^\bullet$.
\end{proposition}

Alternatively, using the $\L^2$-$\L^p$ decay estimates from Theorem~\ref{thm: L2-Lp off diag stokes resolv}, we obtain the following result.

\begin{proposition}
\label{prop: Lp off-diag sg}
    Let $\mu$ satisfy Assumption~\ref{Ass: Coefficients}, $\omega_0$ be given by \eqref{eq: def omega} and $\beta\in [0,\frac{\pi}{2}-\omega_0)$.  Then for all $\nu \in (0,d+2)$ and $p\in [2,\infty)$ satisfying $\frac{1}{2}-\frac{1}{p}\leq \frac{1}{d}$ there exists a constant $C>0$ such that for all balls $B= B(x_0,r)$, $z\in \S_\beta$ and $f\in \L_\sigma^2(\IR^d)$ we have
    \begin{align*}
        \|e^{-zA}f\|_{\L^{p}(B)} \leq \frac{C}{r^{d(\frac{1}{2}-\frac{1}{p})}}\max\Big\{\Big(\frac{r^2}{|z|}\Big)^\frac{1}{2}, \Big(\frac{r^2}{|z|}\Big)^{-\frac{d+1}{2}}\Big\}\bigg(\sum\limits_{n=0}^\infty 2^{-\nu n}\|f\|_{\L^2(C_n(B))}^2 \bigg)^\frac{1}{2}.
     \end{align*} 
     The constant $C$ depends only on $\beta,\nu,d,p$ and $\mu_\bullet,\mu^\bullet$.
\end{proposition}

\subsection{$\L^p$-bounds for the semigroup}

We start with an auxiliary lemma that reduces the $\L^2$-$\L^p$ bounds for the whole semigroup family or its gradient to only one member.

\begin{lemma}
\label{lem: rescaling lemma}
    Let $\mu$ satisfy Assumption~\ref{Ass: Coefficients}, $\omega_0$ be given by \eqref{eq: def omega} and $p\in[2,\infty)$. Furthermore, let $T(z)\in \{e^{-zA},|z|^\frac{1}{2}\nabla e^{-zA}\}$ for $z\in\S_{\frac{\pi}{2}-\omega_0}$ and assume there exists a constant $C>0$ depending only on $d,p$, $\mu_\bullet$ and $\mu^\bullet$ such that
    \begin{align}
    \label{eq: T1}
        \|T(1)f\|_{\L^{p}} \leq C\|f\|_{\L^2}
    \end{align}
    holds for every $f\in\L^2_\sigma(\IR^d)$. Then, for every $\beta\in[0,\frac{\pi}{2}-\omega_0)$ there exists a constant $C'>0$ depending only on $d,p$, $\beta$ and $\mu_\bullet,\mu^\bullet$ such that for all $z\in\S_\beta$ and $f\in \L^2_\sigma(\IR^d)$ we have
    \begin{align*}
        \|T(z)f\|_{\L^{p}} \leq C'|z|^{-\frac{d}{2}(\frac{1}{2}-\frac{1}{p})}\|f\|_{\L^2}.
    \end{align*}
\end{lemma}

\begin{proof}
    We only prove the assertion for $T(z)=e^{-zA}$, the other case being similar. Assume $z\in \S_{\beta}$ and write $z=s^2e^{i\theta}$ for a suitable $s>0$ and $\theta\in (-\beta,\beta)$. Notice that $e^{i\theta}A_s$ belongs to the same class as $A$, see Observation~\ref{observation}. Hence, we can apply inequality~\eqref{eq: T1} to $T(1)=e^{-e^{i\theta}A_s}$ to get with Lemma~\ref{lem: rescaled resolvent and semigroup} the estimate
    \begin{align*}
        \| e^{-zA}f\|_{\L^p} =\|\Phi_{s^{-1}}e^{-e^{i\theta}A_s}\Phi_{s}f\|_{\L^p}\leq C' s^{\frac{d}{p}-\frac{d}{2}}\|f\|_{\L^2}=C'|z|^{-\frac{d}{2}(\frac{1}{2}-\frac{1}{p})}\|f\|_{\L^2},
    \end{align*}
    for some constant $C'>0$ depending only on $d,p$, $\beta$ and $\mu_\bullet,\mu^\bullet$.
\end{proof}

Combining Sobolev's theorem with the uniform $\L^2$-boundedness of the gradient of the semigroup, we immediately obtain the following hypercontractivity property.

\begin{lemma}
    Let $\mu$ satisfy Assumption~\ref{Ass: Coefficients}, $\omega_0$ be given by \eqref{eq: def omega} and $\beta\in [0,\frac{\pi}{2}-\omega_0)$. For all $p\in[2,\infty)$ satisfying $\frac{1}{2}-\frac{1}{p}\leq \frac{1}{d}$ the semigroup $(e^{-zA})_{z\in\S_\beta}$ satisfies $\L^2$-$\L^p$ bounds.
\end{lemma}

\begin{proof}
    Combining Sobolev's embedding with Proposition~\ref{prop: L2 semigroup bounds} yields
    \begin{align*}
        \|e^{-A}f\|_{\L^{p}} \lesssim \|\nabla e^{-A}f\|_{\L^{2}}+\|f\|_{\L^2} \lesssim\|f\|_{\L^2}
    \end{align*}
    for every $f\in\L^2_\sigma(\IR^d)$. Now, the claim follows from Lemma~\ref{lem: rescaling lemma}.
\end{proof}

While the previous result establishes the $\L^2$-$\L^p$ bounds of the semigroup within the upper Sobolev range $2 \leq p \leq 2^*$, our goal is to extrapolate this bound beyond these classical limits. Following the established framework for elliptic operators, see \cite{Auscher-Lp}, we employ an abstract extrapolation result of \u{S}ne\u{\i}berg~\cite{Sneiberg} to extend the hypercontractivity of the Stokes semigroup outside the Sobolev range, as in the classical elliptic theory. The result of \u{S}ne\u{\i}berg reads as follows.

\begin{proposition}  
\label{prop: Sneiberg}
    Let $(X_0,X_1)$ and $(Y_0,Y_1)$ be two interpolation couples and let $T\in \cL(X_0,Y_0)\cap \cL(X_1,Y_1)$. Then
    \begin{align*}
        \big\{\vartheta\in (0,1) : T:[X_0,X_1]_\vartheta\to [Y_0,Y_1]_\vartheta \text{ is an isomorphism} \big\}
    \end{align*}
    is an open set.
\end{proposition}
We refer to \cite[Chap.~2]{Lunardi} for an introduction to complex interpolation spaces and to \cite[Sec.~1.3.5]{Egert-Dissertation} for an extensive treatment of \u{S}ne\u{\i}berg's theorem. Our goal is to apply this result to the weak Stokes resolvent $1+\cA$. Since this requires a priori boundedness of the operator on $\W^{1,p}_\sigma(\IR^d)$ for $p \neq 2$, we first extend its definition to the full range of $p$ as follows.

\begin{definition}
    Let $1<p<\infty$. For $u\in \W^{1,p}_\sigma(\IR^d)$ define the weak generalized Stokes operator $\mathcal{A}u \in \W^{-1,p}_\sigma(\IR^d)$ by
    \begin{align*}
        (\mathcal{A}u)(v)\coloneqq \int\limits_{\IR^d}\mu \nabla u\cdot \overline{\nabla v}\,\d x
    \end{align*}
    for $v\in \W^{1,p'}_\sigma(\IR^d)$.
\end{definition}

We observe, by Hölder's inequality and Assumption~\ref{Ass: Coefficients}, that $\cA$ is a bounded operator from $\W^{1,p}_\sigma(\IR^d)$ to $\W^{-1,p}_\sigma(\IR^d)$. Moreover, we know from Proposition~\ref{Prop: L2 resolvent bounds} that $1+\cA$ is an isomorphism from $\W^{1,2}_\sigma(\IR^d)$ to $\W^{-1,2}_\sigma(\IR^d)$. Hence, to be able to apply \u{S}ne\u{\i}berg's stability result, we have to show that the underlying function spaces form a complex interpolation scale.

\begin{lemma}
\label{lem: complex family}
    Let $0<\theta<1$, $1<p_0\leq p_1 <\infty$ and define
    \begin{align*}
        \frac{1}{p_\theta}\coloneqq \frac{1-\theta}{p_0}+\frac{\theta}{p_1}.
    \end{align*}
    Then we have
    \begin{align*}
        [\W^{1,p_0}_\sigma(\IR^d), \W^{1,p_1}_\sigma(\IR^d)]_\theta = \W^{1,p_\theta}_\sigma(\IR^d)
    \end{align*}
    and
    \begin{align*}
        [\W^{-1,p_0}_\sigma(\IR^d), \W^{-1,p_1}_\sigma(\IR^d)]_\theta = \W^{-1,p_\theta}_\sigma(\IR^d)
    \end{align*}
    with equivalence of norms.
\end{lemma}

\begin{proof}
    First, recall that $\IP\in \cL(\W^{1,p}(\IR^d;\IC^d))$ and $\IP(\W^{1,p}(\IR^d;\IC^d)) = \W^{1,p}_\sigma(\IR^d)$ for all $1<p<\infty$.
    Hence, we can extend the projection $\IP$ canonically to the sum space $\W^{1,p_0}(\IR^d;\IC^d) +  \W^{1,p_1}(\IR^d;\IC^d)$ whose range is characterized by $\mathcal{Z} \coloneqq   \W^{1,p_0}_\sigma(\IR^d) +  \W^{1,p_1}_\sigma(\IR^d)$. By the retraction-coretraction principle (see for example \cite[Cor.~1.3.6]{Egert-Dissertation} and \cite[Sec.~2.4.2]{Triebel}), we calculate the interpolation spaces as follows
    \begin{align*}
        [\W^{1,p_0}_\sigma(\IR^d), \W^{1,p_1}_\sigma(\IR^d)]_\theta &= [ \mathcal{Z}\cap \W^{1,p_0}(\IR^d;\IC^d) ,\mathcal{Z}\cap \W^{1,p_1}(\IR^d;\IC^d)]_\theta\\
        &=  \mathcal{Z}\cap[ \W^{1,p_0}(\IR^d;\IC^d) , \W^{1,p_1}(\IR^d;\IC^d)]_\theta\\
        &=  \mathcal{Z}\cap \W^{1,p_\theta}(\IR^d;\IC^d)\\
        &= \W^{1,p_\theta}_\sigma(\IR^d).
    \end{align*}
    Finally, the identity for the dual spaces $\W^{-1,p}_\sigma(\IR^d)$ follows from the first result combined with the duality theorem for complex interpolation \cite[Sec.~1.11.3]{Triebel}.
\end{proof}

\begin{theorem}[$\W^{1,p}$-extrapolation of $1+\mathcal{A}$]
\label{thm: extension of 1+A}
    Let $\mu$ satisfy Assumption~\ref{Ass: Coefficients}. There exists $\delta>0 $ depending on $d$ and $\mu_\bullet,\mu^\bullet$ such that the operator $1+\cA$ is an isomorphism between $\W^{1,p}_\sigma(\IR^d) $ and $\W^{-1,p}_\sigma(\IR^d)$ for any $p\in (1,\infty)$ with $|\frac{1}{2}-\frac{1}{p}|<\delta$.
\end{theorem}

\begin{proof}
    Set $p_0= \frac{3}{2}$,  $p_1 = 3$ and $\vartheta = \frac{1}{2}$ such that
    \begin{align*}
        \frac{1}{2} = \frac{1-\vartheta}{p_0} + \frac{\vartheta}{p_1}.
    \end{align*}
    Note that $(X_0,X_1)=(\W^{1,p_0}_\sigma, \W^{1,p_1}_\sigma)$ and $(Y_0,Y_1)=(\W^{-1,p_0}_\sigma, \W^{-1,p_1}_\sigma)$ form interpolation couples and $1 +\cA \in \cL(X_0,Y_0)\cap \cL(X_1,Y_1)$. Moreover, by Proposition~\ref{Prop: L2 resolvent bounds} and  Lemma~\ref{lem: complex family}, we know that $1 +\cA:\W^{1,2}_\sigma(\IR^d) \to \W^{-1,2}_\sigma(\IR^d)$ is a bounded isomorphism and
    \begin{align*}
        [\W^{1,p_0}_\sigma(\IR^d), \W^{1,p_1}_\sigma(\IR^d)]_\vartheta = \W^{1,2}_\sigma(\IR^d)\quad \text{and}\quad  [\W^{-1,p_0}_\sigma(\IR^d), \W^{-1,p_1}_\sigma(\IR^d)]_\vartheta = \W^{-1,2}_\sigma(\IR^d).
    \end{align*}
    Hence, we can apply Proposition~\ref{prop: Sneiberg} to complete the proof.
\end{proof}

We are now in a position to extrapolate the $\L^2$-$\L^p$ bounds of the semigroup to some $p$ beyond the upper Sobolev exponent $2^*$.

\begin{theorem}
\label{thm: L2-L2*+eps bdd semigroup}
    Let $\mu$ satisfy Assumption~\ref{Ass: Coefficients}, $\omega_0$ be given by \eqref{eq: def omega} and $\beta\in [0,\frac{\pi}{2}-\omega_0)$. There exists a constant $\delta>0$ depending on $\beta,d$ and $\mu_\bullet,\mu^\bullet$ such that for all $p\in[2,\infty)$ satisfying $\frac{1}{2}-\frac{1}{p}<\frac{1}{d}+\delta$ the semigroup $(e^{-zA})_{z\in\S_\beta}$ satisfies $\L^2$-$\L^p$ bounds.
\end{theorem}

\begin{proof}
    Since $p=2$ is trivial, we fix $p\in (2,\infty)$ satisfying
    \begin{align*}
        \frac{1}{2}-\frac{1}{d}-\delta <\frac{1}{p}< \frac{1}{2},
    \end{align*}
    where $\delta>0$ is the same as in Theorem~\ref{thm: extension of 1+A}. By Lemma~\ref{lem: rescaling lemma}, it suffices to show
    \begin{align*}
        \|e^{-A}f\|_{\L^p}\lesssim \|f\|_{\L^2}
    \end{align*}
    for every $f\in \L^2_\sigma(\IR^d)$. Define recursively the sequence $(p_k)\subset [2,\infty)$ as follows: set $p_0 = p$ and define for $k\in\IN_0$
    \begin{align*}
        \frac{1}{p_{k+1}} \coloneqq
        \begin{cases}
        \frac{1}{p_k} + \frac{1}{d}, \quad &\text{if}\quad \frac{1}{p_k} + \frac{1}{d}<\frac{1}{2},\\
        \frac{1}{2}, \hfill &\text{else.}
        \end{cases}
    \end{align*}
    Then, we have 
    \begin{align*}
        \frac{1}{p_k}\leq\frac{1}{p_{k+1}} \quad \text{and} \quad \frac{1}{p_{k+1}} - \frac{1}{p_k} \leq \frac{1}{d}
    \end{align*}
    for all $k\geq 0$ and
    \begin{align*}
        \frac{1}{2}-\delta \leq  \frac{1}{p_k}\leq \frac{1}{2}
    \end{align*}
    for all $k>0$. Denote by $k_0\in\IN_0$ the minimal index such that $\frac{1}{p_{k_0}}=\frac{1}{2}$. Choosing $(q_k)_{k=0}^{k_0}$ such that the conditions
    \begin{align*}
        \frac{1}{p_k}\leq \frac{1}{q_k}\leq \frac{1}{p_{k+1}} \quad \text{and} \quad \frac{1}{2}-\delta <  \frac{1}{q_k} \leq \frac{1}{2}
    \end{align*}
    hold for all $0\leq k\leq k_0$, we can calculate
    \begin{align*}
        \frac{1}{q_k}-\frac{1}{d} \leq \frac{1}{p_{k+1}} - \frac{1}{d} \leq \frac{1}{p_k},
    \end{align*}
    as well as
    \begin{align*}
        \frac{1}{p_{k+1}}  \leq \frac{1}{p_k}+\frac{1}{d}  \leq \frac{1}{q_k}+\frac{1}{d},
    \end{align*}
    where the last one is, by definition, equivalent to
    \begin{align*}
        \frac{1}{q_k'}\leq  \frac{1}{p_{k+1}'} +\frac{1}{d}.
    \end{align*}
    Applying Sobolev's embeddings with respect to these relations, we obtain 
    \begin{align*}
        \W^{1,q_k}_\sigma \hookrightarrow \L^{p_{k}}_\sigma,
    \end{align*}
    as well as
    \begin{align*}
        \W^{1,q_k'}_\sigma \hookrightarrow \L^{p_{k+1}'}_\sigma,
    \end{align*}
    which implies by duality
    \begin{align*}
        \L^{p_{k+1}}_\sigma \hookrightarrow \W^{-1,q_k}_\sigma.
    \end{align*}
    These embeddings, in conjunction with Theorem~\ref{thm: extension of 1+A}, imply the boundedness of the map
    \begin{align}
    \label{eq: it sobolev}
        (1+\cA)^{-1} : (\L_\sigma^{p_{k+1}}\hookrightarrow \W^{-1,q_k}_\sigma) \to (\W^{1,q_k}_\sigma\hookrightarrow\L_\sigma^{p_{k}})
    \end{align}
    for all $0\leq k \leq k_0$. Since $(1+\cA)^{-1}f = (1+A)^{-1}f$ for $f\in\L^2_\sigma(\IR^d)$, we can iterate \eqref{eq: it sobolev} to conclude
    \begin{align*}
        \|(1+A)^{-k_0}f\|_{\L^{p_0}} \lesssim \|(1+A)^{-k_0+1}f\|_{\L^{p_1}} \lesssim \dots \lesssim \|f\|_{\L^{p_{k_0}}} = \|f\|_{\L^{2}}.
    \end{align*}
    Finally, by Proposition~\ref{prop: L2 semigroup bounds}, we have
    \begin{align*}
        \|e^{-A}f\|_{\L^{p}} = \Big\|(1+A)^{-k_0}\Big[(1+A)e^{-\frac{1}{k_0}A}\Big]^{k_0}f\Big\|_{\L^{p_0}}\lesssim  \Big\|\Big[(1+A)e^{-\frac{1}{k_0}A}\Big]^{k_0}f\Big\|_{\L^{2}} \lesssim \|f\|_{\L^2},
    \end{align*}
    which finishes the proof.
\end{proof}

We have now everything at hand to prove Theorem~\ref{thm: Lp-extrapolation sg II}.

\begin{proof}[Proof of Theorem~\ref{thm: Lp-extrapolation sg II}]
    We first prove the claim for $p\geq 2$.
    If $d=2$ the claim follows directly from Propositions~\ref{prop: Lp off-diag sg} and~\ref{prop: L2-Lp off imply Lp bdd}. Thus, assume $d\geq 3$.
    Let $\delta>0$ be defined via Theorem~\ref{thm: L2-L2*+eps bdd semigroup} and fix some $0<\delta'<\delta$. Then, the semigroup $(e^{-zA})_{z\in\S_\beta}$ satisfies $\L^2$-$\L^p$ bounds for $p\in[2,\infty)$ defined by $\frac{1}{2}-\frac{1}{p}=\frac{1}{d}+\delta'$.
    Furthermore, by Proposition~\ref{prop: Lp off-diag sg}, we know that $(e^{-zA})_{z\in\S_\beta}$ satisfies $\L^2$-$\L^q$ decay estimates of order $\nu>d$ for $q\in[2,\infty)$ defined by $\frac{1}{2}-\frac{1}{q} = \frac{1}{d}$. Applying Corollary~\ref{cor: decay, hyper, unif} yields uniform $\L^r$-boundedness of the semigroup for all $r\in[q,p_\theta)$ where
    \begin{align*}
        \frac{1}{p_\theta} = \frac{1-\theta}{q} + \frac{\theta}{p} = \frac{1}{2}-\frac{1}{d}-\theta\delta'
    \end{align*}
    for some $\theta\in (0,1)$. Interpolation with the $\L^2$-bounds from Proposition~\ref{prop: L2 semigroup bounds} implies $\L^r$-boundedness for all $r\in[2,\infty)$ satisfying $\frac{1}{2}-\frac{1}{r}<  \frac{1}{d}+ \theta\delta'$.
    To obtain the result for $p<2$, we apply the result to the adjoint operator $(e^{-zA})^*=e^{-\bar z A^*}$ to conclude the $\L^{r'}$-boundedness of $e^{-zA}$ by duality for all $r'\in(1,2]$ satisfying $\frac{1}{r'}-\frac{1}{2} < \frac{1}{d}+\theta\delta'$. Setting $\varepsilon =\theta\delta'$ proves the claim.
\end{proof}

\section{Gradient estimates}
\label{sec: gradient est}
In this section, we establish the corresponding $\L^p$-theory for the gradient of the semigroup. Our approach relies fundamentally on the following duality lemma.

\begin{lemma}
\label{lem: duality}
    Let $\mu$ satisfy Assumption~\ref{Ass: Coefficients}, $\omega_0$ be given by \eqref{eq: def omega} and $\beta\in [0,\frac{\pi}{2}-\omega_0)$. For $z\in\S_\beta$ the adjoint of $\nabla e^{-\bar z A^*}$ is given by $-e^{-z\cA} \cP\div$.
\end{lemma}

\begin{proof}
    Let $G\in \C_c^\infty(\IR^d;\IC^{d\times d})$ and $f\in \L^2_\sigma(\IR^d)$. Then we have
    \begin{align*}
        \langle G , \nabla e^{-\bar zA^*}f \rangle_{\L^2,\L^2} &= \langle   -\div (G) ,  e^{-\bar zA^*}f \rangle_{\L^2,\L^2}\\
         &= \langle   -\IP\div (G) ,  e^{-\bar z A^*}f \rangle_{\L^2_\sigma,\L^2_\sigma}= \langle   -e^{-z\cA}\cP\div (G) , f \rangle_{\L^2_\sigma,\L^2_\sigma}.
    \end{align*}
    Finally, the density of $\C_c^\infty(\IR^d;\IC^{d\times d})$ in $\L^2(\IR^d;\IC^{d\times d})$ implies the claim.
\end{proof}

This result allows us to establish $\L^p$-estimates for the operator family $(|z|^\frac{1}{2}\nabla e^{-zA})_{z\in\S_{\frac{\pi}{2}-\omega_0}}$ by analyzing the dual family $(|z|^\frac{1}{2}e^{-z\cA} \cP\div)_{z\in\S_{\frac{\pi}{2}-\omega_0}}$ on $\L^{p'}$ instead. This perspective is particularly advantageous if $p < 2$, which we investigate next.

\subsection{The case $p<2$}

We start by recovering $\L^2$-$\L^p$ decay estimates of $(|z|^\frac{1}{2}e^{-z\cA} \cP\div)_{z\in\S_{\frac{\pi}{2}-\omega_0}}$ as in the proof of Proposition~\ref{prop: off-diag semigroup}.

\begin{proposition}
\label{prop: Lp off-diag sg div}
    Let $\mu$ satisfy Assumption~\ref{Ass: Coefficients}, $\omega_0$ be given by \eqref{eq: def omega} and $\beta\in [0,\frac{\pi}{2}-\omega_0)$.  For every $\nu \in (0,d+2)$ and $p\in [2,\infty)$ satisfying $\frac{1}{2}-\frac{1}{p}\leq \frac{1}{d}$ there exists a constant $C >0$ such that for all balls $B= B(x_0,r)$, $z\in \S_\beta$ and $F\in\L^2(\IR^d;\IC^{d\times d})$ we have
    \begin{align*}
        \||z|^\frac{1}{2}e^{-z\cA}\cP\div(F)\|_{\L^{p}(B)} \leq \frac{C}{r^{d(\frac{1}{2}-\frac{1}{p})}}\max\Big\{\Big(\frac{r^2}{|z|}\Big)^\frac{1}{2}, \Big(\frac{r^2}{|z|}\Big)^{-\frac{d+1}{2}}\Big\}\bigg(\sum\limits_{n=0}^\infty 2^{-\nu n}\|F\|_{\L^2(C_n(B))}^2 \bigg)^\frac{1}{2}.
     \end{align*}
     The constant $C>0$ depends only on $\beta,\nu,d,p$ and $\mu_\bullet,\mu^\bullet$.
\end{proposition}

Next, we derive $\L^2$-$\L^p$ bounds for $(|z|^\frac{1}{2}e^{-z\cA} \cP\div)_{z\in\S_{\frac{\pi}{2}-\omega_0}}$ from the corresponding bounds from the semigroup $(e^{-zA})_{z\in\S_{\frac{\pi}{2}-\omega_0}}$ itself.

\begin{lemma}
\label{lem: L2-Lp bdd of div family}
    Let $\mu$ satisfy Assumption~\ref{Ass: Coefficients}, $\omega_0$ be given by \eqref{eq: def omega} and $\beta\in [0,\frac{\pi}{2}-\omega_0)$. There exists a constant $\delta>0$ depending on $\beta,d$ and $\mu_\bullet,\mu^\bullet$ such that the family $(|z|^\frac{1}{2}e^{-z\cA}\cP\div)_{z\in\S_\beta}$ satisfies $\L^2$-$\L^p$ bounds for all $p\in[2,\infty)$ such that $\frac{1}{2}-\frac{1}{p}<\frac{1}{d}+\delta$.
\end{lemma}

\begin{proof}
    Let $\delta>0$ be from Theorem~\ref{thm: L2-L2*+eps bdd semigroup}. Then, combining Proposition~\ref{prop: L2 semigroup bounds} with Theorem~\ref{thm: L2-L2*+eps bdd semigroup} yields
    \begin{align*}
        \||z|^\frac{1}{2}e^{-z\cA}\cP\div(F)\|_{\L^p} &= \|(e^{-\frac{z}{2}A})(|z|^\frac{1}{2}e^{-\frac{z}{2}\cA}\cP\div)(F)\|_{\L^p} \\&\lesssim |z|^{-\frac{d}{2}(\frac{1}{2}-\frac{1}{p})}\||z|^\frac{1}{2}e^{-\frac{z}{2}\cA}\cP\div(F)\|_{\L^2}\\
        &\lesssim  |z|^{-\frac{d}{2}(\frac{1}{2}-\frac{1}{p})}\|F\|_{\L^2}
    \end{align*}
    for every $F\in\L^2(\IR^d;\IC^{d\times d})$ and  $p\in[2,\infty)$ such that $\frac{1}{2}-\frac{1}{p}<\frac{1}{d}+\delta$.
\end{proof}

\begin{proposition}
Let $\mu$ satisfy Assumption~\ref{Ass: Coefficients}, $\omega_0$ be given by \eqref{eq: def omega} and $\beta\in [0,\frac{\pi}{2}-\omega_0)$. There exists $\varepsilon>0$ depending on $\beta, d$ and $\mu_\bullet,\mu^\bullet$ such that for all $p\in [2,\infty)$ satisfying
\begin{align*}
    \frac{1}{2}-\frac{1}{p}<  \frac{1}{d}+\varepsilon
\end{align*}
there exists a constant $C>0$ such that for all $z \in \S_\beta$ and $F\in \L^2(\IR^d;\IC^{d\times d})\cap \L^{p}(\IR^d;\IC^{d\times d})$ we have
\begin{align*}
    \||z|^\frac{1}{2}e^{-z\cA}\cP\div(F) \|_{\L^{p}}\leq C \|F\|_{\L^{p}}.
\end{align*}
The constant $C$ depends only on $\beta$, $d$, $p$ and $\mu_\bullet,\mu^\bullet$.
\end{proposition}

\begin{proof}
    The proof follows the same line of reasoning as the proof of Theorem~\ref{thm: Lp-extrapolation sg II} from Corollary~\ref{cor: decay, hyper, unif} with Proposition~\ref{prop: Lp off-diag sg div} and Lemma~\ref{lem: L2-Lp bdd of div family}.
\end{proof}

By the adjoint identity from Lemma~\ref{lem: duality}, we deduce uniform $\L^p$-bounds for the operator family $(|z|^\frac{1}{2}\nabla e^{-zA})_{z\in\S_{\frac{\pi}{2}-\omega_0}}$ in the case $p<2$ from the previous proposition.

\begin{corollary}
\label{cor: Lp bound for div fam p<2}
    Let $\mu$ satisfy Assumption~\ref{Ass: Coefficients}, $\omega_0$ be given by \eqref{eq: def omega} and $\beta\in [0,\frac{\pi}{2}-\omega_0)$. There exists $\varepsilon>0$ depending on $\beta, d$ and $\mu_\bullet,\mu^\bullet$ such that for all $p\in (1,2]$ satisfying
    \begin{align*}
        \frac{2d}{d+2}-\varepsilon < p \leq 2
    \end{align*}
    there exists a constant $C>0$ such that for all $z \in \S_\beta$ and $f\in \L_\sigma^2(\IR^d)\cap \L_\sigma^{p}(\IR^d)$ we have
    \begin{align*}
      \||z|^\frac{1}{2}\nabla e^{-zA}f \|_{\L^{p}}\leq C \|f\|_{\L^{p}}.
    \end{align*}
    The constant $C$ depends only on $\beta$, $d$, $p$ and $\mu_\bullet,\mu^\bullet$.
\end{corollary}

\subsection{The case $p>2$}

Here, we will directly work with  $(|z|^\frac{1}{2}\nabla e^{-zA})_{z\in\S_{\frac{\pi}{2}-\omega_0}}$.

\begin{proposition}
\label{prop: L2-L2+eps bdd grad sg}
    Let $\mu$ satisfy Assumption~\ref{Ass: Coefficients}, $\omega_0$ be given by \eqref{eq: def omega} and $\beta\in [0,\frac{\pi}{2}-\omega_0)$. There exists $\delta>0$ depending on $\beta,d$ and $\mu_\bullet,\mu^\bullet$ such that $(|z|^\frac{1}{2}\nabla e^{-zA})_{z\in\S_\beta}$ satisfies $\L^2$-$\L^p$ bounds for all $p\in(2,\infty)$ such that $\frac{1}{2}-\frac{1}{p}<\min\{\frac{1}{d},\delta\}$.
\end{proposition}

\begin{proof}
    Let $\delta>0$ be the same as in Theorem~\ref{thm: extension of 1+A}. By assumption, we see that $\frac{1}{p}<\frac{1}{2}<\frac{1}{p}+\frac{1}{d}$. Hence, Sobolev's theorem and a duality argument imply the continuous embedding $\L^2_\sigma(\IR^d) \hookrightarrow \W_\sigma^{-1,p}(\IR^d)$. Then with Theorem~\ref{thm: extension of 1+A}, we know that
    \begin{align*}
        (1+\cA)^{-1} : (\L_\sigma^2\hookrightarrow\W^{-1,p}_\sigma) \to \W^{1,p}_\sigma 
    \end{align*}
    is a bounded operator. Combining this fact with Proposition~\ref{prop: L2 semigroup bounds} yields for all $f\in\L^2_\sigma(\IR^d)$ the estimate
    \begin{align*}
        \|\nabla e^{-A}f \|_{\L^{p}} &= \|\nabla (1+A)^{-1} (1+A) e^{-A}f \|_{\L^{p}}\\
        &\leq\|(1+A)^{-1} (1+A) e^{-A}f \|_{\W^{1,p}}\\
        &\lesssim \| (1+A) e^{-A}f \|_{\L^{2}}\\
        &\lesssim \|f \|_{\L^{2}}.
    \end{align*}
    Finally, Lemma~\ref{lem: rescaling lemma} yields the claim.
\end{proof}

\begin{proof}[Proof of Theorem~\ref{thm: Lp bound for grad fam}]
    By Corollary~\ref{cor: Lp bound for div fam p<2}, it suffices to show the claim for $p>2$. To this end, we apply Corollary~\ref{cor: decay, hyper, unif} with Propositions~\ref{prop: off-diag semigroup} and~\ref{prop: L2-L2+eps bdd grad sg} to conclude the claim.
\end{proof}

We conclude this section by proving the $\L^q$-$\L^p$ smoothing estimates from Theorem~\ref{thm: smoothing}.

\begin{proof}[Proof of Theorem~\ref{thm: smoothing}]
    We divide the proof into four steps. The first three steps deduce the estimates~\eqref{eq: smoothing sg} for the semigroup. Subsequently, we combine these with the uniform $\L^p$-bounds from Theorem~\ref{thm: Lp bound for grad fam} to obtain the corresponding bounds~\eqref{eq: smoothing grad} for its gradient.\\

    \noindent \textbf{Step 1: \eqref{eq: smoothing sg} for $2\leq q\leq p$.}
    By Theorems~\ref{thm: L2-L2*+eps bdd semigroup} and~\ref{thm: Lp-extrapolation sg II}, we have
    \begin{align*}
        \|e^{-zA}f\|_{\L^p}\lesssim |z|^{-\frac{d}{2}(\frac{1}{2}-\frac{1}{p})}\|f\|_{\L^2} \quad \text{and} \quad \|e^{-zA}f\|_{\L^p}\lesssim \|f\|_{\L^p} 
    \end{align*}
    for all $f\in \L^2_\sigma(\IR^d)\cap \L^p_\sigma(\IR^d)$. Interpolating both estimates yields
    \begin{align*}
        \|e^{-zA}f\|_{\L^p}\lesssim |z|^{-\frac{d}{2}(\frac{1}{q}-\frac{1}{p})}\|f\|_{\L^q} 
    \end{align*}
    for every $2\leq q\leq p$ and $f\in \L^2_\sigma(\IR^d)\cap \L^q_\sigma(\IR^d)$.\\
    
    \noindent \textbf{Step 2: \eqref{eq: smoothing sg} for $q\leq p\leq 2$.}
     Let $2\leq p'\leq q'$ be the Hölder conjugate exponents of $p$ and $q$. Applying Step~1 to the adjoint operator family $(e^{-\bar z A^*})_{z\in\S_\beta}$ yields
    \begin{align*}
        \|e^{-\bar z A^*}f\|_{\L^{q'}}\lesssim |z|^{-\frac{d}{2}(\frac{1}{p'}-\frac{1}{q'})}\|f\|_{\L^{p'}} 
    \end{align*}
    for every  $f\in \L^2_\sigma(\IR^d)\cap \L^{p'}_\sigma(\IR^d)$. Thus, using a duality argument, we conclude
    \begin{align*}
        \|e^{-zA}f\|_{\L^{p}}\lesssim |z|^{-\frac{d}{2}(\frac{1}{q}-\frac{1}{p})}\|f\|_{\L^{q}} 
    \end{align*}
    for all $q\leq p\leq 2$ and $f\in \L^2_\sigma(\IR^d)\cap \L^{q}_\sigma(\IR^d)$.\\
    
    \noindent \textbf{Step 3: \eqref{eq: smoothing sg} for $q\leq 2\leq p$.}
    Using the semigroup property, as well as the previous two steps, we obtain
    \begin{align*}
        \|e^{-zA}f\|_{\L^p} = \|e^{-\frac{z}{2}A}e^{-\frac{z}{2}A}f\|_{\L^p}  \lesssim |z|^{-\frac{d}{2}(\frac{1}{2}-\frac{1}{p})}\|e^{-\frac{z}{2}A}f\|_{\L^2} \lesssim |z|^{-\frac{d}{2}(\frac{1}{q}-\frac{1}{p})}\|f\|_{\L^q},
    \end{align*}
    which proves \eqref{eq: smoothing sg}.\\

    \noindent \textbf{Step 4: \eqref{eq: smoothing grad}.}
    Using the uniform $\L^p$-boundedness from Theorem~\ref{thm: Lp bound for grad fam}, the estimate \eqref{eq: smoothing sg} and the semigroup property, we obtain
    \begin{align*}
        \||z|^\frac{1}{2}\nabla e^{-z A}f\|_{\L^p} = \||z|^\frac{1}{2}\nabla e^{-\frac{z}{2}A}e^{-\frac{z}{2}A}f\|_{\L^p}\lesssim \|e^{-\frac{z}{2}A}f\|_{\L^p} \lesssim |z|^{-\frac{d}{2}(\frac{1}{q}-\frac{1}{p})}\|f\|_{\L^q}
    \end{align*}
    for every $f\in \L^2_\sigma(\IR^d)\cap \L^{q}_\sigma(\IR^d)$.
\end{proof}

\section{$\H^\infty$-calculus}
\label{sec: Hinfty calc}

In this section, we want to show that not only the semigroup but also the whole functional calculus of $A$ can be extended from $\L^2_\sigma(\IR^d)$ to $\L^p_\sigma(\IR^d)$, see Theorem~\ref{thm: bdd calc on Lp}. A first approach to this was presented by Tolksdorf in~\cite{Tolksdorf-Caccioppoli} but only in the range $p\in (1,\infty)$ satisfying
\begin{align*}
    \Big|\frac{1}{2}-\frac{1}{p}\Big|< \frac{1}{d}.
\end{align*}
His arguments are based on extrapolation of the maximal regularity and a transference principle due to Kunstmann--Weis~\cite{Kunstmann-Weis}. Here, we reprove and extend this result by means of non-local decay estimates. Our approach follows the one for elliptic systems developed in the monograph of Auscher \cite{Auscher-Lp}. In particular, we extend the range for which $A$ possesses a bounded $\H^\infty$-calculus on $\L^p_\sigma(\IR^d)$ to $p\in (1,\infty)$ satisfying
\begin{align*}
    \Big|\frac{1}{2}-\frac{1}{p}\Big|< \frac{1}{d} +\varepsilon
\end{align*}
for some (possibly small) $\varepsilon>0$. This is consistent with the theory of elliptic systems, see \cite[Chap.~5]{Auscher-Lp}. Here, we use the following extrapolation theorem, which is an adaptation of \cite[Thm.~1.2]{Auscher-Lp} to our fluid situation.

\begin{theorem}
\label{thm: extrapolation p>2}
    Let $p_0\in (2,\infty]$. Let $T:\L^2_\sigma(\IR^d)\to \L^2_\sigma(\IR^d)$ be a sublinear operator and let $(S_r)_{r>0}$ be a family of linear operators on $\L^2_\sigma(\IR^d)$. Let $\cM$ denote the Hardy--Littlewood maximal operator and assume there exists $C>0$ such that
    \begin{align}
    \label{eq: BK 1}
        \bigg(\fint\limits_{B} |T(1-S_{r})f|^2\,\d x\bigg)^\frac{1}{2}\leq C\Big(\mathcal{M}(|f|^2)(x)\Big)^\frac{1}{2}
    \end{align}
    and
    \begin{align}
    \label{eq: BK 2}
        \bigg(\fint\limits_{B} |TS_{r}f|^{p_0}\,\d x\bigg)^\frac{1}{p_0}\leq C\Big(\mathcal{M}(|Tf|^2)(x)\Big)^\frac{1}{2}
    \end{align}
    for all $f\in\L^2_\sigma(\IR^d)$, all open balls $B=B(x_0,r)\subset \IR^d$ and all $x\in B$ with the usual replacement of the integrals by an $\esssup$ if $p_0=\infty$. Suppose further that
    \begin{align}
    \label{eq: BK 3}
        T:\L^2_\sigma(\IR^d)\cap \L^{p}_\sigma(\IR^d)\to \L^2_\sigma(\IR^d)\cap \L^{p}_\sigma(\IR^d)
    \end{align}
    is well-defined for some $p\in(2,p_0)$. Then there exists $C'>0$ depending only on $d, p, p_0$ and $C$ such that for all $f\in \L^2_\sigma(\IR^d)\cap \L^{p}_\sigma(\IR^d)$ we have
    \begin{align*}
        \|Tf\|_{\L^p} \leq C'\|f\|_{\L^p}.
    \end{align*}
\end{theorem}

The proof of \cite[Thm.~1.2]{Auscher-Lp} applies verbatim. Indeed, it is based only on $\L^p$ and good-$\lambda$ estimates for the Hardy--Littlewood maximal operator, and a splitting of the function $f$ with respect to the family $(S_r)_{r>0}$ that does not destroy any divergence-free constraints. Hence, we will omit the proof and leave the details to the interested reader.

\begin{example}
\label{ex: Ar family}
    For $m\in \IN$, consider the family $(S_r)_{r>0}$ defined by
    \begin{align*}
        S_r \coloneqq 1-(1-e^{-r^2A})^m = -\sum\limits_{k=1}^m\binom{m}{k}(-1)^ke^{-kr^2A}.
    \end{align*}
    Then, $(S_r)_{r>0}$ is a family of bounded linear operators on $\L^2_\sigma(\IR^d)$ and there exists $\varepsilon>0$ such that \eqref{eq: BK 2} is satisfied in the case $T=\Id$ and for all $p_0\in (2,\infty)$ with $\frac{1}{2}-\frac{1}{p_0}<\frac{1}{d}+\varepsilon$. Indeed, we argue as follows: assume $d\ge 3$. Let $\delta>0$ be as in Theorem~\ref{thm: L2-L2*+eps bdd semigroup} and fix $0<\delta'<\delta$. Then for $p\in[2,\infty)$ with $\frac{1}{2}-\frac{1}{p}=\frac{1}{d}+\delta'$ the semigroup satisfies $\L^2$-$\L^p$ bounds. Furthermore, by Proposition~\ref{prop: Lp off-diag sg} we know that, for every $\beta\in[0,\frac{\pi}{2}-\omega_0)$, $(e^{-zA})_{z\in\S_\beta}$ satisfies $\L^2$-$\L^q$ decay estimates of order $\nu>d$ for $q\in[2,\infty)$ with $\frac{1}{2}-\frac{1}{q} = \frac{1}{d}$. Interpolating both properties with Proposition~\ref{prop: int principles} yields $\L^2$-$\L^{p_\theta}$ decay estimates of order $\nu(1-\theta)>0$ for $\theta\in(0,1)$ and
    \begin{align*}
        \frac{1}{p_\theta} = \frac{1-\theta}{q} + \frac{\theta}{p} = \frac{1}{2}-\frac{1}{d}-\theta\delta'
    \end{align*}
    which is equivalent to
    \begin{align*}
        \frac{1}{2}-\frac{1}{p_\theta} = \frac{1}{d} +\theta\delta'.
    \end{align*}
    Choose $\theta$ such that $\nu(1-\theta)>d$ and fix $\varepsilon = \theta\delta'$. Then, for all $g\in \L^2_\sigma(\IR^d)$ and balls $B=B(x_0,r)\subset \IR^d$ we have
    \begin{align*}
        \bigg(\fint\limits_{B} |S_rg|^{p_\theta}\,\d x\bigg)^\frac{1}{p_\theta}
        &\lesssim r^{-\frac{d}{p_\theta}}\sum\limits_{k=1}^m\binom{m}{k} \|e^{-kr^2A}g\|_{\L^{p_\theta}(B)}\\
        &\leq r^{-\frac{d}{p_\theta}}\sum\limits_{k=1}^m\binom{m}{k} C(k)r^{-\frac{d}{2}(1-\frac{2}{p_\theta})}\bigg(\sum\limits_{n=0}^\infty 2^{-\nu(1-\theta)n} \|g\|^2_{\L^{2}(C_n(B))}\bigg)^\frac{1}{2}\\
        &\lesssim \sum\limits_{k=1}^m\binom{m}{k} C(k)\bigg(\sum\limits_{n=0}^\infty 2^{(d-\nu(1-\theta))n} \mathcal{M}(|g|^2)(x)\bigg)^\frac{1}{2}\\
        &\lesssim \Big(\mathcal{M}(|g|^2)(x)\Big)^\frac{1}{2}.
    \end{align*}
    If $p_0\in [2,\infty)$ satisfies $p_0\leq p_\theta$, we use Jensen's inequality to obtain
    \begin{align*}
         \bigg(\fint\limits_{B} |S_rg|^{p_0}\,\d x\bigg)^\frac{1}{p_0}\leq \bigg(\fint\limits_{B} |S_rg|^{p_\theta}\,\d x\bigg)^\frac{1}{p_\theta}
        &\lesssim  \Big(\mathcal{M}(|g|^2)(x)\Big)^\frac{1}{2}.
    \end{align*}
    In the case $d= 2$, Proposition~\ref{prop: Lp off-diag sg} gives $\L^2$-$\L^p$ decay estimates of order $\nu>d$ for every $p\in[2,\infty)$, so we may set $p_\theta=p\in[2,\infty)$ and run the same calculations as above.
\end{example}

For the proof of Theorem~\ref{thm: bdd calc on Lp}, we need the following representation of the functional calculus via the semigroup. A proof for elliptic operators can be found in \cite[Lem.~5.1]{Egert-Lp}. Since the argumentation literally carries over to our situation, we will omit any details.

\begin{lemma}
\label{lem: rep FC via sg}
    Let $\mu$ satisfy Assumption~\ref{Ass: Coefficients} and $\omega_0$ be given by \eqref{eq: def omega}. For $\omega_0 < \theta<\vartheta<\varrho<\frac{\pi}{2}$ and $\varphi\in \H^\infty_0(\S_\varrho)$ we can represent $\varphi(A)$ via
    \begin{align*}
        \varphi(A) = \int\limits_{\Gamma_+}e^{-zA}\eta_+(z)\, \d z - \int\limits_{\Gamma_-}e^{-zA}\eta_-(z)\, \d z,
    \end{align*}
    where $\Gamma_{\pm} = \IR_+ e^{\pm i(\frac{\pi}{2}-\theta)}$ and
    \begin{align*}
        \eta_\pm(z) = \frac{1}{2\pi i}\int\limits_{\gamma_\pm} e^{\zeta z}\varphi(\zeta)\, \d \zeta
    \end{align*}
    with $\gamma_{\pm} = \IR_+ e^{\pm i\vartheta}$.
\end{lemma}

Moreover, we employ the following reduction lemma, which is standard and can be found for elliptic operators, for example, in the proof of \cite[Thm.~5.1]{Egert-Lp} and \cite[Thm.~17.4]{Bechtel}. The underlying argumentation carries over to the Stokes operator with minor modifications, and we thus omit the details.

\begin{lemma}[Reduction lemma]
\label{lem: reduction lemma}
    Let $\mu$ satisfy Assumption~\ref{Ass: Coefficients}, $\omega_0$ be given by \eqref{eq: def omega} and $\omega_0 < \varrho<\pi$.
    Assume that there exists $p\in[2,\infty)$ such that
    \begin{align*}
        \|\varphi(A)f\|_{\L^{p}}\lesssim \|f\|_{\L^{p}}
    \end{align*}
    holds for all $f\in \L^2_\sigma(\IR^d)\cap\L^{p}_\sigma(\IR^d)$ and $\varphi\in \H^\infty_0(\S_\varrho)$ with $\|\varphi\|_\infty=1$. Then, we have
    \begin{align*}
        \|\varphi(A)f\|_{\L^q}\lesssim \|\varphi\|_{\infty}\|f\|_{\L^q}
    \end{align*}
    for all $f\in \L^2_\sigma(\IR^d)\cap\L^{q}_\sigma(\IR^d)$, $q\in [p', p]$ and $\varphi\in \H^\infty(\S_\varrho)$.
\end{lemma}

With these tools at hand, we are now ready to prove the main result of this section.

\begin{proof}[Proof of Theorem~\ref{thm: bdd calc on Lp}]
    Without loss of generality, we may assume $\omega_0 < \varrho < \frac{\pi}{2}$. Let $\varepsilon > 0$ be the smaller of the two constants provided by Example~\ref{ex: Ar family} and Theorem~\ref{thm: Lp-extrapolation sg II}, and let $p_0 \in (2,\infty)$ be arbitrary with $\frac{1}{2} - \frac{1}{p_0} < \frac{1}{d}+ \varepsilon$. We claim that it suffices to show
    \begin{align}
        \label{eq:claim-p}
        \|\varphi(A)f\|_{\L^p} \lesssim \|f\|_{\L^p}
    \end{align}
    for all $f \in \L^2_\sigma(\IR^d) \cap \L^p_\sigma(\IR^d)$ and all $\varphi \in \H^\infty_0(\S_\varrho)$ with $\|\varphi\|_\infty = 1$, with an implicit constant independent of $\varphi$ and $f$. If this is true, then Lemma~\ref{lem: reduction lemma} applied with $p$ yields
    \begin{align*}
        \|\varphi(A)f\|_{\L^q} \lesssim \|\varphi\|_\infty \|f\|_{\L^q}
    \end{align*}
    for all $q \in [p', p]$, all $f \in \L^2_\sigma(\IR^d) \cap \L^q_\sigma(\IR^d)$ and all $\varphi \in \H^\infty(\S_\varrho)$. Since $p \in (2,p_0)$ was arbitrary, the assertion follows. It therefore remains to prove \eqref{eq:claim-p} for a fixed $p \in (2,p_0)$.  For this purpose, we apply Theorem~\ref{thm: extrapolation p>2} to  $T=\varphi(A)$ with the smoothing family $(S_r)_{r>0}$ defined in Example~\ref{ex: Ar family} with $m>\frac{d}{2}+1$. We divide the rest of the proof into three steps, verifying each assumption of Theorem~\ref{thm: extrapolation p>2}.\\

    \noindent \textbf{Step 1: Verification of \eqref{eq: BK 3}.} 
    Fix $f\in\L^2_\sigma(\IR^d)\cap\L^p_\sigma(\IR^d)$. By Lemma~\ref{lem: rep FC via sg}, we have for $\omega_0 < \theta<\vartheta<\varrho<\frac{\pi}{2}$ the representation
    \begin{align*}
        \varphi(A)f =  \int\limits_{\Gamma_+}(e^{-zA}f)\eta_+(z)\, \d z - \int\limits_{\Gamma_-}(e^{-zA}f)\eta_-(z)\, \d z,
    \end{align*}
    where $\Gamma_{\pm}$ and $\gamma_\pm$ are defined as in Lemma~\ref{lem: rep FC via sg}. Now, the $\L^p$-boundedness of the semigroup (Theorem~\ref{thm: Lp-extrapolation sg II}) implies
    \begin{align*}
        \|\varphi(A)f\|_{\L^p}\lesssim \|f\|_{\L^p}\bigg(\int\limits_{\Gamma_+}\int\limits_{\gamma_+} |e^{\zeta z}||\varphi(\zeta)|\, |\d \zeta|\, |\d z|+\int\limits_{\Gamma_-}\int\limits_{\gamma_-} |e^{\zeta z}||\varphi(\zeta)|\, |\d \zeta|\, |\d z|\bigg).
    \end{align*}
    Since $|\arg (\zeta z)| = \frac{\pi}{2}-\theta + \vartheta>\frac{\pi}{2}$, there exists $c(\theta,\vartheta)>0$ such that 
    \begin{align*}
        \|\varphi(A)f\|_{\L^p}&\lesssim \|f\|_{\L^p}\bigg(\int\limits_{\Gamma_+}\int\limits_{\gamma_+} e^{-c|\zeta| |z|}|\varphi(\zeta)|\, |\d \zeta|\, |\d z|+\int\limits_{\Gamma_-}\int\limits_{\gamma_-} e^{-c|\zeta| |z|}|\varphi(\zeta)|\, |\d \zeta|\, |\d z|\bigg)\\
        &\lesssim \|f\|_{\L^p}\bigg(\int\limits_{\gamma_+} \frac{|\varphi(\zeta)|}{|\zeta|}\, |\d \zeta| + \int\limits_{\gamma_-} \frac{|\varphi(\zeta)|}{|\zeta|}\, |\d \zeta|\bigg)
    \end{align*}
    where the right-hand side is finite by assumption on $\varphi$. Hence, the map
    \begin{align*}
        \varphi(A): \L^2_\sigma(\IR^d)\cap \L^{p}_\sigma(\IR^d)\to \L^2_\sigma(\IR^d)\cap \L^{p}_\sigma(\IR^d)
    \end{align*}
    is a priori well-defined.\\
    
    \noindent \textbf{Step 2: Verification of \eqref{eq: BK 2}.} 
    Fix $f\in \L^2_\sigma(\IR^d)$ and a ball $B=B(x_0,r)\subset \IR^d$. Observe by commutativity that $\varphi(A)S_{r} =  S_{r}\varphi(A)$ on $\L^2_\sigma(\IR^d)$. Hence, Example~\ref{ex: Ar family} with $g=\varphi(A)f$ yields
    \begin{align*}
         \bigg(\fint\limits_{B} |\varphi(A)S_{r}f|^{p_0}\,\d x\bigg)^\frac{1}{p_0} &= \bigg(\fint\limits_{B} |S_{r}g|^{p_0}\,\d x\bigg)^\frac{1}{p_0}\\
         &\leq C\Big(\mathcal{M}(|g|^2)(x)\Big)^\frac{1}{2}\\
         &= C\Big(\mathcal{M}(|\varphi(A)f|^2)(x)\Big)^\frac{1}{2},
    \end{align*}
    which proves \eqref{eq: BK 2}.\\

   \noindent \textbf{Step 3: Verification of \eqref{eq: BK 1}.}
    Define the function $\psi(z) \coloneqq \varphi(z)(1-e^{-r^2z})^m$ and notice that
    \begin{align*}
        \psi(A) =\varphi(A)(1-e^{-r^2A})^m= \varphi(A)(1-S_r).
    \end{align*}
    Fix a function $f\in \L^2_\sigma(\IR^d)$ and a ball $B=B(x_0,r)$. Then, we estimate
    \begin{align}
    \label{eq: bk1}
        \bigg(\fint\limits_B |\psi(A)f(y)|^2\,\d y\bigg)^\frac{1}{2} \leq \frac{1}{|B|^\frac{1}{2}}\|\psi(A)f_{4r}\|_{\L^2(B)} + \frac{1}{|B|^\frac{1}{2}}\|\psi(A)(f-f_{4r})\|_{\L^2(B)},
    \end{align}
    where $f_{4r} = f_{2(2r)}\in\L^2_\sigma(\IR^d)$ is defined in Lemma~\ref{lem: Bogovskii lemma} applied to the ball $2B$. Using Proposition~\ref{prop: bdd Hinf on L2} and Lemma~\ref{lem: Bogovskii lemma}, we can bound the on-diagonal part of \eqref{eq: bk1} by
    \begin{align}
    \label{eq: bk1.2}
        \frac{1}{|B|^\frac{1}{2}}\|\psi(A)f_{4r}\|_{\L^2(B)}  \lesssim \frac{1}{|B|^\frac{1}{2}}\|f_{4r}\|_{\L^2}\lesssim \frac{1}{|B|^\frac{1}{2}}\|f\|_{\L^2(4B)}\leq \Big( \cM(|f|^2)(x)\Big)^\frac{1}{2} 
    \end{align}
    for all $x\in B$. To estimate the off-diagonal term in \eqref{eq: bk1}, we use the semigroup representation of the functional calculus 
    \begin{align}
    \label{eq: bk0}
        \psi(A) = \int\limits_{\Gamma_+}e^{-zA}\eta_+(z)\, \d z - \int\limits_{\Gamma_-}e^{-zA}\eta_-(z)\, \d z
    \end{align}
    from Lemma~\ref{lem: rep FC via sg}. Notice that the normalization of $\varphi$ and the complex mean value theorem imply $|\psi(\zeta)|\lesssim \min\{1,r^{2m}|\zeta|^m\}$ for all $\zeta \in\gamma_\pm$. Hence, the functions $\eta_\pm$ admit the following pointwise bounds
    \begin{align}
    \label{eq: point bdd eta}
        |\eta_\pm(z)|\lesssim\frac{1}{|z|}\min\Big\{1,\frac{r^{2m}}{|z|^{m}}\Big\}
    \end{align}
    for all $z\in\Gamma_\pm$. Combining \eqref{eq: bk0} with Minkowski's integral inequality and \eqref{eq: point bdd eta} yields
    \begin{align}
    \label{eq: bk2}
         \frac{1}{|B|^\frac{1}{2}}\|\psi(A)(f-f_{4r})\|_{\L^2(B)}&\lesssim  \frac{1}{|B|^\frac{1}{2}} ( I_+ + I_-),
    \end{align}
    where we set
    \begin{align*}
        I_\pm \coloneqq \int\limits_{\Gamma_\pm} \frac{1}{|z|}\min\Big\{1,\frac{r^{2m}}{|z|^m}\Big\} \|e^{-zA}(f-f_{4r})\|_{\L^2(B)}\,|\d z| .
    \end{align*}
    Applying Proposition~\ref{prop: imp off-diag semigroup} and Lemma~\ref{lem: Bogovskii lemma} with $d<\nu<d+2$ to $I_\pm$ gives
    \begin{align}
    \label{eq: bounds Ipm}
        I_\pm &\lesssim \int\limits_{\Gamma_\pm} \frac{1}{|z|}\min\Big\{1,\frac{r^{2m}}{|z|^m}\Big\}\frac{|z|^\frac{1}{2}}{r} \max\Big\{1, \Big(\frac{r^2}{|z|}\Big)^{-\frac{d+1}{2}}\Big\}\bigg(\sum\limits_{n=0}^\infty 2^{-\nu n}\|f\|_{\L^2(C_n(B))}^2\bigg)^\frac{1}{2}\,|\d z|.
    \end{align}
    Since $\nu>d$ we have 
    \begin{align*}
        \bigg(\sum\limits_{n=0}^\infty 2^{-\nu n}\|f\|_{\L^2(C_n(B))}^2\bigg)^\frac{1}{2} \leq \bigg(\sum\limits_{n=0}^\infty 2^{-\nu n}|2^nB|\mathcal{M}(|f|^2)(x)\bigg)^\frac{1}{2} \lesssim |B|^\frac{1}{2}\big(\mathcal{M}(|f|^2)(x)\big)^\frac{1}{2}
    \end{align*}
    for all $x\in B$. Plugging these estimates together with \eqref{eq: bounds Ipm} into \eqref{eq: bk2} implies
    \begin{align*}
        \frac{1}{|B|^\frac{1}{2}}\|\psi(A)(f-f_{4r})\|_{\L^2(B)} \lesssim \big(\mathcal{M}(|f|^2)(x)\big)^\frac{1}{2} \int\limits_{0}^\infty \frac{1}{s^\frac{1}{2}r}\min\Big\{1,\frac{r^{2m}}{s^m}\Big\} \max\Big\{1, \Big(\frac{r^2}{s}\Big)^{-\frac{d+1}{2}}\Big\}\,\d s
    \end{align*}
    for all $x\in B$. Finally, since $m>\frac{d}{2}+1$ we have
    \begin{align*}
        \int\limits_{0}^\infty \frac{1}{s^\frac{1}{2}r}\min\Big\{1,\frac{r^{2m}}{s^m}\Big\} \max\Big\{1, \Big(\frac{r^2}{s}\Big)^{-\frac{d+1}{2}}\Big\}\,\d s
        =\int\limits_{0}^{r^2}\frac{\,\d s}{s^\frac{1}{2}r} + \int\limits_{r^2}^\infty \frac{1}{s^\frac{1}{2}r} \Big(\frac{r^2}{s}\Big)^{m-\frac{d+1}{2}}\,\d s
        \lesssim 1,
    \end{align*}
    which yields
    \begin{align*}
        \frac{1}{|B|^\frac{1}{2}}\|\psi(A)(f-f_{4r})\|_{\L^2(B)} \lesssim \big(\mathcal{M}(|f|^2)(x)\big)^\frac{1}{2}
    \end{align*}
    for all $x\in B$. Plugging these estimates together with \eqref{eq: bk1.2} into \eqref{eq: bk1} yields
    \begin{align*}
         \bigg(\fint\limits_B |\psi(A)f(y)|^2\,\d y\bigg)^\frac{1}{2} \lesssim \big(\mathcal{M}(|f|^2)(x)\big)^\frac{1}{2}
    \end{align*}
    for all $x\in B$, which proves \eqref{eq: BK 1}.
\end{proof}


\begin{bibdiv}
\begin{biblist}

\bibitem{Auscher-Kato}
P.~Auscher, S.~Hofmann, M.~Lacey, A.~McIntosh, and Ph.~Tchamitchian.
\newblock The solution of the Kato square root problem for second order elliptic operators on $\IR^n$.
\newblock {\em Ann.\@ Math.~(2)}~\textbf{156} (2002), 633--654.


\bibitem{Auscher-Lp}
P.~Auscher.
\newblock On necessary and sufficient conditions for $\L^p$-estimates of {Riesz} transforms associated to elliptic operators on $\IR^n$ and related estimates.
\newblock \textit{Mem.\@ Am.\@ Math.\@ Soc.}~\textbf{871} (2007).

\bibitem{Auscher-Egert}
P.~Auscher and M.~Egert.
\newblock {\em Boundary value problems and {Hardy} spaces for elliptic systems with block structure.}
\newblock Cham: Birkh{\"a}user, 2023.

\bibitem{Baadi_Egert_Kosmala}
K.~Baadi, M.~Egert and B.~Kosmala.
\newblock {$\mathrm{L}^p$} bounds for parabolic {Riesz} transforms with rough coefficients: {The} case {$1<p \leq 2$}.
\newblock Preprint, \url{https://arxiv.org/abs/2607.05181v2} (2026).

\bibitem{Bechtel}
S.~Bechtel.
\newblock {\em Square roots of elliptic systems in locally uniform domains.}
\newblock Oper.\@ Theory: Adv.\@ Appl.\@~\textbf{303} (2024).

\bibitem{periodic}
A.~Bensoussan, J.~Lions and G.~Papanicolaou.
\newblock {\em Asymptotic analysis for periodic structures}.
\newblock Providence, RI: AMS, 2011.

\bibitem{Berg_Lofstrom}
J.~Bergh and J.~Löfström.
\newblock {\em Interpolation spaces. An introduction.}
\newblock Berlin–Heidelberg–New York: Springer, Grundlehren 223, 1976.

\bibitem{Diening}
L.~Diening and M.~Růžička.
\newblock Strong Solutions for Generalized Newtonian Fluids.
\newblock {\em J.~Math.~Fluid Mech.}~\textbf{7} (2005), 413--450.

\bibitem{Egert-Dissertation}
M.~Egert.
\newblock \textit{On {Kato}'s conjecture and mixed boundary conditions.}
\newblock G{\"o}ttingen: Sierke; Darmstadt: TU Darmstadt, Fachbereich Mathematik (Diss.) (2015).


\bibitem{Egert-Lp}
M.~Egert.
\newblock $\L^p$-estimates for the square root of elliptic systems with mixed boundary conditions.
\newblock {\em J.\@ Differ.\@ Equations}~\textbf{265} (2018), 1279--1323.


\bibitem{Engel_Nagel}
K.~Engel and R.~Nagel.
\newblock {\em One-parameter semigroups for linear evolution equations.}
\newblock Berlin: Springer, 2000.


\bibitem{Galdi}
G.~Galdi.
\newblock {\em An introduction to the mathematical theory of the Navier--Stokes equations. Steady-state problems.}
\newblock New York: Springer, 2011.

\bibitem{Giga}
Y.~Giga.
\newblock Solutions for semilinear parabolic equations in {{\(L^ p\)}} and regularity of weak solutions of the {Navier}-{Stokes} system.
\newblock {\em J. Differ. Equations}~\textbf{61} (1986), 186--212.

\bibitem{Grafakos}
L.~Grafakos.
\newblock {\em Classical {Fourier} analysis, 3rd ed.}
\newblock New York: Springer, 2014.

\bibitem{Gu_Shen}
S.~Gu and Z.~Shen.
\newblock Homogenization of Stokes systems and uniform regularity estimates.
\newblock {\em SIAM J.\@ Math.\@ Anal.}~\textbf{47} (2015), 4025--4057.


\bibitem{Haardt_Tolksdorf}
L.~Haardt and P.~Tolksdorf.
\newblock On Kato's square root property for the generalized Stokes operator.
\newblock {\em J.\@ Funct.\@ Anal.}~\textbf{290} (2026).


\bibitem{Haase}
M.~Haase.
\newblock {\em The functional calculus for sectorial operators}.
\newblock Basel: Birkh{\"a}user, 2006.


\bibitem{HMM}
S.~Hofmann, S.~Mayboroda and A.~McIntosh.
\newblock Second order elliptic operators with complex bounded measurable coefficients in $\L^p$, {Sobolev} and {Hardy} spaces.
\newblock {\em Ann. Sci. {\'E}c. Norm. Sup{\'e}r. (4)}~\textbf{44} (2011), 723--800.

\bibitem{Kato} 
T.~Kato. 
\newblock {\em Perturbation theory for linear operators.}
\newblock Berlin: Springer, 1995.

\bibitem{Kato-Lp}
T.~Kato.
\newblock Strong $\L^p$-solutions of the {Navier}-{Stokes} equation in $\IR^m$, with applications to weak solutions.
\newblock {\em Math. Z.}~\textbf{187} (1984), 471--480.

\bibitem{Kunstmann-Weis}
P.~Kunstmann and L.~Weis.
\newblock New criteria for the $\H^\infty$-calculus and the Stokes operator on bounded Lipschitz domains.
\newblock {\em J.~Evol.~Equ.}~\textbf{17} (2017), 387--409.

\bibitem{Lunardi}
A.~Lunardi.
\newblock {\em Interpolation theory, 3rd ed.}
\newblock Pisa: Edizioni della Normale, 2018.


\bibitem{Pruss_Simonett}
J.~Pr\"uss and G.~Simonett.
\newblock \textit{Moving interfaces and quasilinear parabolic evolution equations}.
\newblock Cham: Birkh\"auser/Springer, 2016.

\bibitem{Sneiberg}
I.~\u{S}ne\u{\i}berg.
\newblock Spectral properties of linear operators in interpolation families of {Banach} spaces.
\newblock {\em Mat. Issled.}~\textbf{9} (1974), 214--229.

\bibitem{Sohr}
H.~Sohr.
\newblock {\em The Navier--Stokes equations. An elementary functional analytic approach.}
\newblock Basel: Birkh{\"a}user, 2001.

\bibitem{Stein_singular}
E.~Stein.
\newblock {\em Singular integrals and differentiability properties of functions.}
\newblock Princeton, NJ: Princeton University Press, 1970.


\bibitem{Tolksdorf-Caccioppoli}
P.~Tolksdorf.
\newblock A non-local approach to the generalized Stokes operator with bounded measurable coefficients.
\newblock {\em Calc.\@ Var.\@ Partial Differ.\@ Equ.}~\textbf{65} (2026).

\bibitem{Tolksdorf}
P.~Tolksdorf.
\newblock $\L^p$-extrapolation of non-local operators: maximal regularity of elliptic integrodifferential operators with measurable coefficients.
\newblock {\em J. Evol. Equ.}~\textbf{21} (2021), 3129--3151.

\bibitem{Tolksdorf-off-diagonal}
P.~Tolksdorf.
\newblock On off-diagonal decay properties of the generalized {Stokes} semigroup with bounded measurable coefficients.
\newblock {\em J.\@ Elliptic Parabol.\@ Equ.}~\textbf{7} (2021), 323--340.

\bibitem{Triebel}
H.~Triebel.
\newblock {\em Interpolation Theory, Function Spaces, Differential Operators.}
\newblock Berlin: {Deutscher} {Verlag} der {Wissenschaften}, 1978.
    


\end{biblist}
\end{bibdiv}
\end{document}